\documentclass[aos]{imsart}
\RequirePackage{amsthm,amsmath,amsfonts,amssymb}
\RequirePackage[authoryear]{natbib}%% uncomment this for author-year citations
\RequirePackage[colorlinks,citecolor=blue,urlcolor=blue]{hyperref}%% uncomment this for coloring bibliography citations and linked URLs

\startlocaldefs
\newcommand{\E}{\textit{e}}

\usepackage[utf8]{inputenc}
\usepackage{amsmath}
\usepackage{amssymb}
\usepackage{amsthm}
\usepackage{thmtools}
\renewcommand{\thmcontinues}[1]{continued}

\usepackage{multirow}
\usepackage{enumerate}
\usepackage{bbm}
\usepackage{bm}

\usepackage{natbib}
\usepackage{todonotes}
\usepackage{mathtools}
\usepackage{appendix}
\usepackage{url}
\usepackage{nameref}
\usepackage{framed}
\usepackage{xcolor}

\newcommand*{\reals}{\mathbb{R}}
\newcommand*{\nats}{\mathbb{N}}

\newcommand*{\inner}[2]{\langle #1, #2\rangle}

\newcommand*{\indicator}[1]{\mathbbm{1}\left\{#1\right\}}
\newcommand*{\bracks}[1]{\left\{#1\right\}}
\newcommand*{\sqbrack}[1]{\left[#1\right]}
\newcommand*{\paren}[1]{\left(#1\right)}
\newcommand*{\abs}[1]{\left|#1\right|}
\newcommand*{\norm}[1]{\left\lVert#1\right\rVert}

\newcommand*{\rme}{\mathrm{e}}
\newcommand*{\rmd}{\mathrm{d}}
\newcommand*{\KL}{\mathrm{KL}}

\newcommand*{\etr}{\mathrm{etr}}

\newcommand*{\GL}{\mathrm{GL}}

\newcommand*{\LTP}{\mathrm{LT}^+}

\newtheorem{theorem}{Theorem}%[section]
\newtheorem{corollary}[theorem]{Corollary}
\newtheorem{proposition}[theorem]{Proposition}

\theoremstyle{remark}
\newtheorem{remark}{Remark}

\theoremstyle{definition}

\theoremstyle{remark} 

\theoremstyle{remark}
\newtheorem{example}{Example}

\theoremstyle{remark}
\newtheorem{lemma}{Lemma}

\theoremstyle{definition}

\theoremstyle{remark}
\newtheorem{assumption}{Assumption}

\defcitealias{grunwald_safe_2023}{GHK}

\endlocaldefs

\begin{document}

\begin{frontmatter}
  \title{E-statistics, group invariance and anytime-valid testing}
%\title{A sample article ti with some additional note\thanksref{T1}}

  \runtitle{E-statistics, group invariance and anytime-valid testing}
%\thankstext{T1}{A sample of additional note to the title.}

\begin{aug}
%%%%%%%%%%%%%%%%%%%%%%%%%%%%%%%%%%%%%%%%%%%%%%%
%% Only one address is permitted per author. %%
%% Only division, organization and e-mail is %%
%% included in the address.                  %%
%% Additional information can be included in %%
% Please list one affiliation per author (department and university only). If an
% author has more than one affiliation, the others should be listed in the
% Acknowledgments section.
\author[A]{\fnms{Muriel Felipe}~\snm{Pérez-Ortiz}\ead[label=e1]{m.f.perez.ortiz@tue.nl}},
\author[B]{\fnms{Tyron}~\snm{Lardy}\ead[label=e2]{t.d.lardy@math.leidenuniv.nl}}
\author[C]{\fnms{Rianne}~\snm{de Heide}\ead[label=e3]{r.de.heide@vu.nl}}
\and
\author[D]{\fnms{Peter D.}~\snm{Grünwald}\ead[label=e4]{pdg@cwi.nl}}
%%%%%%%%%%%%%%%%%%%%%%%%%%%%%%%%%%%%%%%%%%%%%%
%% Addresses                                %%
%%%%%%%%%%%%%%%%%%%%%%%%%%%%%%%%%%%%%%%%%%%%%%
\address[A]{Eindhoven University of Technology, Eindhoven, The Netherlands\printead[presep={,\
  }]{e1}}
\address[B]{Leiden University, Leiden, The Netherlands\printead[presep={,\
  }]{e2}}
\address[C]{Vrije Universiteit, Amsterdam, The Netherlands\printead[presep={,\ }]{e3}}
\address[D]{Centrum Wiskunde \& Informatica, Amsterdam, The Netherlands\printead[presep={,\ }]{e4}}
\end{aug}

\begin{abstract}
  We study worst-case-growth-rate-optimal (GROW) $\E$-statistics for hypothesis
  testing between two group models. 
  %PETER2 almost nobody knows what 'acts freely'means 
  % If the underlying group $G$ acts freely 
It is known that under a mild condition on the action of the underlying group $G$ on the data,
%on the observation space, 
there exists a maximally invariant statistic. We show that among all $\E$-statistics, invariant or not, the
  likelihood ratio of the maximally invariant statistic is GROW, both in the absolute and in the relative sense, and that an
  anytime-valid test can be based on it. 
%  By virtue of a representation theorem of Wijsman, 
The GROW $\E$-statistic is equal to a Bayes factor with a
  right Haar prior on $G$. Our treatment avoids nonuniqueness issues that sometimes arise for such priors in Bayesian contexts. A crucial
  assumption on the group $G$ is its amenability, a well-known group-theoretical
  condition, which holds, for instance, in
  %%
%  \{M: it used to say "in general scale-location families", but suggested to me that it also held for multivariate scale-location families---it does not.}
%P: one usually means that the scale is one-dimensional if one refers to scale-location. So I removed one-dimensional but did not re-include 'general'
  %%
  scale-location families.
Our results also apply to
  finite-dimensional linear regression.
\end{abstract}

%\begin{keyword}[class=MSC]
%\kwd[Primary ]{???}
%\kwd{???}
%\kwd[; secondary ]{???}
%\end{keyword}
%PETER 10/8: fill in the above and below before submission (? do we have to?)
%\begin{keyword}
%\kwd{???}
%\kwd{???}
%\end{keyword}

\end{frontmatter}
\newcommand{\commentout}[1]{}

\maketitle

\section{Introduction}
\label{sec:new-introduction}
We develop \E-statistics and anytime-valid methods \citep{ramdas2023savi} for composite hypothesis testing problems where both null and alternative models remain unchanged under a group of transformations. 
Assume that the parameter of interest is a function $\delta = \delta(\theta)$ that is invariant under these transformations. 
Here, $\theta \in \Theta$ is the parameter of a probabilistic model ${\cal P} = \{\mathbf{P}_\theta:\theta\in\Theta\}$ on an observation space ${\cal X}$. In the simplest case that we address, we are interested in testing whether the invariant parameter $\delta$ takes one of two values, that is, 
\begin{equation}
  {\cal H}_0: \delta(\theta) =  \delta_0
  \text{ \ vs. \ } {\cal
    H}_1:\delta(\theta) = \delta_1.
  \label{eq:target_hypothesis_problem}
\end{equation}
A prototypical example is the one-sample t-test where $\mathcal{P}=\{N(\mu,\sigma): (\mu,\sigma)\in \mathbb{R}\times \mathbb{R}^+\}$ and the parameter of interest is the effect size $\delta(\mu,\sigma) = \mu/\sigma$, an invariant function of the model parameters under changes of scale. 
%% 
%%This particular example can be extended to testing whether one of the coefficients in a regression model is %%zero \citep[for more examples, see][]{berger1998bayes}.{M:i felt this sentence "This particular %%example..." cuts the flow P: I suggest to replace by text in blue}
Other examples include tests about the correlation coefficient, which is invariant under affine transformations, and the variance of the principal components, an invariant under rotations around the origin  \citep[for more examples, see][]{berger1998bayes}.
Data can be reduced by only considering its invariant component. Roughly speaking, by replacing the data $X^n = (X_1, \dots, X_n)$ with an invariant statistic $M_n = m_n(X^n)$, one discards all information that is not relevant to the parameter $\delta$ (see the formal definitions in Section~\ref{sec:preparation}). For example, for the one-sample t-test, we can set $M_n$ equal to the t-statistic $M_{\mathcal{S},n} \propto \hat\mu_n/\hat\sigma_n$ but also to $M_n = \paren{X_1/\abs{X_1}, \dots, X_n / \abs{X_1}}$. Both are invariant functions under rescaling of all data points by the same factor that retain, as we will see, as much information as possible about the data.  

By reducing the data through an invariant function, an invariant test can be obtained. Through the lens of the invariance-reduced data $M_n$, the composite hypotheses about $\theta$ simplify and  \eqref{eq:target_hypothesis_problem} becomes simple-vs.-simple in terms of $\delta$. 
Indeed, because it is an invariant function, the density of $M_n$ depends only on $\delta$. Let us denote  $p^{M_n}$ and  $q^{M_n}$ the densities of $M_n$ under $\mathcal{H}_0$ and $\mathcal{H}_1$, respectively.
Both fixed-sample-size and sequential tests can be based on assessing the value of the likelihood ratio
\begin{equation}\label{eq:lr}
T^{M_n}  := \frac{q^{M_n}(m_n(X^n))}{p^{M_n}(m_n(X^n))}.
\end{equation} 
However, it is not a priori clear whether this reduction affects the optimality of the resulting tests. In other words, does the family of invariant tests, i.e. tests that can be written as a function of (\ref{eq:lr}), contain the best ones?

For fixed-sample size tests, with power as a criterion, the answer is positive: 
a celebrated theorem of Hunt and Stein~\cite[Section~8.5]{lehmann_testing_2005} shows that, when looking for a test that has max-min power, no loss is incurred by looking only among group-invariant tests.  In classical sequential testing, the principle of invariance has been used \citep{cox_sequential_1952,hall_relationship_1965}, but no optimality results are known. In this article we address this question and provide an analogue of the Hunt-Stein theorem  
within the setting of {\em anytime-valid\/} tests. 
We replace power by GROW (see again below), the natural optimality criterion in this context, and we show that, under some regularity conditions, $T^{M_n}$ is the {\em optimal $\E$-statistic\/}
for testing  \eqref{eq:target_hypothesis_problem}. 
%under some regularity conditions. 
%in show that an analogue of the theorem of Hunt and Stein holds when tests
%are based one instead considers  \E-statistics---also known as \E-values.
%%
% M: removed footnote
%\footnote{E-statistics are mostly known as $\E$-variables or, in analogy to $p$-values, %$\E$-values; we call them $\E$-statistics here to emphasize that they are, in fact, statistics %of the data}
%%Our results imply that  is optimal in this sense 
%This means that any test---sequential or otherwise---that monitors the value the %likelihood ratio $T^{M_n}$ makes use of an optimal $\E$-statistic. On first encounter, %the motivation to consider \E-statistics may not be clear.

The \E-statistic (also known as \E-variable or \E-value) is a central concept within the theory of \emph{anytime-valid} testing~\citep{vovk_e-values_2021,shafer_language_2019,grunwald_safe_2023,ramdas_admissible_2020}, interest in which has recently exploded ---\cite{ramdas2023savi} provide a comprehensive overview.
The main objective that is achieved by testing with $\E$-statistics is finite-sample type-I error control in two common situations: when experiments are optionally stopped---sampling is stopped at a data-dependent sample size---, and when aggregating the evidence of interdependent experiments. In the latter case, called optional continuation \citep[][\citetalias{grunwald_safe_2023} from now on]{grunwald_safe_2023}, the decision to start a new experiment may depend in unknowable ways on the outcome of previous experiments \citep{vovk_e-values_2021}. 
We will use the qualifier anytime-valid as an umbrella term that covers both optional stopping and continuation, and study invariance reductions for anytime-valid tests; we stress that, as elaborated in Appendix~\ref{app:av}, anytime-valid testing, while taking place in a sequential setting, is different from classical, Wald-style sequential testing, in which power {\em is\/} meaningful. While \E-statistics have also found applications beyond anytime-validity, for example in multiple testing \citep{WangR22,RenB22} and when not just the stopping time but also the relevant  loss function or significance level  may depend in unknowable ways on the data itself (decision-theoretic robustness, \cite{Grunwald23}), our results focus on optimality in the anytime-valid context. In this context, power is not a meaningful measure of optimality (see Section~\ref{sec:optimality}).
A natural replacement of power is the \emph{GROW} criterion, which stands for {\em growth rate optimal in the worst case}. 
Informally, among all \E-statistics, those that are GROW accumulate evidence against the null as fast as possible (in terms of sample size). Some other authors refer to GROW as `maximal  \E-power' \citep{zhang2023exact} or as `optimizing  the Kelly criterion' \citep{ramdas2023savi}. 
Sometimes, it is beneficial to consider instead the growth rate relative to an oracle that knows the distribution of the data, not in absolute terms. \E-statistics that are optimal in this relative sense are called relatively GROW.
Especially this relative criterion (or closely related variations of it) has often been used to design $\E$-statistics; recent examples include  \citep{henzi2023safe,waudby2020estimating}; see \cite{ramdas2023savi} for a more comprehensive list. 

Under regularity conditions, a GROW \E-statistic can be found by minimizing the Kullback-Leibler (KL) divergence between the convex hull of the null and alternative models (\citetalias{grunwald_safe_2023}). 
Indeed, the likelihood ratio of the distributions that achieve this minimum $\KL$ is a GROW \E-statistic. 
As such, \E-statistics can be seen as composite generalizations of likelihood ratios. 
In particular, any likelihood ratio of a statistic that has the same distribution under all elements of the null and another single distribution under the alternative is an $\E$-statistic (\citetalias{grunwald_safe_2023}). 
As a consequence, for any invariant function of the data $M_n$, the likelihood ratio statistic $T^{M_n}$ from \eqref{eq:lr} is an $\E$-statistic for the testing problem \eqref{eq:target_hypothesis_problem}. 
As our main contribution, we show that, under regularity conditions, if $M_n$ is a maximally invariant statistic of the data or of a sufficient statistic for $\theta$, then the $\KL$ divergence between $q^{M_n}$ and $p^{M_n}$ equals the minimum $\KL$ divergence between the convex hulls of the null and alternative models. By the result of \citetalias{grunwald_safe_2023} mentioned above that links $\KL$ minimization to GROW \E-statistics, $T^{M_n}$ is  GROW. A maximally invariant statistic, informally, loses as little information as possible about the data while being invariant. 
For example, 
with $V_n = (X_1/|X_1|, \dots, X_n/|X_1|)$, 
setting $M_n := V_n$ as in the beginning of the introduction for the t-test gives a maximal invariant, while using $M_n' := V_{n-1}$ gives an invariant that is not maximal. 
Furthermore, the t-statistic is not maximally invariant for the raw data, but it is a maximally invariant function of $(\hat\mu_n,\hat\sigma_n)$ which is a sufficient statistic. Consequently, the likelihood ratio statistic $T^{M_{\mathcal{S},n}}$, where   $M_{\mathcal{S},n}$ is the t-statistic, and $T^{M_n}$ with $M_n = V_n$ coincide and are both GROW.

Additionally, we show that the GROW \E-statistic coincides with the relatively GROW \E-statistic in the group-invariant setting. Hence, $T^{M_n}$ is relatively GROW as well. 
This growth rate optimality motivates the use of $T^{M_n}$ in optional continuation settings. 
As a further contribution, we show that every time that $M_n$ is a maximal invariant the sequence $T= (T^{M_n})_{n \in {\mathbb N}}$ is a nonnegative martingale. 
This extends its use and optimality to optional stopping.

The rest of this article is organized as follows. 
In Section~\ref{sec:preparation} we introduce notation, formally lay the groundwork for group-invariant testing, review \E-statistics and their optimality criteria, and discuss related work. 
Section~\ref{sec:main} is devoted to stating our main results:
%PETER2 it doesn't just 'include', the list is exhaustive
%, which includes 
showing that the \E-statistic $T^{M_n}$ for a maximally invariant function $M_n=m_n(X^n)$ is both GROW and relatively GROW, proving that $T^{M_n}$ is suited for both optional continuation and optional stopping, and extending these results to composite hypotheses, 
%PETER added this. I don't really like it but it's just to please referee 1
i.e. sets $\Delta_1$ and $\Delta_0$ of $\delta$'s, both with and without a prior distribution imposed on them (for general discussion on how to choose $\delta_j,\Delta_j$ or such priors, we refer to \citetalias[Section 6]{grunwald_safe_2023}).
Next, in Section~\ref{sec:test-mult-norm}, we apply our results to two examples. 
We end this article with Section~\ref{sec:discussion}, where we provide additional discussion about the technical conditions that our results require and about related work on group-invariant testing; and Section~\ref{sec:proofs}, where we give all proofs that were omitted earlier. 

\section{Preparation for the Main Results}\label{sec:preparation} 
This section is structured as follows. We first introduce notation. 
Then, in section~\ref{sec:intro-group-invariance}, we introduce the formal setup and our running example, the t-test. 
In Section~\ref{sec:intro-e-statistics}, we define $\E$-statistics, our main objects of study, and in Section~\ref{sec:optimality} we define our optimality criteria.
Finally, Section~\ref{sec:previous_work} highlights previous work.

%TODO: `a measurable function T=T(X^n)' might be confusing, when X^n is already defined as a random variable. Either it is then a statistic, it should be T=T(x^n)

\subsection{Notation}\label{sec:notation}   
We write $X$ for a random variable taking values in the observation space $\mathcal{X}$, endowed with a measurable structure, and $X^n:=(X_1,\dots,X_n)$ for $n$ independent copies of $X$ under the distributions that are to be considered. Statistics of the data are denoted as $T=t(X^n)$, where $t$ is implicitly assumed to be a measurable function.
We use letters $\mathbf{P}$ and $\mathbf{Q}$ to refer to distributions of $X$.
For a statistic $T = t(X^n)$, we write $\mathbf{P}^T$ for the image measure of $\mathbf{P}$
under $t$, that is, $\mathbf{P}^T\{T\in B\} = \mathbf{P}\{t(X^n)\in B\}$. 
When writing conditional expectations, we write $\mathbf{E}^{\mathbf{P}}[f(X) | Y]$,
and $\mathbf{P}^{X| y}$ for the conditional distribution of $X$ given
$Y=y$. 
We only deal with situations where such conditional distributions exist.
If we are considering a set of distributions parameterized in terms of a parameter space $\Theta$, we write $\mathbf{E}^\mathbf{P}_\theta[f(X)]$ rather than $\mathbf{E}^{\mathbf{P}_\theta}[f(X)]$ for the sake of readability. 
Furthermore, for a prior distribution $\mathbf{\Pi}$ on $\Theta$, we write $\mathbf{\Pi}^\theta\mathbf{P}_\theta$
for the marginal distribution that assigns probability
$\mathbf{\Pi}^\theta \mathbf{P}_\theta\{X\in B\} = \int
\mathbf{P}_{\theta}\{X\in B\}\rmd\mathbf{\Pi}(\theta)$ to any measurable set
$B$. For the posterior distribution of $\theta$ given $X$ we write
$\mathbf{\Pi}^{\theta | X}$. 
The Kullback-Leibler (KL) divergence between $\mathbf{Q}$ and $\mathbf{P}$ is denoted by  $\KL(\mathbf{Q}, \mathbf{P}) = \mathbf{E}^\mathbf{Q}[\ln( \mathrm{d}Q / \mathrm{d}P)]$~\citep{kullback_information_1951}.
Given two subsets $H,K$ of a group $G$ we write
$HK = \{hk : h\in H, k\in K\}$ for the set of all products between
elements of $H$ and elements of $K$. Similarly, for $g\in G$ and $K\subseteq G$, we write $gK = \{gk : k\in K\}$ for the translation of $K$ by
$g$, and $K^{-1} = \{k^{-1} : k\in K\}$ for the set of inverses of $K$. 
We say that $K$ is symmetric if $K = K^{-1}$.
If $G$ acts on $\mathcal{X}$, then we denote the action of $G$ on ${\cal X}$ by $(g,x)\mapsto gx$ for $g\in G$ and $x\in {\cal X}$, and extend the action to ${\cal X}^n$ component-wise; that is, $(g, x^n)\mapsto gX^n := (gx_1, \dots, gx_n)$ for $g\in G$ and $x^n \in {\cal X}^n$.
We write $gB=\{gb:b\in B\}$ for the left translate of a subset $B\subseteq \mathcal{X}$ by $g$.
If $G$ acts on $\Theta$, the notation is completely analogous.

\subsection{Group invariance}
\label{sec:intro-group-invariance}
We consider a group $G$ that acts freely on both the observation space $\mathcal{X}$ and the parameter space $\Theta$.
%PETER2 'The fact' is strange. You might say 'the statement', but even better: 
%. The fact that the action is free means that, if $gx=x$ for some $g\in G$ and $x\in \mathcal{X}$, then $g$ is the identity element of the group. This definition is completely analogous for the action of $G$ on $\Theta$. 
Recall that $G$ acts freely on a set ${\cal Z}$ if anytime that $gz=z$ for some $g \in G$ and $z \in \mathcal{Z}$, then $g$ is the identity element of the group $G$. 
A probabilistic model ${\cal P} = \{\mathbf{P}_\theta:\theta\in\Theta\}$ on ${\cal X}$ is said to be invariant under the action of $G$ if the distribution $\mathbf{P}_\theta$ satisfies
\begin{equation}\label{eq:def_invariant_model}
  \mathbf{P}_{\theta}\{X\in B\} = \mathbf{P}_{g\theta}\{X\in gB\}
\end{equation}
for any $g\in G$, measurable $B\subseteq {\cal X}$, and $\theta\in\Theta$.
Furthermore, a function $m(x)$ is said to be invariant under the action of $G$ if $m(gx)=m(x)$ for all $x\in \mathcal{X}$ and $g\in G$; in other words, $m$ is constant 
%PETER2 a word was missing here. added 'on' (is that the right proposition?) or `within'?
on
the orbits of $G$.
Moreover, $m$ is said to be maximally invariant if it indexes the orbits of $\mathcal{X}$ under the action of $G$; that is, $m(x)=m(x')$ for $x,x'\in \mathcal{X}$ if and only if there exists a $g\in G$ such that $x=gx'$. 
A statistic is called (maximally) invariant if the corresponding function is.
These definitions are 
%PETER2 removed 
%again 
completely analogous for functions defined on $\Theta$.
In particular, we study situations where the parameter of interest $\delta = \delta(\theta)$ is a maximally invariant function of the parameter $\theta$. 
We then say that $\delta$ is a maximally invariant parameter.

We now reparametrize the problem
described in (\ref{eq:target_hypothesis_problem}) using the group $G$. Using that the action of the
group on the parameter space is free, we can reparametrize each orbit in
$\Theta/G$ with $G$. Indeed, we can pick an arbitrary but fixed element in the
orbit $\theta_0\in \delta^{-1}(\delta_0)$ and, for any other element $\theta\in \delta^{-1}(\delta_0)$, we can identify $\theta$ with the group element $g(\theta)\in G$ that transports $\theta_0$ to $\theta$, that is, such that $g(\theta)\theta_0 = \theta$. 
Hence, with a slight abuse of notation, we can identify $\theta\in\delta^{-1}(\delta_0)$ with $g = g(\theta)\in G$ and identify $\mathbf{P}_{\theta} = \mathbf{P}_{g(\theta)\theta_0}$ with $\mathbf{P}_g$.
The same identification can be carried out in the alternative model by an analogous choice of $\theta_1\in \delta^{-1}(\delta_1)$. 
The starting problem
(\ref{eq:target_hypothesis_problem}) may now be rewritten in the form
\begin{equation}\label{eq:main_problem_group_parametrization}
  {\cal H}_0: X^n\sim \mathbf{P}_{g}, \  g\in G
  \text{, vs. }
  {\cal H}_1: X^n\sim \mathbf{Q}_{g}, \ g\in G. 
\end{equation}
To make notation more succinct, we use ${\cal Q} = \{\mathbf{Q}_{g}\}_{g\in G}$ to denote the alternative hypothesis and
${\cal P} = \{\mathbf{P}_{g}\}_{g\in G}$ for the null. 
%We assume that each member of ${\cal Q}$ is absolutely continuous with respect to each member of ${\cal P}$.
%Note that, a
As will follow from our discussion, our results are insensitive to the choices of $\theta_0\in \delta^{-1}(\delta_0)$ and $\theta_1\in\delta^{-1}(\delta_1)$. 
%TODO: I think the \mathcal{P} and \mathcal{Q} notation is barely used, and can be removed.

As mentioned in the introduction, tests for~\eqref{eq:main_problem_group_parametrization} are classically based on the likelihood ratio  $T^{M_n}$ of a maximally invariant statistic $M_n=m_n(X^n)$, as in~\eqref{eq:lr}.
While the distribution of $M_n$ might be unknown, it is well-known that its likelihood ratio can be computed by integration over the group $G$ whenever the following hold:
(1) the action is continuous and proper, 
(2) $G$ is a $\sigma$-compact locally compact  topological group, and 
(3) for all $g$, $\mathbf{P}_g$ and $\mathbf{Q}_g$ are dominated by a relatively left invariant measure $\nu$. 
In (1), an action is proper if the map $G\times \mathcal{X}^n\to\mathcal{X}^n\times \mathcal{X}^n$ defined by $(g,x^n)\mapsto (gx^n,x^n)$ is proper, that is, the inverse of any compact set is compact. 
In (2), a topological group is a group equipped with a topology, such that the group operation, seen as a function $G\times G\to G$, is continuous. 
Furthermore, since $G$ is assumed to be locally compact, there exists a measure $\rho$ on $G$ that is right invariant~\citep[see][VII,§1,n\textsuperscript{o} 2]{bourbaki_integration_2004}.
This means that for any $g\in G$ and any $B\subseteq G$ that is measurable, it holds that $\rho\{Bg\} = \rho\{B\}$. 
The measure $\rho$ is called the Haar measure, it is unique up to a multiplicative factor, and it is finite if and only if $G$ is compact. 
Using disintegration-of-measure results from \citet[][VIII.27]{bourbaki_integration_2004}, \citet{andersson_distributions_1982} shows that $T^{M_n}$ can be
computed as
\begin{equation}\label{eq:papi_wijsman}
  T^{M_n}
  =
  \frac{q^{M_n}(m_n(X^n))}{p^{M_n}(m_n(X^n))}
  =
  \frac{
    \int_Gq_{g}(X^n)\rmd\rho(g)
  }{
    \int_Gp_{g}(X^n)\rmd\rho(g )
  },
\end{equation}
where $p_g$ and $q_g$ denote the densities of $\mathbf{P}_g$ and $\mathbf{Q}_g$ respectively.
This is known as Wijsman's representation theorem \citep[for extended statement and discussion, see ][Theorem
5.9]{eaton_group_1989}.
Note that \eqref{eq:papi_wijsman} implies that the likelihood ratio $T^{M_n}$ is independent of the choice of maximal invariant $M_n$.
Remarkably, work by Stein, reported by~\cite{hall_relationship_1965}, shows that it does not even matter whether we consider a maximal invariant of the original data, or whether we first reduce the data through sufficiency and then consider a maximal invariant of the sufficient statistic.
In the t-test example, this shows that the likelihood ratio of the t-statistic is equal to that of $M_n$ as in the start of the introduction. We further discuss this result in Appendix~\ref{app:sufficiency}.

Finally, the classical theorem of Hunt and Stein \cite[][Section~8.5]{lehmann_testing_2005} shows that, under some regularity conditions, when looking for a test that is max-min optimal in the sense of power, it is sufficient to look among invariant tests, i.e. tests that can be written as a function of $T^{M_n}$ as in (\ref{eq:lr}). One of the crucial assumptions underlying their result is the \emph{amenability} of $G$.
%TODO: Why do we define amenability with a definition, but not proper action, free action, topological group, etc.? I would vote to be consistent.
% \begin{definition}[Amenability]\label{def:amenability_as_huntstein}
%   A group $G$ is amenable if there exists a sequence of
%   \textit{almost-right-invariant} probability distributions, that is, a sequence
%   $\mathbf{\Pi}_1, \mathbf{\Pi}_2, \dots$ such that, for any measurable set
%   $B\subseteq G$ and group element $g\in G$,
%   \begin{equation*}
%     \lim_{k\to\infty}
%     \abs{
%       \mathbf{\Pi}_k\bracks{B}
%       -
%       \mathbf{\Pi}_k\bracks{Bg}
%     }
%     =
%     0.
%   \end{equation*}
% \end{definition}
A group $G$ is amenable if there exists a sequence of almost-right-invariant probability distributions, that is, a sequence $\mathbf{\Pi}_1,\mathbf{\Pi}_2,\dots$ such that, for any measurable set $B\subseteq G$ and $g\in G$  
\begin{equation*}
    \lim_{k\to\infty}
    \abs{
      \mathbf{\Pi}_k\bracks{H\in B}
      -
      \mathbf{\Pi}_k\bracks{H \in Bg}
    }
    =
    0.
\end{equation*}
Amenable groups have been thoroughly studied \citep{paterson_amenability_2000}
and include, among others, all finite, compact, commutative, and solvable
groups. 
The easiest example of a nonamenable group is the free group in two elements and
any group containing it. 
Another prominent example of a nonamenable group is that of
invertible $d\times d$ matrices with matrix multiplication.

\begin{example}[t-test under Gaussian assumptions]\label{ex:t-test}
  Consider an i.i.d.\ sample $X^n = (X_1,\dots,X_n)$ of size $n\in\nats$ from an unknown
  Gaussian distribution $N(\mu,\sigma)$, with $\mu\in \reals$ and $\sigma \in \reals^+$. 
  In the 1-sample t-test, we are interested in testing whether $\mu/\sigma = \delta_0$ or $\mu/\sigma = \delta_1$ for some $\delta_0,\delta_1\in\reals$. 
  For $c\in \mathbb{R}^+$, we have that $cX \sim N(c\mu,c\sigma)$, so it follows that the Gaussian model is invariant under scale transformations.
  The corresponding group is $G = (\reals^+, \ \cdot\ )$, which acts on $\mathcal{X}^n$ by component-wise multiplication and on $\Theta$ by $(c, (\mu,\sigma))\mapsto (c\mu,c\sigma)$ for each $c\in G$ and $(\mu,\sigma)\in \Theta$. 
  The parameter of interest, $\delta = \mu/\sigma$, is scale-invariant and indexes the orbits of the action of $G$ on $\Theta$. 
  A maximally invariant statistic is $M_n := 
   \paren{X_1/\abs{X_1}, \dots, X_n / \abs{X_1}}$.
  The right Haar measure $\rho$ on $G$ is given by
  $\rmd \rho (\sigma)= \rmd \sigma / \sigma$, so that the likelihood ratio of $M_n$ can be expressed, as in \eqref{eq:papi_wijsman}, by
  \begin{equation}\label{eq:ttesthaar}
    T^{M_n} =
    \frac{
      \int_{\sigma > 0}
      \frac{1}{\sigma^{n}}
      \exp
      \left(
        - \frac{n}{2}
        \left[
          \left(
            \frac{\bar{X}_n}{\sigma} - \delta_1
          \right)^2
          +
            \frac{1}{n} \sum_{i=1}^n \paren{\frac{X_i - \bar{X} }{\sigma}}^2
        \right]
      \right)  \frac{\rmd \sigma}{\sigma}
    }
    {
      \int_{\sigma > 0} \frac{1}{\sigma^{n}}
      \exp \left(- \frac{n}{2}
        \left[
          \left( \frac{\bar{X}_n}{\sigma} - \delta_0
          \right)^2
          +
            \frac{1}{n} \sum_{i=1}^n \paren{\frac{X_i - \bar{X} }{\sigma}}^2
        \right]
        \right)
        \frac{\rmd \sigma}{\sigma}
      },
  \end{equation}
  where $\bar{X}_n := \frac{1}{n}\sum_{i=1}^n X_i$.
  The results by Stein, discussed in Appendix~\ref{app:sufficiency}, show that the likelihood ratio of the t-statistic, i.e. $M_{\mathcal{S},n}\propto \hat\mu_n/\hat\sigma_n$, is equal to the expression obtained in \eqref{eq:ttesthaar}.
  % The expression on the right hand side of \eqref{eq:ttesthaar} was first obtained by \citet{cox_sequential_1952} who realized that it was equivalent to the likelihood ratio of the maximal invariant. 
\end{example}

\subsection{The family of \E-statistics, and optional continuation and stopping}
\label{sec:intro-e-statistics}
We now define  $\E$-statistics, our measure of evidence against the null hypothesis. 
The family of $\E$-statistics comprises all nonnegative real statistics whose expected value is bounded by one under all elements of the null, that is, all statistics $T_n=t_n(X^n)$ such that $T_n \geq 0$ and
\begin{equation} %TODO: Changed notation from $\mathbf{E}_{\theta_0}^\mathbf{P}$
  \label{eq:def_evalue_intro}
  \sup_{g\in G}\mathbf{E}^{\mathbf{P}}_g[T_n]\leq 1.
\end{equation}
An example of an \E-statistic is the likelihood ratio statistic in any simple-vs-simple testing problem (see e.g. \citetalias[Section 1]{grunwald_safe_2023} or \citet{ramdas2023savi}). In particular, \eqref{eq:lr} is an $\E$-statistic for the hypotheses in~\eqref{eq:main_problem_group_parametrization}. 
\E-statistics are appropriate in optional continuation contexts because of the following two properties that  are consequences of \eqref{eq:def_evalue_intro}.
\begin{enumerate}
\item The type-I error of the test that rejects the null hypothesis anytime that
  $T_n\geq 1 / \alpha$ is smaller than $\alpha$, a consequence of
 \eqref{eq:def_evalue_intro} and Markov's inequality.
\item Suppose that $X^n$ and $X^m$ are the independent outcomes of two subsequent experiments.
Let $T_n=t_n(X^n)$ be an $\E$-statistic for $X^n$ and let $\{T_{m,\varphi}: \varphi\in \Phi\}$ be a family of $\E$-statistics for $X^m$ indexed by some set $\Phi$. 
Suppose further that, after observing the first sample $X^n$, the specific $T_{m,\varphi}$ used to measure evidence for
the second sample is chosen as a function of $X^n$, that is, we use $T_{m,\hat\varphi}$ where $\hat\varphi = \hat\varphi(X^n)$ is some function of $X^n$.
Then
  $T_{n + m}:= T_nT_{m,\hat\varphi}$ is also an $\E$-statistic, irrespective of the definition of $\hat\varphi$.
%PETER2 added
In particular, this includes the scenario where we only continue to the second experiment if a certain outcome is observed in the first one. Indeed, $\Phi$ may contain a special value ${\bf 1}$ so that $t_m(X^m; {\bf 1})= 1$ is constant, irrespective of $X^m$. Then, $T_{n+m}= T_n$ every time that  $\hat\varphi = {\bf 1}$. 
\end{enumerate}
Together, these two properties imply that the test that rejects the null if $T_{n+m} \geq 1/\alpha$ has type-I error bounded by $\alpha$, no matter the definition of $\hat\varphi$.
Such type-I error guarantees are essentially impossible using p-values (\citetalias[Section 1.3]{grunwald_safe_2023}). 
Some---not all---types of $\E$-statistics can additionally be used in two related settings: (a) \emph{optional stopping}, when there is a single sequence of data $X_1, X_2, \ldots$
and we want to do a test with type-I error guarantees based on all data seen so far, irrespective of
when we stop; and (b) optional continuation as in 2. 
above, but with individual $\E$-statistics whose sample size is itself not fixed but determined by some stopping rule. 
As is well-known, for both (a) and (b) it is sufficient that $(T_n)_{n \in {\mathbb N}}$ is a nonnegative martingale with respect to some filtration $\mathcal{F}$ \citep[see e.g.][or \citetalias{grunwald_safe_2023}]{ramdas2023savi}.
The first part follows from Ville's inequality for nonnegative martingales: the probability that there will {\em ever\/} be a sample size $n$ at which $T_n \geq 1/\alpha$ is bounded by $\alpha$. 
We thus have type-I error control under optional stopping, which takes care of (a) above. 
The optional stopping theorem implies that for every stopping time $\tau$ adapted to $\mathcal{F}$, $T_{\tau}$ is also an $\E$-statistic, taking care of (b). 
For completeness, we provide more details in Appendix~\ref{app:av}, including a subtlety regarding (b): while they seem unlikely to arise in practice, there do exist stopping times $\tau'$ relative to the data that are not stopping times relative to ${\cal F}$. We show an example where $T_{\tau'}$ is not an \E-statistic and (b) breaks. 

\subsection{Optimality criteria for \E-statistics}
\label{sec:optimality}
The standard optimality criterion for hypothesis tests satisfying a certain type-I
error guarantee is worst-case
power maximization for a  fixed-sample-size or, with classic sequential tests, for a fixed stopping rule. 
This criterion cannot be used when the stopping rule is unknown because knowledge of the stopping rule is required by the definition of
power. 
Additionally, an $\E$-statistic that optimizes power at fixed stopping time will take the value zero with positive probability, making it useless for optional continuation by multiplication. 
A more sensible criterion for \E-statistics under optional continuation is growth rate optimality in the worst case (\citetalias{grunwald_safe_2023}). Should it exist, an
$\E$-statistic $T^*_n$ is GROW if it maximizes the worst-case expected
logarithmic value under the alternative hypothesis, that is, if it maximizes
\begin{equation} \label{eq:grow_objective}
  T_n
  \mapsto
  \inf_{g\in G}\mathbf{E}^{\mathbf{Q}}_g[\ln T_n]
\end{equation}
over all $\E$-statistics. 
The following theorem, 
%PETER2 'adapted'suggests that its contents has changed
%adapted 
%% M: Changed it
stated in our notation
for group-invariant problems, shows that in most cases the GROW \E-statistic takes the form of a particular Bayes factor.

\begin{theorem}[\citetalias{grunwald_safe_2023} Theorem~1 in Section~4.3]
  \label{theo:ghk_theo1}
  Suppose that there exists a statistic $V_n=v_n(X^n)$  such that
  \begin{equation}
    \label{eq:whattominimize}
    \inf_{\mathbf{\Pi}_0, \mathbf{\Pi}_1}
    \KL(\mathbf{\Pi}_1^{g} \mathbf{Q}_{g},
    \mathbf{\Pi}_0^{g} \mathbf{P}_{g})
    =
    \min_{\mathbf{\Pi}_0, \mathbf{\Pi}_1}
    \KL(\mathbf{\Pi}_1^{g}\mathbf{Q}^{V_n}_{g}, \mathbf{\Pi}_0^{g}\mathbf{P}^{V_n}_{g})
    < \infty,
  \end{equation}
  where $\mathbf{\Pi}_0$ and $\mathbf{\Pi}_1$ are probability distributions on $G$. 
  Let $\mathbf{\Pi}^\star_0$ and $\mathbf{\Pi}^\star_1$ be probability distributions that achieve the minimum on the right hand side.
  % achieved, that is,
  % \begin{equation*}
  %   \min_{\mathbf{\Pi}_0, \mathbf{\Pi}_1}
  %   \KL(\mathbf{\Pi}_1^{\theta_1}\mathbf{Q}^{V_n}_{\theta_1}, \mathbf{\Pi}_0^{\theta_0}\mathbf{P}^{V_n}_{\theta_0})
  %   =
  %   \KL(\mathbf{\Pi}^{\star\theta_1}_1\mathbf{Q}^{V_n}_{\theta_1}, \mathbf{\Pi}^{\star\theta_0}_0\mathbf{P}^{V_n}_{\theta_0}).
  % \end{equation*}
  Then 
  \begin{equation*}
    \max_{T_n \text{ $\E$-stat.}}\inf_{g\in G}\mathbf{E}^{\mathbf{Q}}_g[\ln T_n]
    = \KL(\mathbf{\Pi}^{\star g}_1\mathbf{Q}^{V_n}_{g}, \mathbf{\Pi}^{\star g}_0\mathbf{P}^{V_n}_{g}),
  \end{equation*}
%  In that case, the maximum on the left hand side of the previous display 
and the maximum on the left is
  achieved by $T^*_n$ as given by
  \begin{equation*}
    T_n^* :=
    \frac{
      \int q^{V_n}_{g}(v_n(X^n))\rmd\mathbf{\Pi}_1^\star(g)
    }{
      \int p^{V_n}_{g}(v_n(X^n)) \rmd \mathbf{\Pi}^{\star}_0(g)
    }. 
  \end{equation*}
In other words, the $\E$-statistic $T^*_n$ is GROW for testing $\{\mathbf{P}_g\}_{g\in G}$ against $\{\mathbf{Q}_g\}_{g\in G}$. 
 % that is, $T^*_n$ is GROW for testing ${\cal P}$ against ${\cal Q}$.
\end{theorem}

% The objective here is to gather evidence, measured by
% $T_n(X^n)$, as fast as possible. To this end, it is sensible to maximize
% expectation of $f(T_n(X^n))$ under the alternative, for some increasing function
% $f$. \citet{shafer_language_2019} and  argue extensively why it makes sense
% to take $f$ as the logarithm, an idea also known as {\em Kelly betting\/}
% \citep{Kelly56}. 
% Relatedly, this criterion produces tests with the smallest
% expected sample size until the null can rejected in a specific testing setting
% \citep{breiman_optimal_1961}. As a consequence, as pointed out and investigated by  (Section 6 of their paper), this finding allows us to {\em relate\/} growth optimality to power after all: 
% suppose we fix some maximum sample size $n_{\max}$ yet  if there exists $n< n_{\max}$ at which $T_n \geq 1/\alpha$, we stop (and reject the null) early, at the smallest such $n$. The power of this procedure increases with the growth rate  $\mathbf{E}_{\theta_1}[\ln T_n(X^n)]$: optimal growth rate leads to high, but not optimal power (it is usually within a small constant factor). 
The statistic $V_n$ may be any measurable function, taking values in any set ${\cal V}_n$ equipped with a corresponding $\sigma$-algebra, but in all examples in our paper we can take ${\cal V}_n = \reals^m$ for some $m \leq n$. By allowing $V_n \neq X^n$, the theorem also covers cases in which the infimum on the left in (\ref{eq:whattominimize}) is not achieved. This will be the case when in the next section we apply Theorem~\ref{theo:ghk_theo1} to obtain Corollary~\ref{cor:main_corollary} 
whenever, as in the t-test example, the group $G$ is not compact. 

Given its worst-case nature, the GROW $\E$-statistic, while appropriate in some scenarios (e.g. testing exponential families with given minimum effect sizes and no nuisance parameters), is too
conservative in others (\citetalias{grunwald_safe_2023}).
\citetalias{grunwald_safe_2023} propose, for those cases, to maximize a relative form of
(\ref{eq:grow_objective}), leading to less conservative $\E$-statistics. 
We say that an $\E$-statistic $T^*_n$ is relatively GROW if
it maximizes the gain in expected logarithmic value relative to an oracle that
is given the particular distribution in the alternative hypothesis from which
data are generated, that is, if $T^*_n$ maximizes, over all $\E$-statistics,
\begin{equation}\label{eq:relative_grow_objective}
  T_n
  \mapsto
  \inf_{g\in G}
  \bracks{
    \mathbf{E}_{g}^{\mathbf{Q}}[\ln T_n]
    -
    \sup_{T'_n\text{ \E-stat.}}\mathbf{E}_{g}^{\mathbf{Q}}[\ln T'_n]
  }.
\end{equation}
As we will see and contrary to the general case, in the group-invariant
setting, any GROW $\E$-statistic is also relatively GROW. 
Hence, both criteria coincide and the differences that have been observed between them (raising the sometimes difficult question: which one to choose?) 
are not a concern for our purposes \citep{ramdas2023savi}.

\subsection{Previous and related work}
\label{sec:previous_work}
Group-invariant problems have a long tradition in statistics. They have been studied both for fixed-sample-size experiments \cite{eaton_group_1989,lehmann_testing_2005} and classical, Wald-type sequential experiments 
\citep{Rushton50,cox_sequential_1952}. 
For fixed-sample-size tests, our main result can be viewed, to some extent, as an anytime-valid analogue of the Hunt-Stein theorem. 
The proof techniques that are needed for our result are, however, distinct. 
At the core of the proof of the Hunt-Stein theorem lies the fact that the power is a linear function of the test under consideration. In its proof, an approximate symmetrization of the test is carried out using almost-right-invariant priors without affecting power guarantees. 
This line of reasoning cannot be directly translated to our setting because of the nonlinearity of the objective function that characterizes the optimal $\E$-statistics that we consider (see Section~\ref{sec:optimality}).
As for sequential tests with group invariance, most previous work (including the pioneering \cite{Rushton50,cox_sequential_1952} and in fact, as far as we could ascertain, all work pre-dating \cite{robbins1970statistical}) dealt, like Wald's original SPRT, with a priori fixed stopping rules and is not directly comparable to our anytime-valid work (see Appendix~\ref{app:av} for elaboration of this point).
Notable exceptions are the works of \cite{robbins1970statistical} and \cite{lai1976confidence}, who do consider anytime validity.
\citet{lai1976confidence} also used the expression in~\eqref{eq:ttesthaar}  for the t-test, which, in our terminology, is using the fact that it gives an \E-statistic.
However, our main concern, optimality of \E-statistics, has not been explored in this context. 

Related ideas can also be found in the Bayesian literature, where group-invariant inference with right Haar priors has been studied \citep{dawid_marginalization_1973,berger1998bayes}. 
It has been shown that, in contrast to some other improper priors, inference based on right Haar priors yields admissible procedures in a decision-theoretical sense \citep{eaton_group_2002,eaton_consistency_1999}.
However, there have also been concerns that the underlying group (and hence the right Haar prior) is not uniquely defined in some situations, and that different choices lead to different conclusions \citep{sun_objective_2007,berger_objective_2008}.
Interestingly, as we  briefly discuss in Section~\ref{sec:discussion} and at length in Appendix~\ref{app:herecomesthesun}, this issue cannot arise in our setting. 
In the same appendix, we point out similarities and the main difference to the information-theoretic work of \citet{liang2004exact}, who provide exact min-max procedures for predictive density estimation for general location and scale
families under Kullback-Leibler loss. 
In a nutshell, despite some similarities, the precise min-max result that they prove is not comparable to the results presented here.

\section{Main Results}\label{sec:main}
In this section, we state the main results of this article. 
In Section~\ref{sec:simp-invar-hypoth} we show that the likelihood ratio  $T^{M_n}$ for a maximal invariant $M_n$ is simultaneously GROW and relatively GROW.
Next, in Section~\ref{sec:anyt-valid-test}, we show that $T^{M_n}$ can be used to build an anytime-valid test.
Finally, in Section~\ref{sec:comp-invar-hypoth} we extend these results to the case that the hypotheses remain composite after reduction by invariance.

\subsection{GROW for simple invariant hypotheses}\label{sec:simp-invar-hypoth}
In order to build intuition, we first demonstrate our line of reasoning using the very special case of finite groups. 
%Below, Assumption~\ref{ass:summary_topological} contains all assumptions required in the main result, Theorem~\ref{theo:main_theorem}. To start, 
So, assume for now that $G$ is a finite group, for instance, a group of permutations. 
Since the uniform probability distribution $\mathbf{\Pi}_{\mathrm{U}(G)}$ on $G$ is right invariant, the Haar measure $\rho$ coincides with $\mathbf{\Pi}_{\mathrm{U}(G)}$ up to scaling. 
By Wijsman's representation theorem~\eqref{eq:papi_wijsman}, the likelihood ratio for any maximal invariant $M_n = m_n(X^n)$ can be written as
\begin{equation}\label{eq:baby_wijsman}
  T^{M_n} =
  \frac{q^{M_n}(m_n(X^n))}{ p^{M_n}(m_n(X^n))}
  =
  \frac{\tfrac{1}{|G|}\sum_{g\in G}q_{g}(X^n)}{\tfrac{1}{|G|} \sum_{g\in G}p_{g}(X^n)}.
\end{equation}
Furthermore, Theorem~\ref{theo:ghk_theo1} above takes a simple form for finite parameter spaces, as is the case here, namely
\begin{equation}\label{eq:baby_kl_maxmin_duality}
  \max_{T_n \text{ $\E$-stat. }}\min_{g\in
    G}\mathbf{E}^{\mathbf{Q}}_g[\ln T_n]
  =
  \min_{\mathbf{\Pi}_0, \mathbf{\Pi}_1}
  \KL(\mathbf{\Pi}_1^{g} \mathbf{Q}_{g}, \mathbf{\Pi}_0^{g}
  \mathbf{P}_{g}),
\end{equation}
where the minimum on the right hand side is taken over all pairs of distributions on $G$. 
We now employ the information processing inequality \citep[Section
2.8]{cover_elements_2006} which says that  $\KL$ divergence decreases when taking functions of the data (i.e. if $\mathbf{A}$ and $\mathbf{B}$ are distributions for $X$ and $U= u(X)$, then $\KL(\mathbf{A} \| \mathbf{B}) \geq \KL(\mathbf{A}^{U} \| \mathbf{B}^{U})$). In our setting, the information processing equality implies that for any pair $(\mathbf{\Pi}_0,\mathbf{\Pi}_1)$ of probability distributions on $G$, 
\begin{equation}\label{eq:maxmin3}
  \KL(\mathbf{\Pi}_1^{g} \mathbf{Q}_{g}, \mathbf{\Pi}_0^{g}
  \mathbf{P}_{g})
  \geq
  \KL(\mathbf{Q}^{M_n}, \mathbf{P}^{M_n}).
  %=
  %\min_{g\in G}\mathbf{E}^{\mathbf{Q}}_{g}[\ln T^{M_n}].
\end{equation}
This lower bound can be rewritten as $\KL(\mathbf{Q}^{M_n}, \mathbf{P}^{M_n})= \KL(\mathbf{\Pi}_{\mathrm{U}(G)}^{g} \mathbf{Q}_{g}, \mathbf{\Pi}_{\mathrm{U}(G)}^{g} \mathbf{P}_{g})$ because of the second equality in \eqref{eq:baby_wijsman}.
Therefore, the minimum $\KL$ on the right hand side of \eqref{eq:baby_kl_maxmin_duality} is achieved for the particular choice of two uniform priors on $G$. 
Finally, we have that $\mathbf{E}_g^{\mathbf{Q}}[\ln T^{M_n}]=\KL(\mathbf{Q}^{M_n}, \mathbf{P}^{M_n})$ for all $g\in G$. 
Putting everything together
\begin{equation*}
  \max_{T_n \text{ $\E$-stat. }}\min_{g\in
    G}\mathbf{E}^{\mathbf{Q}}_g[\ln T_n]=\KL(\mathbf{Q}^{M_n},\mathbf{P}^{M_n})=\min_{g\in G}\mathbf{E}^{\mathbf{Q}}_g[\ln T^{M_n}];
\end{equation*}
in other words, $T^{M_n}$ is a GROW $\E$-statistic. A natural question is whether this same reasoning can be reproduced for infinite groups. 
If the Haar measure $\rho$ could always be chosen to be a probability measure, we could replace $\mathbf{\Pi}_{\mathrm{U}(G)}$ by $\rho$ everywhere in the reasoning above and conclude that $T^{M_n}$ is GROW in general. 
However, $\rho$ is finite if and only if $G$ is compact \citep[see e.g.][Proposition
3.3.5]{reiter_classical_2000}. 
This is a severe limitation; it would not even cover our guiding example, the t-test, because the group $(\reals^+, \ \cdot \ )$ is not compact (see Example~\ref{ex:t-test}). 
The main technical contribution of this article is the extension of the above optimality result to
amenable groups (see Section~\ref{sec:intro-group-invariance}).
Setting technical details aside, the core of the proof of Theorem~\ref{theo:main_theorem} is replacing the Haar measure above by a sequence of almost-right-invariant probability measures and showing that the $\KL$ converges to its infimum. 
Our arguments require the following additional assumptions. 

\begin{assumption}\label{ass:summary_topological}
  Let $G$ be a topological group acting on a topological space ${\cal X}^n$, both equipped with their Borel $\sigma$-algebra. The group $G$, the observation space $\mathcal{X}^n$, and the probabilistic models under consideration satisfy the following three properties:
  \begin{enumerate}
  \item\label{item:ass_top_spaces} As topological spaces, $G$ and $\mathcal{X}^n$ are Polish---separable, completely metrizable and locally compact. 
  \item\label{item:ass_action} The action of $G$ on ${\cal X}^n$ is free,
    continuous and proper.
  \item\label{item:ass_dominated} The models $\{\mathbf{P}_{g}\}_{g\in G}$ and
    $\{\mathbf{Q}_{g}\}_{g\in G}$ are invariant and have densities with respect
    to a common measure $\mu$ on ${\cal X}^n$ that is relatively left invariant
    with some multiplier $\chi$---$\mu\bracks{gB} = \chi(g)\mu\bracks{B}$ for any measurable set $B\subseteq {\cal X}^n$ and $g\in G$. All densities have a single common support. 
  \end{enumerate}
\end{assumption}

Assumption~\ref{ass:summary_topological} holds in most cases of interest for the purpose of parametric inference; some examples where it holds are given in Section~\ref{sec:test-mult-norm}.
The topological assumptions on $G$ and ${\cal X}$ have two purposes.
The first is to ensure that Wijsman's representation theorem \eqref{eq:papi_wijsman} holds.
Though \eqref{eq:papi_wijsman} requires slightly weaker assumptions than those presented here, see Section~\ref{sec:intro-group-invariance}, the strengthened conditions are needed for the second purpose: to ensure that the observation space ${\cal X}^n$ can be put in bijective and
bimeasurable\footnote{We call an invertible map bimeasurable if both the map and
  its inverse are measurable.} correspondence with a subset of
$G\times {\cal X}^n/G$, where the group $G$ acts naturally by multiplication on
the first component~\citep{bondar_borel_1976}. 
This will be used extensively in the proofs given in Section~~\ref{sec:proofs}.
With these assumptions, everything is in place to state the main results of this article.

\begin{theorem}
  \label{theo:main_theorem}
  Let $M_n = m_n(X^n)$ be a maximally invariant statistic under
  the action of the group $G$ on ${\cal X}^n$. 
  Assume that $G$ is amenable, that Assumption~\ref{ass:summary_topological} holds, and that there
  is $\varepsilon>0$ such that
  \begin{equation}\label{eq:kl_moment_assumption}
    \mathbf{E}^{\mathbf{Q}}_1\sqbrack{\abs{\ln\frac{q_{1}(X^n)}{p_{1}(X^n)}}^{1 + \varepsilon}},
    \mathbf{E}^{\mathbf{Q}^{M_n}}\sqbrack{\abs{\ln\frac{q^{M_n}(M_n)}{p^{M_n}(M_n)}}^{1 + \varepsilon}}
    <\infty,
  \end{equation}
  where the subindex $1$ refers to the unit element of $G$. Then 
  \begin{equation*}
    \inf_{\mathbf{\Pi}_0, \mathbf{\Pi}_1}
    \KL(\mathbf{\Pi}_1^g\mathbf{Q}_{g}, \mathbf{\Pi}_0^g\mathbf{P}_{g})
    =
    \KL(\mathbf{Q}^{M_n}, \mathbf{P}^{M_n}),
  \end{equation*}
  where the infimum is over all pairs $(\mathbf{\Pi}_0, \mathbf{\Pi}_1)$
  of probability distributions on $G$.
\end{theorem}

\begin{corollary}\label{cor:main_corollary}
  Under the assumptions of Theorem~\ref{theo:main_theorem}, a GROW
  $\E$-statistic for testing $\mathcal{H}_1$ against $\mathcal{H}_0$ as in \eqref{eq:main_problem_group_parametrization} is
  given by the likelihood ratio of any maximally invariant statistic $M_n=m_n(X^n)$, i.e.
  \begin{equation*}
    T^{M_n} = \frac{q^{M_n}(m_n(X^n))}{p^{M_n}(m_n(X^n))}.
  \end{equation*}
\end{corollary}

Corollary~\ref{cor:main_corollary} follows from the combination of Theorem~\ref{theo:main_theorem} with Theorem~\ref{theo:ghk_theo1}.
The results are stated in terms of the likelihood ratio of any maximal invariant for the original data. However, as mentioned briefly in Section~\ref{sec:intro-group-invariance} and in detail in Appendix~\ref{app:sufficiency}, one can use instead any maximal invariant for a sufficient statistic of the original data, rather than for the data itself.
The resulting likelihood ratio is identical and the optimality results therefore remain valid.
Next, we show that in the group-invariant setting, any statistic that is
GROW is also relatively GROW, meaning that any \E-statistic that maximizes
\eqref{eq:grow_objective} also maximizes \eqref{eq:relative_grow_objective}.
This is not true in general; the result relies crucially on the invariance of the models. 
For example, for contingency tables, the two $\E$-statistics are vastly different \citep{turner_two-sample_2023}. 

\begin{theorem}\label{thm:relgrow_is_grow}
  Suppose that Part \ref{item:ass_dominated} of Assumption~\ref{ass:summary_topological} is satisfied 
  and that, for each $g \in G$, there exists $h\in G$ such that
  $\KL(\mathbf{Q}_{g}, \mathbf{P}_{h})$ is finite.
  Then the map defined by
  \begin{equation*}
    g\mapsto \sup_{T_n \text{ \E-stat.}}\mathbf{E}^{\mathbf{Q}}_{g}[\ln T_n]\label{eq:1}
  \end{equation*}
  is constant. 
  Consequently, any maximizer of \eqref{eq:grow_objective}
  also maximizes \eqref{eq:relative_grow_objective}, that is, an $\E$-statistic is
   GROW if and only if it is relatively GROW for the hypothesis testing
  problem \eqref{eq:main_problem_group_parametrization}.
\end{theorem}

\begin{corollary}\label{cor:also_regrow}
  $T^{M_n}$ from Corollary~\ref{cor:main_corollary} is not only GROW, it is also
  relatively GROW.
\end{corollary}

\begin{example}[continues=ex:t-test]
  It is known that he group $G = ({\mathbb R}^+,\cdot)$ of the
  t-test is amenable---the sequence of probability distributions $(\mathrm{Uniform}([-n,n]))_{n\in\nats}$ is almost right invariant.
  It is readily verified that Assumption~\ref{ass:summary_topological} and condition~\eqref{eq:kl_moment_assumption} are also satisfied. 
  Hence, Corollary~\ref{cor:main_corollary} implies that the likelihood ratio for the t-statistic, given in \eqref{eq:ttesthaar}, is a GROW $\E$-statistic. 
  Moreover, it follows from Corollary~\ref{cor:also_regrow} that it is also relatively GROW.
\end{example}

\subsection{Anytime-validity}\label{sec:anyt-valid-test}
As discussed in Section~\ref{sec:intro-e-statistics}, any \E-statistic can be used in the context of optional continuation with fixed sample sizes, but not all \E-statistics are suitable for optional stopping and optional continuation with data-dependent sample sizes. A sufficient condition that allows us to engage in these two additional uses is that the sequence of \E-statistics is a nonnegative martingale.
We now show that this is the case for the sequence $(T^{M_n})_{n\in \mathbb{N}}$.
\begin{proposition}\label{prop:grow_also_av}
  If $(M_n)_{n\in\nats}$ is a sequence of maximally invariant statistics
  $M_n = m_n(X^n)$ for the action of $G$ on ${\cal X}^n$, then the process
  $(T^{M_n})_{n\in\nats}$ is a nonnegative martingale with respect to the filtration $(\sigma(M_1,\dots,M_n))_{n\in \mathbb{N}}$ under any of the elements of the
  null hypothesis.
\end{proposition}
In particular, Proposition~\ref{prop:grow_also_av} implies that under every stopping time $\tau$ defined relative to the filtration induced by $(M_n)_{n\in\nats}$, $T^{M_{\tau}}$ is itself an $\E$-statistic; see Appendix~\ref{app:av} for the (standard) proof.
There is an interesting subtlety here however: if $\tau'$ is a stopping time relative to the filtration induced by $(X_n)_{n \in {\mathbb N}}$ but not relative to the coarser filtration induced by $(M_n)_{n\in\nats}$, then $T^{M_{\tau'}}$ is not necessarily an $\E$-statistic anymore.
Thus, with such $T^{M_{\tau'}}$, we cannot engage in optional continuation. 
This is generally not a problem, since most stopping times encountered in practice are stopping times relative to the filtration induced by $(M_n)_{n\in \nats}$.
This includes the aggressive stopping time `stop at the smallest $n$ at which $T^{M_n} \geq 1/\alpha$'.
However, in Appendix~\ref{app:filtration_counterexample} we give an explicit example of a stopping time $\tau'$ relative to the filtration induced by $(X_n)_{n\in\nats}$ in the t-test such that $T^{M_{\tau'}}$ is not an $\E$-statistic. 

\subsection{GROW for composite invariant hypotheses}
\label{sec:comp-invar-hypoth}
Until now we have considered null and alternative hypotheses that become simple
when viewed through the lens of the maximally invariant statistic. As we saw, in
the t-test this corresponds to testing simple hypotheses about the effect size
$\delta$. 
In this section we consider hypotheses that are composite in the maximally invariant parameter. 
We also consider problems in which a fixed prior is placed on the maximally invariant parameter $\delta$. This implements the method of mixtures, a standard method to combine test martingales \citep{Wald45,darling1968some}, which was already used in the context of the anytime-valid t-test \citep{lai1976confidence}.

Suppose that the initial hypotheses are not defined by a single value of the maximally invariant parameter $\delta=\delta(\theta)$, as in~\eqref{eq:target_hypothesis_problem}, but are instead given by
\begin{equation}\label{eq:orig_target_composite_problem}
      {\cal H}_0: \delta(\theta) =  \delta,\ \ \delta\in \Delta_0
  \text{ \ vs. \ } {\cal
    H}_1:\delta(\theta) = \delta,\ \ \delta\in \Delta_1,
\end{equation}
where $\Delta_0$ and $\Delta_1$ are two sets of possible values of $\delta = \delta(\theta)$. 
In Section~\ref{sec:intro-group-invariance}, we reparametrized $\{\mathbf{P}_\theta\}_{\theta\in\Theta: \delta(\theta)=\delta_0}$ and $\{\mathbf{P}_\theta\}_{\theta\in\Theta: \delta(\theta)=\delta_1}$ in terms of $G$, and denoted the resulting models as $\{\mathbf{P}_g\}_{g\in G}$ and $\{\mathbf{Q}_g\}_{g\in G}$ respectively.
Instead of only considering $\delta_0$ and $\delta_1$, we can do the same for all $\delta\in \Delta_0$ and $\delta\in \Delta_1$. 
We denote the resulting models as $\{\mathbf{P}_{g,\delta}\}_{g\in G, \delta\in \Delta_0}$ and $\{\mathbf{Q}_{g,\delta}\}_{g\in G,\delta\in \Delta_1}$.
As an example, $\mathbf{P}_{g,\delta_0}$ and $\mathbf{Q}_{g,\delta_1}$ correspond to what were previously simply $\mathbf{P}_g$ and $\mathbf{Q}_{g}$. 
The problem~\eqref{eq:orig_target_composite_problem} may now be rewritten as
\begin{equation}\label{eq:target_composite_problem}
  {\cal H}_0: X^n\sim \mathbf{P}_{g, \delta},\ \  \delta\in \Delta_0,\  g\in G
  \text{ \ vs. \ }
  {\cal H}_1: X^n\sim \mathbf{Q}_{g, \delta},\ \ \delta\in \Delta_1, \ g\in G.
\end{equation}
Since the distribution of a maximally invariant function of the data $M_n = m_n(X^n)$ depends on the
parameter $\delta$,  these hypotheses are not simple when data are reduced through invariance. 
The main objective of this section is to show that, when searching
for a GROW $\E$-statistic for \eqref{eq:target_composite_problem}, it is enough
to do so for the invariance-reduced problem
\begin{equation}\label{eq:inv_red_composite_problem}
  {\cal H}_0: M_n\sim \mathbf{P}_{\delta}^{M_n},\ \  \delta\in \Delta_0
  \text{ \ vs. \ }
  {\cal H}_1: M_n\sim \mathbf{Q}_{\delta}^{M_n},\ \ \delta\in \Delta_1.
\end{equation}
We follow the same steps that we followed in
Section~\ref{sec:simp-invar-hypoth}, and begin by showing that if there exists a
minimizer for the $\KL$ minimization problem associated to
(\ref{eq:inv_red_composite_problem}), then it has the same value as that
associated to (\ref{eq:target_composite_problem}).
\begin{proposition}\label{prop:composite_invariant_grow}
  Assume that there exists a pair of probability distributions
  $\mathbf{\Pi}^\star_0,\mathbf{\Pi}_1^\star$ on $\Delta_0$ and $\Delta_1$ that
  satisfy
  \begin{equation}\label{eq:superpair}
    \KL(\mathbf{\Pi}_1^{\star \delta}\mathbf{Q}_{\delta}^{M_n},
    \mathbf{\Pi}_0^{\star \delta}\mathbf{P}_{\delta}^{M_n})
    =
    \min_{\mathbf{\Pi}_0, \mathbf{\Pi}_1}
    \KL(\mathbf{\Pi}_1^\delta\mathbf{Q}_{\delta}^{M_n}, \mathbf{\Pi}_0^\delta\mathbf{P}_{\delta}^{M_n}).
  \end{equation}
  For each $g\in G$, define the probability distributions
  $\mathbf{P}^{\star }_g = \mathbf{\Pi}_0^{\star \delta}\mathbf{P}_{g, \delta}$
  and $\mathbf{Q}^\star_g = \mathbf{\Pi}_1^{\star \delta}\mathbf{Q}_{g,\delta}$ on
  ${\cal X}^n$. If the models $\{\mathbf{P}^\star_g\}_{g\in G}$ and
  $\{\mathbf{Q}^\star_g\}_{g\in G}$ satisfy the assumptions of
  Theorem~\ref{theo:main_theorem}, then
  \begin{equation*}
    \inf_{\mathbf{\Pi}_0, \mathbf{\Pi}_1}
    \KL(\mathbf{\Pi}_1^{g,\delta}\mathbf{Q}_{g,\delta}, \mathbf{\Pi}_0^{g,\delta}\mathbf{P}_{g,\delta})
    =
    \min_{\mathbf{\Pi}_0, \mathbf{\Pi}_1}
    \KL(\mathbf{\Pi}_1^\delta\mathbf{Q}_{\delta}^{M_n}, \mathbf{\Pi}_1^\delta\mathbf{P}_{\delta}^{M_n}).
  \end{equation*}
\end{proposition}
From this proposition, using Theorem~\ref{theo:ghk_theo1} and the steps
used for Corollaries \ref{cor:main_corollary} and \ref{cor:also_regrow}, we can
conclude that the ratio of the Bayes marginals for the invariance-reduced data
$M_n$ using the optimal priors $\mathbf{\Pi}^{\star}_0$ and
$\mathbf{\Pi}^\star_1$ is both a GROW and a relatively GROW $\E$-statistic for
(\ref{eq:target_composite_problem}). We now state the corollary and apply it to
to our running example, the t-test.
\begin{corollary}\label{cor:composite}
  Under the assumptions of Proposition~\ref{prop:composite_invariant_grow}, the
  statistic given by
  \begin{equation*}
    T^\star
    =
    \frac{
      \int q_\delta^{M_n}(m_n(X^n)) \rmd \mathbf{\Pi}^{\star}_1(\delta)
    }{
      \int p_\delta^{M_n}(m_n(X^n)) \rmd \mathbf{\Pi}^{\star}_0(\delta)
    }
  \end{equation*}
  is a (both absolute and relative) GROW  $\E$-statistic for
  (\ref{eq:target_composite_problem}).
\end{corollary}
%PETER 10/8 This Corollary is too abstract on its own! I think we should add the t-test again as an example (even though an instance is given later for Hotelling - but that is super-abstract as well); because the likelihood ratio is a monotone function of the t-statistic, it must follow that the optimal priors put mass 1 on the point that is closest to the alternative, both in \Delta_0 and \Delta_1. This we must still add before submitting to the annals!
%TYRON: It is odd to me that this is stated in terms of the t-statistic, because we have not proven that it is the same also for composite hypotheses (or is this a simple argument that I am missing?)
\begin{example}[continues=ex:t-test]
  Suppose, in the t-test setting, that we are interested in testing
  \begin{equation*}%\label{eq:t-tet-but-composite}
    {\cal H}_0: \delta\in (-\infty, \delta_0] \text{ \ \ vs. \ \ }
    {\cal H}_1: \delta\in [\delta_1, \infty)
  \end{equation*}
  for some $\delta_0 < \delta_1$, where, recall, $\delta = \mu / \sigma$ is the
  maximally invariant parameter. 
  Corollary~\ref{cor:composite} shows that no
  loss is incurred if we only look among $\E$-statistics that are a function of
  the maximally invariant function $M_n$, the t-statistic. Since the density
  of the t-statistic is monotone in $\delta$, we can use Proposition 3 of \citetalias[Section 3.1.]{grunwald_safe_2023} to infer that the minimum
  in (\ref{eq:superpair}) is achieved by the probability distributions
  $\mathbf{\Pi}^\star_0$ and $\mathbf{\Pi}^\star_1$ that put all of their mass
  on $\delta_0$ and $\delta_1$, respectively. Corollary~\ref{cor:composite}
  yields that $T^*_n = p^{M_n}_{\delta_1} / p^{M_n}_{\delta_0}$ is GROW among
  all possible $\E$-statistics of the original data (not only the
  scale-invariant ones). This result can be extended to other families with this type
  of monotonicity property.
\end{example}
%PETER 10/8, added
Another approach to deal with the unknown parameter values is to employ proper prior distributions, as is standard practice both within Bayesian statistics and with $\E$-statistics.
That is, we may want to use specific priors $\tilde{\mathbf{\Pi}}_0$ and $\tilde{\mathbf{\Pi}}_1$ on $\Delta_0$ and
$\Delta_1$ respectively. If we define for each $g$ the probability distributions
$\tilde{\mathbf{P}}_g = \tilde{\mathbf{\Pi}}_0^\delta\mathbf{P}_{g,\delta}$ and
$\tilde{\mathbf{Q}}_g = \tilde{\mathbf{\Pi}}_1^\delta\mathbf{Q}_{g,\delta}$, and
the resulting models $\{\tilde{\mathbf{P}}_g\}_{g\in G}$ and
$\{\tilde{\mathbf{Q}}_g\}_{g\in G}$ also satisfy the conditions of
Corollary~\ref{cor:main_corollary}, the proof of
Proposition~\ref{prop:composite_invariant_grow} also provides the following
corollary.
\begin{corollary}\label{cor:delta_fixed_prior}
  Let $\tilde{\mathbf{\Pi}}_0$ and $\tilde{\mathbf{\Pi}}_1$ be two probability
  distributions on $\Delta_0$ and $\Delta_1$, respectively. Let
  $\{\tilde{\mathbf{P}}_g\}_{g\in G}$ and $\{\tilde{\mathbf{Q}}_g\}_{g\in G}$ be
  two probability models defined by
  $\tilde{\mathbf{P}}_g = \tilde{\mathbf{\Pi}}_0^\delta\mathbf{P}_{g,\delta}$
  and
  $\tilde{\mathbf{Q}}_g = \tilde{\mathbf{\Pi}}_1^\delta\mathbf{Q}_{g,\delta}$.
  If $\{\tilde{\mathbf{P}}_g\}_{g\in G}$ and $\{\tilde{\mathbf{Q}}_g\}_{g\in G}$
  satisfy the conditions of Corollary~\ref{cor:main_corollary}, then the \E-statistic
  %(or more precisely, the conditions of Theorem~\ref{theo:main_theorem} with $\tilde{\bf P}_g$ in the role of ${\bf P}_g$ and $\tilde{\bf Q}_g$ in the role of ${\bf Q}_g$), then
  \begin{equation}\label{eq:finalbits}
    \tilde{T}_n
    =
    \frac{
      \int q_\delta(m_n(X^n)) \rmd\tilde{\mathbf{\Pi}}_1(\delta)
    }{
      \int p_\delta(m_n(X^n)) \rmd\tilde{\mathbf{\Pi}}_0(\delta)
    }
  \end{equation}
  is both GROW and relatively GROW for testing
  $\{\tilde{\mathbf{P}}_g\}_{g\in G}$ against
  $\{\tilde{\mathbf{Q}}_g\}_{g\in G}$.
\end{corollary}
\begin{example}[continues=ex:t-test]
  \cite{Jeffreys61} proposed a Bayesian version of the t-test based on the Bayes factor (\ref{eq:ttesthaar}) with $\delta_0$ to $0$ and a Cauchy prior centered at $0$ on $\delta_1$. Popularized as the {\em Bayesian t-test\/} \citep{rouder-2009-bayes}, it is an instance of (\ref{eq:finalbits}) with $\tilde{\mathbf{\Pi}}_1$ set to aforementioned Cauchy prior and $\tilde{\mathbf{\Pi}}_0$ putting mass 1 on $\delta_0 = 0$. It is itself an $\E$-statistic \citepalias{grunwald_safe_2023}, but condition (\ref{eq:kl_moment_assumption}) of Theorem~\ref{theo:main_theorem} does not hold because the Cauchy distribution does not have any moments. Thus, we cannot verify whether (\ref{eq:finalbits}) has the relative GROW property. However, as soon as we replace the Cauchy prior by any prior centered at $0$ for which, for some $\varepsilon > 0$, the $(2+\varepsilon)$-th moment exists (such as e.g. a normal distribution centered at 0, as has also been proposed for this problem), we can use Lemma~\ref{lem:submit} in the next section (applied with $d=1$) to infer that assumption \eqref{eq:kl_moment_assumption} holds. 
  Finally, Proposition~\ref{cor:delta_fixed_prior} can be applied to conclude that the corresponding Bayes factor is then relatively GROW.
 \end{example}

\section{Testing multivariate normal distributions under group invariance}
\label{sec:test-mult-norm}
We show how the theory developed in the previous sections can be applied to
hypothesis testing under normality assumptions. 
The latter is particularly suited for the group-invariant setting, because the family of normal distributions carries a natural invariance under scale-location transformations, as we have already seen in Example~\ref{ex:t-test}.
Different subsets of scale-location transformations correspond to different parameters of interest.
We develop two examples in detail. 
The first is an alternative to Hotelling's $T^2$ for testing whether the (multivariate) mean of the distribution is identically zero.
The corresponding group is that of lower triangular matrices with positive entries on the diagonal.
This test is in direct relation with the step-down procedure of
\citet{roy_tests_1958}\footnote{Even though not explicitly in group-theoretic
  terms, the test of \citet{roy_tests_1958} test is based on a different
  maximally invariant function of the data. 
  The fact that the test statistic of
  \citet{roy_tests_1958} is maximally invariant is
  shown by \citet{subbaiah_comparison_1978}} \citep[see
also][]{subbaiah_comparison_1978}. 
The second example that we consider is, in the setting of linear regression, a test for whether or not a specific regression coefficient is identically zero. 
In this case, the group is a subset of the affine linear group.

% The parameter space $\Theta$ consists of all pairs $(\mu,\Sigma)$, where
% $\mu\in\reals^d$, and $\Sigma$ is a positive definite $d\times d$ matrix. For
% the space where data takes values, we may restrict all our invariance
% considerations to a sufficient statistic by the discussion in
% Section~\ref{sec:invar-suff}. A sufficient statistic is the pair consisting of
% the unbiased estimators for the mean and the covariance matrix. Hence, for
% sample size $n$, there is no loss in considering data to be of the form
% $X_n = (\bar{Y}_n, \bar{V}_n)$, where $\bar{Y}_n$ and $\bar{V}_n$ are
% independent from each other, $n\bar{Y}_n \sim N\paren{n\mu, \Sigma}$, and
% $(n-1) \bar{V}_n\sim W(n - 1, \Sigma)$, a Wishart distribution with $n-1$
% degrees of freedom and shape parameter $\Sigma$. From this point of view, the
% observation space ${\cal X}$ and the parameter space $\Theta$ coincide. The
% affine linear group acts on the observation space ${\cal Y}$ by location-scale
% transformations, $((v, A), (Y, V)) \mapsto (v + AY, AVA')$, and identically on
% the parameter space $\Theta$.

\subsection{The lower triangular group}\label{sec:lower_triang}
Consider data $X^n = (X_1, \dots, X_n)$ where $X_i\in {\cal X} = \reals^d$. We
assume each $X_i$ to have a Gaussian distribution $N(\mu,\Sigma)$ with unknown
mean $\mu\in\reals^d$ and covariance matrix $\Sigma$. We consider a test for
whether the mean $\mu$ of the distribution is zero. 
%Before stating explicitly
%our hypothesis testing problem, we first reparametrize the Gaussian model using
%Cholesky's decomposition. Indeed, 
To formalize the test, recall that 
the Cholesky decomposition of a positive definite matrix $\Sigma$ is  $\Sigma = \Lambda \Lambda'$ for a unique $\Lambda \in \LTP(d)$. 
Here, $\LTP(d)$ denotes the group of lower triangular matrices with positive entries on the diagonal, which is amenable.
We can therefore parametrize the Gaussians in terms of $(\mu, \Lambda)$, taking the parameter space to be $\Theta = \reals^d \times \LTP(d)$. 
% In this parametrization, the likelihood of
% the original data $X^n = (X_1, \dots, X_n)$ takes the form 
% \begin{align*}
%   p_{\Lambda, \delta}^{X^n} (X^n)
%   =
%   \frac{1}{(2\pi)^{n}(\det
%   \Lambda)^{n}}
%   \exp\paren{ -\frac{1}{2} \sum_{i=1}^n \norm{ \Lambda^{-1}
%   X_i  - \delta }^2}.
% \end{align*}
In this parametrization, consider the following hypothesis testing problem, which generalizes the t-test (Example~\ref{ex:t-test}) to dimensions $d\geq 1$:
\begin{equation}\label{eq:lower_triang_test}
  {\cal H}_0: \Lambda^{-1}\mu = \delta_0
  \text{ \ vs. \ }
  {\cal H}_1: \Lambda^{-1}\mu = \delta_1.
\end{equation}
A test for whether $\mu$ is zero can be obtained by setting $\delta_0 = 0$. 
The group $\LTP(d)$ acts freely and continuously on $\mathcal{X}^n$ through component-wise matrix multiplication, i.e. $(L,X^n) \mapsto (LX_1,\dots, LX_n)$ for any $L\in \LTP(d)$.
This action is continuous and free, and can be shown to be proper on the restriction of $\mathcal{X}^n$ to matrices of rank $d$ if $n\geq d+1$. 
If $X_i\sim N(\mu,\Lambda)$, then $LX_i\sim N(L\mu,L\Lambda)$, so that $\LTP(d)$ acts on $\Theta$ by % is amenable and 
$  (L,  (\mu, \Lambda)) \mapsto  (L\mu, L \Lambda)$
for each $(\mu,\Lambda)\in \Theta$ and $L\in \LTP(d)$.
A maximally invariant parameter under this action is $\delta(\mu,\Lambda) = \Lambda^{-1}\mu$, 
so that~\eqref{eq:lower_triang_test} is indeed a test of the form described in Section~\ref{sec:intro-group-invariance}.
Furthermore, seen as a subset of $\reals^{d\times n}$, the restriction of the Lebesgue measure to $\mathcal{X}^n$ is relatively left-invariant with multiplier $\chi(L)=|\det(L)|^n$.
It follows that Assumption~\ref{ass:summary_topological} holds and therefore, the likelihood ratio of any maximally invariant statistic is  GROW by Corollary~\ref{cor:main_corollary}.

By the results of \cite{hall_relationship_1965}, recapped in Appendix~\ref{app:sufficiency}, this likelihood ratio must coincide  with that of an invariantly sufficient statistic for $\delta$.
We now proceed to compute one such statistic. 
Recall that the pair $S_n=s_n(X^n)=(\bar{X}_n,\bar{V}_n)$, consisting of the unbiased estimators $\bar{X}_n$ and $\bar{V}_n$ for the mean and covariance matrix respectively, is a sufficient statistic for $(\mu,\Sigma)$. 
Analogous to the technique we used for the parameter space, we can perform the Cholesky decomposition $\bar{V}_n=L_nL_n'$.
The statistic $M_{{\cal S}, n} = m_{{\cal S},n}(S_n) = \sqrt{n/(n-1)} L_n^{-1}\bar{X}_n$
is maximally invariant under the action of $\LTP(d)$ on $S_n$; in other words, $M_{\mathcal{S},n}$ is invariantly sufficient for $\delta$.
Hence, the GROW $\E$-statistic can be written as
$T^{M_{{\cal S},n}} = q^{M_{{\cal S},n}} / p^{M_{{\cal S},n}}$. 
Since it was used in Example~\ref{ex:t-test} (underneath Corollary~\ref{cor:delta_fixed_prior}), we give an explicit expression for the likelihood ratio
$T^{M_{{\cal S},n}}$ when $\delta_0 = 0$, from which values for other $\delta_0$ can be computed. It is based on a more general computation in Appendix~\ref{app:furtherproofs}.
%for general 
%From this expression, the likelihood ratio for other values of $\delta_0$.
\begin{lemma}\label{lem:submit}
  For the maximally invariant statistic
  $M_{{\cal S},n} = \sqrt{\tfrac{n}{n-1}}L_n^{-1}\bar{X}_n$, we have
  \begin{equation}\label{eq:lr_lower_triang}
    \frac{
      q^{M_{{\cal S},n}}(m_{{\cal S},n}(S_n))
    }{
      p^{M_{{\cal S},n}}(m_{{\cal S},n}(S_n))
    }
    =
    \rme^{-\frac{n}{2}\norm{\delta_1}^2}
    \int
      \rme^{n\inner{\delta_1}{T A^{-1}_n M_{{\cal S},n}}}
    \rmd\mathbf{P}_{n,I}(T),
  \end{equation}
  where $A_n$ is the lower triangular matrix resulting from the Cholesky
  decomposition $I + M_{{\cal S},n}M_{{\cal S},n}' = A_nA_n'$, and
  $\mathbf{P}_{n,I}^T$ is the distribution according to which
  $nTT'\sim W(n, I)$, a Wishart distribution.
\end{lemma}
\begin{proof}
  This follows from Proposition~\ref{prop:triangular_maxinv_distr} in Appendix~\ref{app:furtherproofs} with
  $\gamma = \sqrt{n}\delta_1$, $X = \sqrt{n}\bar{X}_n$, $m = n - 1$, and  $S = \bar{V}_n$.
\end{proof}

\subsection{Linear regression}\label{sec:regression}
% data
Consider the problem of testing whether one of the coefficients of a linear
regression is zero under Gaussian error assumptions. 
Assume that the observations are of the form $(X_1, Y_1, Z_1), \dots, (X_n, Y_n, Z_n)$, where
$X_i,Y_i \in \reals$ and $Z_i\in \reals^{d}$ for each $i$.
We consider the the linear model given by
\begin{equation*}
  Y_i= \gamma X_i + \beta' Z_i + \sigma\varepsilon_i,
\end{equation*}
where $\gamma\in\reals$, $\beta\in \reals^d$ and $\sigma\in \reals^+$ are the
parameters, and $\varepsilon_1, \dots,\varepsilon_n$ are i.i.d.\ errors with
standard Gaussian distribution $N(0,1)$.
% problem
We are interested in testing
\begin{equation}\label{eq:lin_reg_test}
  {\cal H}_0: \gamma / \sigma = \delta_0
  \text{ \ vs. \ }
  {\cal H}_1: \gamma / \sigma = \delta_1.
\end{equation}
A test for whether $\gamma =0$ is readily obtained by taking $\delta_0 = 0$.
% action on parameter and data space
This problem is invariant under the action of the group $G=\mathbb{R}^+\times \mathbb{R}^d$ given by $((c, v), (X, Y, Z))\mapsto (X, cY + v'Z, Z)$~\citep{Kariya1980a,eaton_group_1989}. 
The corresponding action of $G$ on the parameter space is given by
$((c,v), (\gamma, \beta, \sigma))\mapsto (c\gamma, c\beta + v, c\sigma)$.
A maximally invariant parameter is $\delta(\gamma,\beta,\sigma) = \gamma / \sigma$, so that the problem in~\eqref{eq:lin_reg_test} is of the form described in Section~\ref{sec:intro-group-invariance}. 
Furthermore, it can be shown that the action of $G$ on $\mathcal{X}$ is continuous and proper, and that $G$ is amenable.
Since the Lebesgue measure is again relatively left invariant, it follows that Assumption~\ref{ass:summary_topological} holds.
All that remains is to find a maximally invariant function of the data. 
To this end,
define the vectors $Y^n = (Y_1, \dots, Y_n)'$ and
$X^n = (X_1, \dots, X_n)'$, and the $n\times d$ matrix
$\boldsymbol{Z}_n=[Z_1,\dots,Z_n]'$ whose rows are the vectors $Z_1, \dots,Z_n$. Assume that
$\boldsymbol{Z}_n$ has full rank. 
A maximally invariant function of the data is given by
$M_n=\paren{\frac{\boldsymbol{A}_n^{\prime}Y^n}{\|\boldsymbol{A}_n^{\prime}Y^n\|},
  X^n, \boldsymbol{Z}_n}$, where $\boldsymbol{A}_n$ is an $(n - d)\times n$ matrix whose
columns form an orthonormal basis for the orthogonal complement of the column
space of $\boldsymbol{Z}_n$~\citep{Kariya1980a,BhowmikKing2007}. 
%It follows that $A^{n\prime}A^n=I^{n-d}$ and $A^nA^{n\prime}=I^n-Z^n(Z^{n\prime}Z^n)^{-1}Z^{n\prime}$, where $I^n$ is the
%$n\times n$ identity matrix. 
In order to compute the likelihood ratio for $M_n$, we assume that the mechanism that generates $X^n$ and $\boldsymbol{Z}_n$ is the same under both
hypotheses, so that we only need to consider the distribution of $\boldsymbol{U}_n = \frac{\boldsymbol{A}_n^{\prime}Y^n}{\|\boldsymbol{A}_n^{\prime}Y^n\|}$
conditionally on $X^n$ and $\boldsymbol{Z}_n$. \citet{BhowmikKing2007} show that
for arbitrary effect size $\delta$, the density of this distribution is given by
\begin{align*}
  p^{\boldsymbol{U}_n}_{\delta}(u|X^n, \boldsymbol{Z}_n)
  =
  \frac12
  \Gamma
  \left(\frac k2\right)
  \pi^{-\frac k2}e^{c( \delta )}&\left[_1F_1\left(\frac k2, \frac12, \frac{a^2(u,\delta)}{2} \right)\right.
  \\&+\sqrt 2 a(u, \delta)\left. \frac{\Gamma((1+k)/2)}{\Gamma(k/2)}{_1F_1}\left(\frac {1+k}2, \frac32, \frac{a^2(u,\delta)}{2} \right)\right],
\end{align*}
where $k=n-d$, $u$ is a unit vector in $\reals^k$,
$a\left(u , \delta\right)=\delta X^{n\prime} \boldsymbol{A}_nu, c\left(\delta \right)=
-\frac12 \delta^2 X^{n\prime } \boldsymbol{A}_n\boldsymbol{A}_n^{\prime} X^n$, and $_1F_1$ is
the confluent hypergeometric function. This can be used to compute the likelihood ratio for $M_n$, which is
the relatively GROW \E-statistic for testing
\eqref{eq:lin_reg_test}. 
In fact, \citeauthor{BhowmikKing2007} compute in more generality the density of the maximally invariant statistic when $X$ is allowed to have a non-linear effect on $Y$.
This does not impact the group invariance structure of the model, so that our results
can also be used in this semilinear setting if the hypotheses are adjusted accordingly.

\section{Discussion and Future Work}
\label{sec:discussion}
In this concluding section we bring up an issue that deserves further discussion and may inspire future work. 
We also use this issue to highlight the differences between our work and related work in a Bayesian context.

\subsection{Amenability is not always necessary}\label{sec:ameneyourself}
We have shown that, if a hypothesis testing problem is invariant under a group
$G$ and our assumptions are satisfied, then amenability of $G$ is a sufficient
condition for the likelihood ratio of the maximal invariant to be GROW. A
natural question is therefore whether amenability is also a necessary condition for the
latter to hold.
This is not only of theoretical relevance: some groups that are important for statistical practice are not amenable. 
For instance, the general linear group $\GL(d)$, which is the relevant group in Hotelling's test, is
nonamenable. 
The setup of Hotelling's test is similar to that in Section~\ref{sec:lower_triang},
except that the hypotheses are given by
\begin{equation}\label{eq:hotelling-test-problem}
  {\cal H}_0: \|\Lambda^{-1}\mu\|^2 = 0
  \text{ \ vs. \ }
%  {\cal H}_1: \|\Lambda^{-1}\mu\|^2 = d.
%PETER 10/8 changed d to \gamma since d was used in 2 ways here
  {\cal H}_1: \|\Lambda^{-1}\mu\|^2 = \gamma.
\end{equation}
A maximally invariant statistic is the $T^2$-statistic
$n\bar{X}_n'\bar{V}_n^{-1}\bar{X}_n$, where, as in
Section~\ref{sec:lower_triang}, $\bar{X}_n$ and $\bar{V}_n$ are the unbiased
estimators of the mean and the covariance matrix, respectively. Notice that this
test is equivalent to (\ref{eq:lower_triang_test}) with the alternative expanded
to $\Delta=\{\delta: \|\delta\|^2=\gamma\}$, but that $T^2$ is not a maximal
invariant under the lower triangular group. However, \citet{Giri1963} have shown
that for $d=2$ and $n=3$, the likelihood ratio of the $T^2$-statistic can be
written as an integral over the likelihood ratio in~\eqref{eq:lr_lower_triang}
with a proper prior on $\delta \in \Delta$ as defined there. It follows as a
result of Proposition~\ref{prop:composite_invariant_grow} that the likelihood
ratio of the $T^2$-statistics is also GROW in the case that $d=2$ and $n=3$.
These results can be extended to the case that $d=2$ with arbitrary $n$ by the
work of~\citet{shalaevskii_minimax_1971}. An interesting question is whether amenability can be  replaced by a weaker
condition, and/or whether a counterexample to Theorem~\ref{theo:main_theorem}
for nonamenable groups can be given.

\subsection{Nonuniqueness issues with right Haar priors do not arise}
As the above example  illustrates, it is sometimes possible to represent the same
$\mathcal{H}_0$ and $\mathcal{H}_1$ via (at least) two different groups. 
As we explain in full detail in Appendix~\ref{app:herecomesthesun}, this is generally unproblematic: as soon as the assumptions of Theorem~\ref{theo:main_theorem} hold for at least one of the two groups, we can construct the GROW $\E$-statistic, and it is uniquely defined. Superficially, this may seem to contradict \citet{sun_objective_2007} who point
out that in some settings, the underlying group is not uniquely determined and then the  right Haar prior for the considered model $\mathcal{P}$ is not uniquely defined. Then, different choices of right Haar prior give different Bayesian posteriors---a fact that has sometimes been taken as a criticism of objective Bayesian approaches. Such nonuniqueness is avoided in our approach. The reason is, essentially, that 
whereas the GROW $\E$-statistic ${T}^*_n$ is a ratio between Bayes
marginals for different models ${\cal H}_0$ and ${\cal H}_1$ at the same sample size $n$,
the Bayes predictive distribution based on a single model $\mathcal{P}$ is
a ratio between Bayes marginals for the same $\mathcal{P}$ at different
sample sizes $n$ and $n-1$. The role of `same' and `different' being interchanged, it turns out that this Bayes predictive distribution {\em can\/}  depend on the group on which the right Haar prior for $\mathcal{P}$ is based. Since the Bayes predictive distribution can be rewritten as a marginal over the Bayes posterior, which is 
\citet{sun_objective_2007}'s quantity of interest, it is then not surprising that this Bayes posterior may also change if the underlying group is changed. Instead, one may quantify uncertainty by the 
{\em $\E$-posterior}, an $\E$-statistic-based measure of uncertainty recently put forward by \citet{Grunwald23}: if one replaces the standard Bayes posterior on  $\delta$  
%It is important that we look at the posterior on delta, not on the full model
by the $\E$-posterior based on the GROW $\E$-statistic $T^*_n$, the nonuniqueness issue disappears as well. 

\section{Proofs}\label{sec:proofs}
In this section, we give all the proofs that were omitted earlier. 
We first provide two remarks that will be useful throughout the proofs. 
\begin{remark}\label{rem:dominated_models_left_invariant}
   Without loss of generality, we may modify~\ref{item:ass_dominated} in Assumption~\ref{ass:summary_topological} as follows:
  \begin{enumerate}[1']
    \setcounter{enumi}{2}
  \item\label{item:modified_assumption} The models $\{\mathbf{P}_{g}\}_{n\in\nats}$ and
    $\{\mathbf{Q}_{g}\}_{n\in\nats}$ are invariant and have densities with
    respect to a common measure $\nu$ on ${\cal X}^n$ that is left invariant.
  \end{enumerate}
  The reason that there is no loss in generality is that from any relatively
  left-invariant measure $\mu$ with multiplier $\chi$, a left-invariant measure
  $\nu$ can be constructed. Indeed, \citet[][Chap. 7, §2 Proposition
  7]{bourbaki_integration_2004} shows that, under our assumptions, for any
  multiplier $\chi$ there exists a function $\varphi:{\cal X}^n\to \reals$ with
  the property that $\varphi(gx) = \chi(g)\varphi(x)$ for any $x\in {\cal X}$
  and $g\in G$. With this function at hand, one can define the measure
  $\rmd \nu (x) = \rmd \mu (x) / \varphi(x)$, which is left invariant. After
  multiplication by $\varphi$, probability densities with respect to $\mu$ are
  readily transformed into probability densities with respect to $\nu$.
  The invariance of the models implies that the densities of $\mathbf{P}_g$ and $\mathbf{Q}_g$ with respect to $\nu$ take the form $p_g(x^n) = p_1(g^{-1}x^n)$ and
  $q_{g}(x^n) = q_{1}(g^{-1}x^n)$ for any $x^n\in \mathcal{X}^n$, where $1$ denotes the unit element of the group $G$.
  It follows that for any $g,h\in G$ it holds that $p_g(x^n)=p_h(hg^{-1}x^n)$ for all $x^n\in \mathcal{X}^n$. A similar statement can be made for $q_g$.
\end{remark}
\begin{remark}
  So far, we have only considered the right Haar measure $\rho$ on $G$, however  on any locally compact group $G$ there also exists a left-invariant measure
  $\lambda$, called the left Haar measure. 
  It can be shown that $\lambda$ is relatively right invariant with a multiplier $\Delta$, that is, for any measurable $B\subseteq G$ and $g\in G$ it holds that
  $\lambda\{Bg\} = \Delta(g)\lambda\{B\}$ for any $g\in G$. 
  Moreover, a computation shows that the measure $\rho'$ defined by $\rho'\{B\} = \lambda\{B^{-1}\}$ for each measurable $B\subseteq G$, is right invariant; in other words, $\rho'$ is a
  right Haar measure. 
  We may therefore choose $\rho$ to be equal to $\rho'$ and in the following, we always refer to right and left Haar measures that are related to each other by that identity. 
  In our proofs we will use that for any integrable function $f$ defined on $G$, the identities
  $\int f(h) \rmd\rho(h) = \int f(h) / \Delta(h) \rmd \lambda(h)$ and
  $\int f(h^{-1}) \rmd\lambda(h) = \int\rmd f(h) \rmd\rho(h)$
  hold~\citep[see][Section 1.3]{eaton_group_1989}.

\end{remark}

\subsection{Proofs of Theorem~\ref{thm:relgrow_is_grow}, Proposition~\ref{prop:grow_also_av}, Proposition~\ref{prop:composite_invariant_grow}}
Here we prove all results in the main text except the main Theorem~\ref{theo:main_theorem}, which is deferred to the next subsection. 

%TYRON TODO: It does not make sense to me to introduce q_g/\bar p as the GROW e-statistic here, as it adds nothing to the proof. The crucial thing is only that 'GROW value' is equal to the infimum KL. It's also confusing, because \bar p in the first part does not correspond to \bar p in the second part. 
%TODO: Sometimes g is used as a fixed element (i.e. for any g, h ...) and sometimes it is used in the integral. Confusing!
\begin{proof}[Proof of Theorem~\ref{thm:relgrow_is_grow}]
  Let $g$ be a fixed group element of $G$. Recall from
  Remark~\ref{rem:dominated_models_left_invariant} that we may assume that both
  models are dominated by a left invariant measure $\nu$ on ${\cal X}$. 
  Theorem 1 by \citetalias{grunwald_safe_2023} (its simplest instantiation in their Section 2) implies that
  \begin{equation}\label{eq:baby_kl_duality}
    \sup_{T_n \text{ \E-stat.}}\mathbf{E}^{\mathbf{Q}}_{g}[\ln T_n]
    =
    % \mathbf{E}^{\mathbf{Q}}_{g} \left[\ln \frac{q_{g}(X^n)}{\bar{p}(X^n)} \right]
    % =
    \inf_{\mathbf{\Pi}_0}
    \KL(\mathbf{Q}_{g}, \mathbf{\Pi}_0^{g'}\mathbf{P}_{g'}),
  \end{equation}
  where the infimum is over all distributions $\mathbf{\Pi}_0$ on $G$. We will
  show that for any pair $g,h\in G$ and any prior $\mathbf{\Pi}$ on $G$, there
  exists a prior $\tilde{\mathbf{\Pi}}$ such that
  \begin{equation}\label{eq:kl_is_quasi_invariant}
    \KL(\mathbf{Q}_{g}, \mathbf{\Pi}^{g'}
    \mathbf{P}_{g'})=\KL(\mathbf{Q}_{h}, \tilde{\mathbf{\Pi}}^{g'}
    \mathbf{P}_{g'}).
  \end{equation}
  From this, our claim will follow: by symmetry, the previous display implies
  that
  $g\mapsto \sup_{T_n \text{ \E-stat.}}\mathbf{E}^{\mathbf{Q}}_{g}[\ln
  T_n]$ is constant over $G$ because of its relation to the $\KL$
  minimization in (\ref{eq:baby_kl_duality}). Let
  $\bar{p} = \int p_{g'}\rmd\mathbf{\Pi}(g')$, use both the invariance of $\nu$ and
  of ${\cal Q}$, and compute
  \begin{align*}
    \KL(\mathbf{Q}_{g}, \mathbf{\Pi}^{g'} \mathbf{P}_{g'})
    &=
      \mathbf{E}^{\mathbf{Q}}_{g} \left[\ln\frac{q_{g}(X^n)}{\bar{p}(x^n)}  \right]
    =
      \int
      q_{g}(x^n)\ln\frac{q_{g}(x^n)}{\bar{p}(x^n)}
      \rmd\nu(x^n) \\
    &=
      \int
      q_{h}(hg^{-1}x^n)\ln\frac{q_{h}(hg^{-1}x^n)}{\bar{p}(x^n)}
      \rmd\nu(x^n).
  \end{align*}
  Next, define $\tilde{\mathbf{\Pi}}$ as the probability distribution on $G$
  that assigns $\tilde{\mathbf{\Pi}}\{H\in B\}=\mathbf{\Pi}\{H \in gh^{-1}B\}$ for any
  measurable set $B\subseteq G$. Then
  \begin{equation*}
    \bar{p}(x^n)
    =
    \int p_{g'}(x^n) \rmd\mathbf{\Pi}(g')
    =
    \int
    p_{gh^{-1}g'}(x^n) \rmd\tilde{\mathbf{\Pi}}(g')
    =
    \int
    p_{g'}(hg^{-1}x^n) \rmd \tilde{\mathbf{\Pi}}(g').
  \end{equation*}
  Define $\tilde{p} = \int p_{g'} \rmd\tilde{\mathbf{\Pi}}(g')$. The two last displays
  together imply that
  \begin{equation*}
    \KL(\mathbf{Q}_{g}, \mathbf{\Pi}^{g'} \mathbf{P}_{g'})
    =
    \int
    q_{h}(hg^{-1}x^n)\ln\frac{q_{h}(hg^{-1}x^n)}{\tilde{p}(hg^{-1}x^n)}
    \rmd\nu(x^n).
  \end{equation*}
  After a change of variable and using the invariance of $\nu$, the right hand
  side of this equation equals
  $\KL(\mathbf{Q}_g, \tilde{\mathbf{\Pi}}^{g'}\mathbf{P}_{g'})$. Thus, this last
  equation is nothing but (\ref{eq:kl_is_quasi_invariant}), as was our
  objective. By our previous discusion, the result follows.
\end{proof}

\begin{proof}[Proof of  Proposition~\ref{prop:grow_also_av}]
  Let $g\in G$ be arbitrary but fixed. We start by showing that $T^{M_n}$ equals
  the likelihood ratio for $M^n = (M_1,\dots, M_n)$ between $\mathbf{P}_g$ and
  $\mathbf{Q}_g$. 
  For each $t>1$, the maximally invariant statistic at $n-1$,
  $M_{n-1} = m_{n-1}(X^{n-1})$ is invariant if seen as a function of $X^n$.
  Hence, by the maximality of $m_n$, $M_{n-1}$ can be written as a function of
  $M_n$. Repeating this reasoning $n-1$ times yields that $M_n$ contains all
  information about the value of $M^{n-1} = (M_1,\dots, M_{n-1})$, all the
  maximally invariant statistics at previous times. Two consequences fall from
  these observations. First, no additional information about $T^{M_n}$ is gained
  by knowing the value of $M^{n-1} = (M_1,\dots, M_{n-1})$ with respect to only
  knowing $M_{n-1}$, that is,
  $\mathbf{E}^{\mathbf{P}}_g\sqbrack{T^{M_n} | M_{n-1} } =
  \mathbf{E}^{\mathbf{P}}_g\sqbrack{T^{M_n} | M^{n-1}}$. Second, the likelihood
  ratio between $\mathbf{P}_g$ and $\mathbf{Q}_g$ for the sequence
  $M_1, \dots,M_n$ equals the likelihood ratio for $M_n$ alone, that is,
  \begin{equation*}
    T^{M_n}
    =
    \frac{
      q^{M_1, \dots, M_n}(m_1(X^1), \dots, m_n(X^n))
    }{
      p^{M_1, \dots, M_n}(m_1(X^1), \dots, m_n(X^n))
    }.
  \end{equation*}
  The previous two consequences, and a computation, together imply that $(T^{M_n})_{n\in \mathbb{N}}$ is
  an $M$-martingale under $\mathbf{P}_g$, that is,
  $\mathbf{E}^{\mathbf{P}}_g\sqbrack{T^{M_n} | M^{n-1}} = T^{M_{n-1}}$. Since
  $g\in G$ was arbitrary, the result follows.
\end{proof}

%\paragraph*{Proof of Proposition~\ref{prop:composite_invariant_grow}}
\begin{proof}[Proof of Proposition~\ref{prop:composite_invariant_grow}]
  Let $\mathbf{\Pi}_0^{g,\delta},\mathbf{\Pi}_1^{g,\delta}$ be two probability
  distributions on $G\times \Delta_0$ and $G\times\Delta_1$, respectively. If we
  call $\mathbf{\Pi}_0^{\delta}$ and $\mathbf{\Pi}_1^{\delta}$ their respective
  marginals on $\Delta_0$ and $\Delta_1$, then the information processing
  inequality implies that
  \begin{equation*}
    \KL(
    \mathbf{\Pi}_1^{g,\delta}\mathbf{Q}_{g,\delta},
    \mathbf{\Pi}_0^{g,\delta}\mathbf{P}_{g,\delta}
    )
    \geq
    \KL(
    \mathbf{\Pi}_1^{\delta}\mathbf{Q}_{\delta}^{M_n},
    \mathbf{\Pi}_0^{\delta}\mathbf{P}_{\delta}^{M_n}
    )
    \geq
    \KL(
    \mathbf{\Pi}_1^{\star\delta}\mathbf{Q}_{\delta}^{M_n},
    \mathbf{\Pi}_0^{\star\delta}\mathbf{P}_{\delta}^{M_n}
    ).
  \end{equation*}
  This means that the right-most member of the previous display is a lower bound
  on our target infimum, that is,
  \begin{equation}
    \label{eq:lwr_bd_inf_kl_composite}
    \inf_{\mathbf{\Pi}_0, \mathbf{\Pi}_1}
    \KL(\mathbf{\Pi}_1^{g,\delta}\mathbf{Q}_{g,\delta}
    \mathbf{\Pi}_0^{g,\delta}\mathbf{P}_{g,\delta}) \geq
    \KL(
    \mathbf{\Pi}_1^{\star\delta}\mathbf{Q}_{\delta}^{M_n},
    \mathbf{\Pi}_0^{\star\delta}\mathbf{P}_{\delta}^{M_n}
    ).
  \end{equation}
  To show that this is indeed an equality, it suffices to prove it when taking
  the infimum over a smaller subset of probability distributions
  $\mathbf{\Pi}_0, \mathbf{\Pi}_1$. We proceed to build such a subset. Let
  ${\cal P}(\mathbf{\Pi}_0^{\star \delta})$ be the set of probability
  distributions on $G\times\Delta_0$ with marginal distribution
  $\mathbf{\Pi}_0^{\star \delta}$. Define analogously the set of probability
  distributions ${\cal P}(\mathbf{\Pi}_1^{\star \delta})$ on $G\times\Delta_1$.
  By our assumptions, Theorem~\ref{theo:main_theorem} can be readily used to
  conclude that
  \begin{equation}\label{eq:marginal_kl_minimization}
    \inf_{
      (\mathbf{\Pi}_0, \mathbf{\Pi}_1)
      \in
      {\cal P}(\mathbf{\Pi}_0^{\star \delta})\times
      {\cal P}(\mathbf{\Pi}_1^{\star \delta})
    }
    \KL(\mathbf{\Pi}_1^{g,\delta}\mathbf{Q}_{g,\delta},
    \mathbf{\Pi}_0^{g,\delta}\mathbf{P}_{g,\delta})
    =
    \KL(\mathbf{\Pi}_1^{\star\delta}\mathbf{Q}_{\delta}^{M_n},
    \mathbf{\Pi}_0^{\star\delta}\mathbf{P}_{\delta}^{M_n})
  \end{equation}
  holds; (\ref{eq:lwr_bd_inf_kl_composite}) and
  (\ref{eq:marginal_kl_minimization}) together imply the result that we were
  after.
\end{proof}
\subsection{Proof of the main theorem, Theorem~\ref{theo:main_theorem}}
\label{sec:proof-main-theorem}
For the proof of the main result, we use an equivalent definition of amenability
to the one that was already anticipated in Section~\ref{sec:intro-group-invariance}.
We take the one that suits our purposes best \citep[see][p. 109,
Condition~$A_1$]{bondar_amenability_1981}.
That is, a group $G$ is amenable if there exists an increasing sequence of symmetric compact subsets $C_1\subseteq C_2,\dots \subset G$ such that, for any compact set
  $K\subseteq G$,
  \begin{equation*}
    \frac{\rho\{C_i\}}{\rho\{C_iK\}}\to 1, \ \ \text{ as $i\to \infty$.}
  \end{equation*}
In this formulation, amenability is the existence of \textit{almost invariant}
symmetric compact subsets of the group $G$. We use these sets to build a
sequence of \textit{almost invariant} probability measures when $G$ is
noncompact.

\begin{proof}[Proof of Theorem~\ref{theo:main_theorem}]
  Under our assumptions, Theorem 2 of \citet{bondar_borel_1976} implies the
  existence of a bimeasurable one-to-one map
  $r:{\cal X}^n\to G \times {\cal X}^n / G$ such that $r(x^n) = (h(x^n), m(x^n))$
  and $r(gx^n) = (gh(x^n), m(x^n))$ for $h(x^n)\in G$ and
  $m(x^n)\in {\cal X}^n / G$. 
  Hence, by a change of variables, we can take densities with respect to the
  image measure $\mu$ of $\nu$ under the map $r$ on $G \times {\cal X}^n / G$. 
  Call the random variables $M = m(X^n)$ and $H = h(X^n)$. We can therefore assume, without loss
  of generality, that the data is of the form $(H,M)$, that the group $G$ acts
  canonically by multiplication on the first component, and that the measures
  are with respect to a $G$-invariant measure $\mu = \lambda \times \beta$ where
  $\lambda$ is the Haar measure on $G$ and $\beta$ is some measure on
  ${\cal X}^n / G$ (see Remark~\ref{rem:dominated_models_left_invariant}). Note that rewriting the data in this way does not affect our objective because the $\KL$ divergence remains unchanged under bijective transformations of the data. For
  each $g\in G$, write $\mathbf{P}^{H|m}_g$ and $\mathbf{Q}^{H | m}_g$ for the
  conditional probabilities $\mathbf{P}^H_{g}\{\ \cdot \ | M = m\}$ and
  $\mathbf{Q}^H_g\{\ \cdot \ | M = m\}$, which can be obtained through
  disintegration \citep[see][]{chang_conditioning_1997}, and write
  $p_g( \ \cdot \ | m)$ and $q_g( \ \cdot \ | m)$ for their respective
  conditional densities with respect to the left Haar measure $\lambda$. 
  % What do we mean by "the measures are ..." ? Do we mean that they are dominated by? In any case, 'the measures' itself is pretty vague.

  We turn to our $\KL$ minimization objective. The chain rule for the $\KL$
  divergence implies that, for any probability distribution $\mathbf{\Pi}$ on
  $G$,
  \begin{equation}\label{eq:kl_chain_rule}
    \KL(\mathbf{\Pi}^g\mathbf{Q}_g, \mathbf{\Pi}^g\mathbf{P}_g) =
    \KL(\mathbf{Q}^M, \mathbf{P}^M) +
    \int \KL(\mathbf{\Pi}^g\mathbf{Q}^{H | m}_{g},
      \mathbf{\Pi}^g\mathbf{P}^{H | m}_{g}) \rmd\mathbf{Q}^M(m).
  \end{equation}
  In order to prove our claim, we will build a sequence
  $\{\mathbf{\Pi}_i\}_{i\in\nats}$ of probability distributions on $G$ such that
  the term in (\ref{eq:kl_chain_rule}) pertaining the conditional distributions
  given $M$---the second term on the right hand side---goes to zero, that is,
  such that
  \begin{equation}\label{eq:objective_main_theo}
    \int \KL(\mathbf{\Pi}^g_i\mathbf{Q}^{H | m}_{g},
    \mathbf{\Pi}^g_i\mathbf{P}^{H | m}_{g})  \rmd\mathbf{Q}^M(m)
    \to
    0 \text{\ \  as  \ \ } i\to \infty.
  \end{equation}
  We define the distributions $\mathbf{\Pi}_i$ as the
  normalized restriction of the right Haar measure $\rho$ to carefully chosen
  compact sets $C_i\subset G$, that we describe in brief. In other words, for
  $B\subseteq G$ measurable, we define $\mathbf{\Pi}_i$ by
  \begin{equation}\label{eq:main_theo_measures_definition}
    \mathbf{\Pi}_i\{g\in B\} := \frac{\rho\{B\cap C_i\}}{\rho\{C_i\}},
  \end{equation}
  Next, the choice of compact sets $C_i$. For technical reasons that will become
  apparent later, we pick $C_i = J_iK_iL_i$, where $J_i$, $K_i$, and $L_i$ are
  increasing compact symmetric neighborhoods of the unity of $G$ with the growth
  condition that $C_i$ is not much bigger---measured by $\rho$--than $J_i$. More
  precisely, we choose $C_i$ according to the following lemma.
  \begin{lemma}\label{lem:prior_sequence_existence}
    Under the amenability of $G$ there exist
    sequences $\{J_i\}_{i\in \nats}$, $\{K_i\}_{i\in \nats}$ and
    $\{L_i\}_{i\in \nats}$ of compact symmetric neighborhoods of the unity of
    $G$, each increasing to cover $G$, such that
    \begin{equation*}
      \frac{\rho\{J_i\}}{\rho\{J_iK_iL_i\}} \to 1 \ \ \text{as $i\to \infty$.}
    \end{equation*}
  \end{lemma}
  The proof of this Lemma is given in Appendix~\ref{app:proof_tech_lemmas}.
  There is no risk of dividing by $\infty$ in
  \eqref{eq:main_theo_measures_definition}: by the continuity of the group
  operation each $C_i$ is compact, hence $\rho\{C_i\}<\infty$.
  Lemma~\ref{lem:prior_sequence_existence} ensures that
  $\mathbf{\Pi}_i\{g\in J_i\}\to 1$ as $i\to\infty$, a fact that will be
  useful later in the proof. Write
  $\mathbf{Q}_{i}^{H | m} : = \mathbf{\Pi}^g_i\mathbf{Q}^{H | m}_{g}$, and
  $\mathbf{P}_{i}^{H | m} : = \mathbf{\Pi}^g_i\mathbf{P}^{H | m}_{g}$, and
  $q_i(h|m)$ and $p_i(h|m)$ for their respective densities. We use a change of
  variable and split the integral in our quantity of interest from
  \eqref{eq:objective_main_theo}. To this end, notice that for any function
  $f = f(h,m)$, the expected value
  $\mathbf{E}_g^{\mathbf{Q}}[f(H,M)] = \mathbf{E}_1^{\mathbf{Q}}[f(gH,M)]$. %TODO: This notation is undefined, should we define it?
  Indeed,
  \begin{align*}
    \iint
    f(h,m)
    q_g(h,m)
    \rmd \lambda(g)\rmd \beta(m)
    &=
    \iint
    f(h,m)
    q_1(g^{-1}h,m)
    \rmd \lambda(g)\rmd \beta(m)\\
    &=
    \iint
    f(gh,m)
    q_1(h,m)
    \rmd \lambda(g)\rmd \beta(m).
  \end{align*}
  Use this fact to obtain that
  \begin{align}
%&
    \int
    \KL(\mathbf{\Pi}^g_i\mathbf{Q}^{H|m}_{g},
    \mathbf{\Pi}^g_i\mathbf{P}^{H|m}_{g})
    \rmd\mathbf{Q}(m)
    % &=
    %   \int q_1(g^{-1}h, m)\ln\frac{q_i(h|m)}{p_i(h|m)}  \rmd\lambda(h) \rmd
    %   \mathbf{\Pi}_i(g)\rmd\beta(m)
    %   \nonumber \\
    % & =
    %   \int q_1(h, m)\ln\frac{q_i(gh|m)}{p_i(gh|m)} \rmd\lambda(h) \rmd \mathbf{\Pi}_i(g)\rmd\beta(m)
    %   \nonumber \\
     =
      \int
      \mathbf{E}^{\mathbf{Q}}_1\sqbrack{\ln\frac{q_i(gH|M)}{p_i(gH|M)}}
      \rmd\mathbf{\Pi}_i(g) 
    =  \label{eq:main_theo_objective_split} \\ \nonumber 
%    &  
     % \begin{multlined}[t]
{\small         \underbrace{
          \int \mathbf{E}^{\mathbf{Q}}_1\sqbrack{\indicator{gH\in
              J_iK_i}\ln\frac{q_i(gH|M)}{p_i(gH|M)}} \rmd \mathbf{\Pi}_i(g)
          }_{\text{A}} + 
          \underbrace{
            \int
            \mathbf{E}^{\mathbf{Q}}_1\sqbrack{\indicator{gH\notin
                J_iK_i}\ln\frac{q_i(gH|M)}{p_i(gH|M)}}
            \rmd \mathbf{\Pi}_i(g)
          }_{\text{B}}.}
      %  \end{multlined}
  \end{align}
  We separate the rest of the proof in two steps, one for bounding each term in
  \eqref{eq:main_theo_objective_split}. 
  These steps use two technical lemmas that we prove in Appendix~\ref{app:proof_tech_lemmas}.
%  whose proof we give after showing how they help at achieving our goals.

  \textbf{Bound for A in \eqref{eq:main_theo_objective_split}:} Recall that
  % \begin{equation*}
  %   \ln\frac{q_i(gh|m)}{p_i(gh|m)} = \ln\frac{\int_{J_i K_i
  %   L_i}q_{g'}(gh|m) \rmd \rho(g')}{\int_{J_i K_i
  %   L_i}p_{g'}(gh|m) \rmd \rho(g')}
  % \end{equation*}
  \begin{equation*}
    \ln\frac{q_i(gh|m)}{p_i(gh|m)}
    =
    \ln
    \frac{
      \int \indicator{g'\in J_i K_i L_i}  q_{g'}(gh|m)  \rmd\rho(g')
    }{
      \int \indicator{g'\in J_i K_i L_i}  p_{g'}(gh|m)  \rmd\rho(g')
    }.
  \end{equation*}
  Use $N = J_iK_i$---not necessarily symmetric---and $L = L_i$ in the following
  lemma.
  \begin{lemma}\label{lem:first_term_bound}
    Let $N$ and $L$ be compact subsets of $G$. Assume that $L$ is symmetric.
    Then, for each $m\in {\cal X}^n / G$ it holds that
    \begin{equation*}
      \sup_{h'\in N}
      \bracks{
      \ln\frac{
        \int \indicator{g\in NL} \ q_{g}(h'|m)  \rmd \rho(g)
      }{
        \int  \indicator{g\in NL} \ p_{g}(h'|m)  \rmd \rho(g)
      }
      }
      \leq
      -\ln \mathbf{P}_{1}\{H\in L \ | \ M = m\}.
    \end{equation*}
  \end{lemma}
  With this lemma at hand, conclude that, for all $gh\in J_iK_i$, and
  $m\in \mathcal{M}$
  \begin{equation*}
    \ln\frac{q_i(gh|m)}{p_i(gh|m)}
    \leq
    -\ln \mathbf{P}_1\{H\in L_i \ |  \ M = m\}.
  \end{equation*}
  At the same time this implies that A in \eqref{eq:main_theo_objective_split}
  is smaller than
  \begin{equation*}
    -\int \ln \mathbf{P}_1\{H \in L_i \ | \ M = m\} \rmd\mathbf{Q}(m).
  \end{equation*}
  Since the sets $L_i$ were chosen to satisfy $L_i\uparrow G$, the probability
  $\mathbf{P}_1\{H\in L_i \ | \ M = m\}\to 1$ monotonically for each value of $m$.
  Consequently the quantity in last display tends to 0 by the monotone
  convergence theorem, and so does A in \eqref{eq:main_theo_objective_split}.
  This ends the first step of the proof. Now, we turn to the second term in
  \eqref{eq:main_theo_objective_split}.

  \textbf{Bound for B in \eqref{eq:main_theo_objective_split}:} Our strategy
  at this point is to show that, as $i\to\infty$,
  \begin{equation}\label{eq:main_theo_prob_zero_set}
    \int \mathbf{Q}_1\bracks{gH\notin J_iK_i} \rmd\mathbf{\Pi}_i(g) \to 0,
  \end{equation}
  and to use (\ref{eq:kl_moment_assumption}) to show our goal, that B in
  \eqref{eq:main_theo_objective_split} tends to zero. To show
  \eqref{eq:main_theo_prob_zero_set}, notice that if $g\in J_i$ and
  $h\in K_i$, then $gh\in J_iK_i$, which implies that
  \begin{equation*}
    \int \mathbf{Q}_1\bracks{gH\in J_iK_i} \rmd \mathbf{\Pi}_i(g)
    \geq
    \mathbf{\Pi}_i\bracks{g\in J_i}\mathbf{Q}_1\bracks{H\in K_i}.
  \end{equation*}
  Since the sets $K_i$ increase to cover $G$, we have
  $\mathbf{Q}_1\bracks{H\in K_i}\to 1$ as $i\to \infty$, and by our initial
  choice of sets $J_i,K_i,L_i$, the probability
  $\mathbf{\Pi}_i\bracks{g\in J_i}\to 1$, as $i\to\infty$. Hence
  \eqref{eq:main_theo_prob_zero_set} holds.
  % Two applications of Jensen's inequality and one of Bayes' theorem shows that
  % \begin{align*}
  %   \mathbf{\Pi}_i^g\mathbf{Q}^{h,m}\sqbrack{
  %   \abs{
  %   \ln\frac{q_i(gh|m)}{p_i(gh|m)}
  % }
  % }
  %   &\leq
  %   \mathbf{\Pi}^g_i\mathbf{Q}^{h,m}\sqbrack{\abs{\mathbf{\Pi}^{g'}_{i,
  %   h,m}\sqbrack{\ln\frac{q_{g'}(gh|m)}{p_{g'}(gh|m)}}}}\\
  %   &\leq
  %   \mathbf{\Pi}^g_i\mathbf{Q}^{h,m}\mathbf{\Pi}^{g'}_{i,
  %   h,m}\sqbrack{\abs{\ln\frac{q_{g'}(gh|m)}{p_{g'}(gh|m)}}}\\
  %   &=  \mathbf{Q}^{h,m}_i\mathbf{\Pi}^{g'}_{i,
  %   h,m}\sqbrack{\abs{\ln\frac{q_{g'}(h|m)}{p_{g'}(h|m)}}}\\
  %   &=
  %   \mathbf{\Pi}^{g'}_{i}
  %   \mathbf{Q}^{h,m}_{g'}\sqbrack{\abs{\ln\frac{q_{g'}(h|m)}{p_{g'}(h|m)}}}\\
  %   &=
  %   \mathbf{Q}^{h,m}\sqbrack{\abs{\ln\frac{q_1(h|m)}{p_1(h|m)}}}
  % \end{align*}
  % Hence, as
  % \begin{equation*}
  %   \mathbf{Q}^{h,m}\sqbrack{\abs{\ln\frac{q(h|m)}{p(h|m)}}} \leq
  %   \mathbf{Q}^{h,m}\sqbrack{\abs{\ln\frac{q(h,m)}{p(h,m)}}} +
  %   \mathbf{Q}^{h,m}\sqbrack{\abs{\ln\frac{q(h, m)}{p(h,m)}}} <\infty
  % \end{equation*}
  % (\ref{eq:kl_moment_assumption}) implies that the function $g,h,m$
  To bound the second term, we use the following lemma with $\mathbf{\Pi} = \mathbf{\Pi}_i$.
  \begin{lemma}
    \label{lem:second_term_bound}
    Let $\mathbf{\Pi}$ be a distribution on $G$. Then, for each $h\in G$ and
    $m\in {\cal X}^n / G$, setting $\rmd\mathbf{\Pi}(g | h,m) =
    \frac{q_g(h|m)\rmd\mathbf{\Pi}(g)}{\int q_g(h|m)\rmd\mathbf{\Pi}(g) }$, it holds that
    \begin{equation*}
      \ln\frac{\int q_g(h|m)\rmd \mathbf{\Pi}(g)}{\int p_g(h|m)\rmd
        \mathbf{\Pi}(g)}
      \leq
      \int
      \ln\frac{q_{g}(h|m)}{p_{g}(h|m)}
      \rmd \mathbf{\Pi}(g | h, m ).
    \end{equation*}
  \end{lemma}
  After invoking the previous lemma, apply H\"older's and Jensen's inequality
  consecutively to bound B in \eqref{eq:main_theo_objective_split} by
\newcommand{\abbrev}{\; \ell \;}
  \begin{align}
    &\iint
      % \int
      \sqbrack{\indicator{gh\notin
      J_iK_i}
      \int
      \abbrev(gh|m)
      %\sqbrack{\ln\frac{q_{g'}(gh|m)}{p_{g'}(gh|m)}}
      \rmd\mathbf{\Pi}_i(g' | h, m)
      }
      \rmd\mathbf{Q}_1(h,m)
      \rmd\mathbf{\Pi}_i(g)
      \leq
      \label{eq:main_theo_holder_result}
      \\
    &  {\small
    %\begin{multlined}      \leq 
    \underbrace{\paren{
          \int \mathbf{Q}_1\bracks{gH\notin J_iK_i} \rmd\mathbf{\Pi}_i(g)
        }^{1/q}}_{\to 0 \text{ as } i\to\infty \text{ by
          (\ref{eq:main_theo_prob_zero_set})}}
          %\\
      \paren{
        \iint
        % \int
        \abs{
          \int
          \abbrev(gh|m)
          %\sqbrack{\ln\frac{q_{g'}(gh|m)}{p_{g'}(gh|m)}} 
          \rmd\mathbf{\Pi}_i(g'|h,m)
        }^{p}
        \rmd\mathbf{Q}_1(h,m)
        \rmd \mathbf{\Pi}_i(g)
      }^{1/p}
   % \end{multlined}
   }
    \nonumber
  %   &\begin{multlined}
  %     \leq
  %     \underbrace{
  %     \paren{
  %     \int \mathbf{Q}_1\bracks{gH\notin J_iK_i} \rmd\mathbf{\Pi}_i(g)
  %     }^{1/q}
  %     }_{\to 0 \text{ as } i\to\infty \text{ by
  %     (\ref{eq:main_theo_prob_zero_set})}} \\
  %     \paren{
  %       \iiint
  %     % \int
  %     % \int
  %     \abs{\ln\frac{q_{g'}(gh|m)}{p_{g'}(gh|m)}}^{p}
  %     \rmd \mathbf{\Pi}_i(g' | h,m)
  %     \rmd \mathbf{Q}_1(h,m)
  %     \rmd \mathbf{\Pi}_i(g)
  %   }^{1/p},
  % \end{multlined}\nonumber
  \end{align}
  where here and in the sequel, $\abbrev(gh|m)$ abbreviates
  $\ln\frac{q_{g'}(gh|m)}{p_{g'}(gh|m)}$, and   
  $p = 1 + \varepsilon$ and $q$ is $p$'s H\"older conjugate, that is,
  $1/p + 1/q = 1$. Next, we show that the second factor on the right of 
  (\ref{eq:main_theo_holder_result}) remains bounded as $i\to\infty$. By
  Jensen's inequality, this quantity is smaller than
  \begin{equation*}
    \paren{
      \iiint
      % \int
      % \int
      \abs{
     \abbrev(gh|m)
      %\ln\frac{q_{g'}(gh|m)}{p_{g'}(gh|m)}
      }^{p}
      \rmd \mathbf{\Pi}_i(g' | h,m)
      \rmd \mathbf{Q}_1(h,m)
      \rmd \mathbf{\Pi}_i(g)
    }^{1/p}.
  \end{equation*}
  After a series of rewritings and using our
  Assumption~(\ref{eq:kl_moment_assumption}), we will show that this quantity is
  bounded. First, we 
  deduce that
  \begin{multline*}
  %\begin{align*}
    \iint
    \abs{
    \abbrev(gh|m) %\ln\frac{q_{g'}(gh|m)}{p_{g'}(gh|m)}
    }^{p}
    \rmd \mathbf{\Pi}_i(g' | h,m)
    \rmd \mathbf{Q}_1(h,m)
    \rmd \mathbf{\Pi}_i(g)
    %& = 
    =
    \\
    \iint
    \abs{ 
    \abbrev(h|m) %\ln\frac{q_{g'}(h|m)}{p_{g'}(h|m)}
    }^{p}
    \rmd \mathbf{\Pi}_i(g' | h,m)
    \rmd \mathbf{Q}_g(h,m)
    \rmd \mathbf{\Pi}_i(g)
    %& = \\
    = \\
    \iint
    \abs{
    \abbrev(h|m) % \ln\frac{q_{g'}(h|m)}{p_{g'}(h|m)}
    }^{p}
    \rmd \mathbf{\Pi}_i(g' | h,m)
    \rmd \mathbf{Q}_i(h,m)
    =
    %&=
    \mathbf{E}_1^{\mathbf{Q}}
    \abs{\ln\frac{q_{1}(H|M)}{p_{1}(H|M)}}^{p},
%\end{align*}
 \end{multline*}
where we used again the change of variable that we used to obtain
  (\ref{eq:main_theo_objective_split})---but now in the opposite direction---and in the final equality, we used 
  % At this point, 
  Bayes' theorem.
  %implies that this last quantity is equal to
\commentout{  \begin{equation*}
    \iint
    % \int
    % \int
    \abs{
    \abbrev %\ln\frac{q_{g'}(h|m)}{p_{g'}(h|m)}
    }^{p}
    \rmd \mathbf{Q}_{g'}(h,m)
    \rmd \mathbf{\Pi}_i(g')
    =
    \mathbf{E}_1^{\mathbf{Q}}
    \abs{\ln\frac{q_{1}(H|M)}{p_{1}(H|M)}}^{p}
  \end{equation*}}
  Hence, as
  %% Broke it again because the subindex 1 was missing and adding it makes the line too long
%  \begin{equation*}
    \begin{align*}
      \paren{\mathbf{E}^{\mathbf{Q}}_1\sqbrack{\abs{\ln\frac{q_1(H|M)}{p_1(H|M)}}^p}}^{1/p}
      &\leq     
      \paren{\mathbf{E}^{\mathbf{Q}}_1\sqbrack{\abs{\ln\frac{q_1(H,M)}{p_1(H,M)}}^p}}^{1/p}
      +
      \paren{\mathbf{E}^{\mathbf{Q}}_1\sqbrack{\abs{\ln\frac{q_1(M)}{p_1(M)}}}^p}^{1/p} \\
      &<\infty     
   \end{align*}
%  \end{equation*}
  by (\ref{eq:kl_moment_assumption}), we have shown that
  \eqref{eq:main_theo_holder_result} tends to 0 as $i\to\infty$ and that
  consequently B in \eqref{eq:main_theo_objective_split} tends to $0$ in the
  same limit.

  After completing these two steps, we have shown that both A and B in
  \eqref{eq:main_theo_objective_split} tend to 0 as $i\to \infty$, and that
  consequently the claim of the theorem follows. All is left is to prove lemmas
  \ref{lem:prior_sequence_existence},  \ref{lem:first_term_bound}, and \ref{lem:second_term_bound}. The proofs being straightforward but tedious, we delegated these to Appendix~\ref{app:furtherproofs}.
\end{proof}

\DeclareRobustCommand{\VANDER}[3]{#3}
\section{Acknowledgements}
Peter D. Gr\"unwald is also affiliated with the Mathematical Institute of Leiden University. This work is part of the research program with project number 617.001.651, which is  financed by the Dutch Research Council (NWO). We thank Wouter Koolen for useful conversations and for inspiring the example in Appendix~\ref{app:filtration_counterexample}.
%\bibliographystyle{imsart-number} % Style BST file (imsart-number.bst or
% imsart-nameyear.bst)
\bibliographystyle{imsart-nameyear}
%\bibliography{bibliography}       % Bibliography file (usually '*.bib')
\bibliography{haarbib,haarbib_muriel,haarbib_tyron,haarbib_peter}

\begin{thebibliography}{49}
% BibTex style file: imsart-nameyear.bst, 2017-11-03
% Default style options (sort=1,type=nameyear).
% Used options (sort=1,type=nameyear).

\bibitem[\protect\citeauthoryear{Andersson}{1982}]{andersson_distributions_1982}
\begin{barticle}[author]
\bauthor{\bsnm{Andersson},~\bfnm{Steen}\binits{S.}}
(\byear{1982}).
\btitle{Distributions of {maximal} {invariants} {using} {quotient} {measures}}.
\bjournal{The Annals of Statistics}
\bvolume{10}
\bpages{955--961}.
\bdoi{10.1214/aos/1176345885}
\bmrnumber{MR663446}
\end{barticle}
\endbibitem

\bibitem[\protect\citeauthoryear{Berger, Pericchi and
  Varshavsky}{1998}]{berger1998bayes}
\begin{barticle}[author]
\bauthor{\bsnm{Berger},~\bfnm{James~O}\binits{J.~O.}},
  \bauthor{\bsnm{Pericchi},~\bfnm{Luis~R}\binits{L.~R.}} \AND
  \bauthor{\bsnm{Varshavsky},~\bfnm{Julia~A}\binits{J.~A.}}
(\byear{1998}).
\btitle{Bayes factors and marginal distributions in invariant situations}.
\bjournal{Sankhy{\=a}: The Indian Journal of Statistics, Series A}
\bvolume{60}
\bpages{307--321}.
\end{barticle}
\endbibitem

\bibitem[\protect\citeauthoryear{Berger and Sun}{2008}]{berger_objective_2008}
\begin{barticle}[author]
\bauthor{\bsnm{Berger},~\bfnm{James~O.}\binits{J.~O.}} \AND
  \bauthor{\bsnm{Sun},~\bfnm{Dongchu}\binits{D.}}
(\byear{2008}).
\btitle{Objective priors for the bivariate normal model}.
\bjournal{The Annals of Statistics}
\bvolume{36}
\bpages{963--982}.
\bdoi{10.1214/07-AOS501}
\bmrnumber{MR2396821}
\end{barticle}
\endbibitem

\bibitem[\protect\citeauthoryear{Berk}{1972}]{berk_note_1972}
\begin{barticle}[author]
\bauthor{\bsnm{Berk},~\bfnm{Robert~H.}\binits{R.~H.}}
(\byear{1972}).
\btitle{A {note} on {sufficiency} and {invariance}}.
\bjournal{The Annals of Mathematical Statistics}
\bvolume{43}
\bpages{647--650}.
\bdoi{10.1214/aoms/1177692645}
\end{barticle}
\endbibitem

\bibitem[\protect\citeauthoryear{Bhowmik and King}{2007}]{BhowmikKing2007}
\begin{barticle}[author]
\bauthor{\bsnm{Bhowmik},~\bfnm{Jahar~L.}\binits{J.~L.}} \AND
  \bauthor{\bsnm{King},~\bfnm{Maxwell~L.}\binits{M.~L.}}
(\byear{2007}).
\btitle{Maximal invariant likelihood based testing of semi-linear models}.
\bjournal{Statistical Papers}
\bvolume{48}
\bpages{357--383}.
\bdoi{10.1007/s00362-006-0342-7}
\end{barticle}
\endbibitem

\bibitem[\protect\citeauthoryear{Bondar}{1976}]{bondar_borel_1976}
\begin{barticle}[author]
\bauthor{\bsnm{Bondar},~\bfnm{James~V.}\binits{J.~V.}}
(\byear{1976}).
\btitle{Borel {cross}-{sections} and {maximal} {invariants}}.
\bjournal{The Annals of Statistics}
\bvolume{4}
\bpages{866--877}.
\bdoi{10.1214/aos/1176343585}
\bmrnumber{MR474589}
\end{barticle}
\endbibitem

\bibitem[\protect\citeauthoryear{Bondar and
  Milnes}{1981}]{bondar_amenability_1981}
\begin{barticle}[author]
\bauthor{\bsnm{Bondar},~\bfnm{James~V.}\binits{J.~V.}} \AND
  \bauthor{\bsnm{Milnes},~\bfnm{Paul}\binits{P.}}
(\byear{1981}).
\btitle{Amenability: {A} survey for statistical applications of {Hunt}-{Stein}
  and related conditions on groups}.
\bjournal{Zeitschrift für Wahrscheinlichkeitstheorie und verwandte Gebiete}
\bvolume{57}
\bpages{103--128}.
\bdoi{10.1007/BF00533716}
\end{barticle}
\endbibitem

\bibitem[\protect\citeauthoryear{Bourbaki}{2004}]{bourbaki_integration_2004}
\begin{bbook}[author]
\bauthor{\bsnm{Bourbaki},~\bfnm{N.}\binits{N.}}
(\byear{2004}).
\btitle{Integration {II}: {Chapters} 7–9},
\bedition{1st} ed.
\bseries{Elements of {Mathematics}}.
\bpublisher{Springer-Verlag}, \baddress{Berlin Heidelberg}.
\bdoi{10.1007/978-3-662-07931-7}
\end{bbook}
\endbibitem

\bibitem[\protect\citeauthoryear{Chang and
  Pollard}{1997}]{chang_conditioning_1997}
\begin{barticle}[author]
\bauthor{\bsnm{Chang},~\bfnm{J.~T.}\binits{J.~T.}} \AND
  \bauthor{\bsnm{Pollard},~\bfnm{D.}\binits{D.}}
(\byear{1997}).
\btitle{Conditioning as disintegration}.
\bjournal{Statistica Neerlandica}
\bvolume{51}
\bpages{287--317}.
\bdoi{10.1111/1467-9574.00056}
\end{barticle}
\endbibitem

\bibitem[\protect\citeauthoryear{Cover and Thomas}{2006}]{cover_elements_2006}
\begin{bbook}[author]
\bauthor{\bsnm{Cover},~\bfnm{Thomas~M.}\binits{T.~M.}} \AND
  \bauthor{\bsnm{Thomas},~\bfnm{Joy~A.}\binits{J.~A.}}
(\byear{2006}).
\btitle{Elements of {Information} {Theory}}.
\bseries{Wiley {Series} in {Telecommunications} and {Signal} {Processing}}.
\bpublisher{Wiley-Interscience}, \baddress{New York, NY, USA}.
\end{bbook}
\endbibitem

\bibitem[\protect\citeauthoryear{Cox}{1952}]{cox_sequential_1952}
\begin{barticle}[author]
\bauthor{\bsnm{Cox},~\bfnm{D.~R.}\binits{D.~R.}}
(\byear{1952}).
\btitle{Sequential tests for composite hypotheses}.
\bjournal{Mathematical Proceedings of the Cambridge Philosophical Society}
\bvolume{48}
\bpages{290--299}.
\bdoi{10.1017/S030500410002764X}
\end{barticle}
\endbibitem

\bibitem[\protect\citeauthoryear{Darling and Robbins}{1968}]{darling1968some}
\begin{barticle}[author]
\bauthor{\bsnm{Darling},~\bfnm{DA}\binits{D.}} \AND
  \bauthor{\bsnm{Robbins},~\bfnm{Herbert}\binits{H.}}
(\byear{1968}).
\btitle{Some nonparametric sequential tests with power one}.
\bjournal{Proceedings of the National Academy of Sciences}
\bvolume{61}
\bpages{804--809}.
\bdoi{10.1073/pnas.61.3.804}
\end{barticle}
\endbibitem

\bibitem[\protect\citeauthoryear{Dawid, Stone and
  Zidek}{1973}]{dawid_marginalization_1973}
\begin{barticle}[author]
\bauthor{\bsnm{Dawid},~\bfnm{A.~P.}\binits{A.~P.}},
  \bauthor{\bsnm{Stone},~\bfnm{M.}\binits{M.}} \AND
  \bauthor{\bsnm{Zidek},~\bfnm{J.~V.}\binits{J.~V.}}
(\byear{1973}).
\btitle{Marginalization {Paradoxes} in {Bayesian} and {Structural}
  {Inference}}.
\bjournal{Journal of the Royal Statistical Society, Series B (Methodological)}
\bvolume{35}
\bpages{189--233}.
\bdoi{10.1111/j.2517-6161.1973.tb00952.x}
\end{barticle}
\endbibitem

\bibitem[\protect\citeauthoryear{Durrett}{2019}]{durrett_probability_2019}
\begin{bbook}[author]
\bauthor{\bsnm{Durrett},~\bfnm{Rick}\binits{R.}}
(\byear{2019}).
\btitle{Probability: {Theory} and {examples}},
\bedition{5th} ed.
\bseries{Cambridge {Series} in {Statistical} and {Probabilistic} {Mathematics}}
\bvolume{49}.
\bpublisher{Cambridge University Press}.
\bdoi{10.1017/9781108591034}
\end{bbook}
\endbibitem

\bibitem[\protect\citeauthoryear{Eaton}{1989}]{eaton_group_1989}
\begin{barticle}[author]
\bauthor{\bsnm{Eaton},~\bfnm{Morris~L.}\binits{M.~L.}}
(\byear{1989}).
\btitle{Group {invariance} {applications} in {statistics}}.
\bjournal{Regional Conference Series in Probability and Statistics}
\bvolume{1}
\bpages{i--133}.
\bdoi{10.1214/cbms/1462061029}
\end{barticle}
\endbibitem

\bibitem[\protect\citeauthoryear{Eaton and
  Sudderth}{1999}]{eaton_consistency_1999}
\begin{barticle}[author]
\bauthor{\bsnm{Eaton},~\bfnm{Morris~L.}\binits{M.~L.}} \AND
  \bauthor{\bsnm{Sudderth},~\bfnm{William~D.}\binits{W.~D.}}
(\byear{1999}).
\btitle{Consistency and {strong} {inconsistency} of {group}-{invariant}
  {predictive} {inferences}}.
\bjournal{Bernoulli}
\bvolume{5}
\bpages{833--854}.
\bnote{Publisher: International Statistical Institute (ISI) and Bernoulli
  Society for Mathematical Statistics and Probability}.
\bdoi{10.2307/3318446}
\end{barticle}
\endbibitem

\bibitem[\protect\citeauthoryear{Eaton and Sudderth}{2002}]{eaton_group_2002}
\begin{barticle}[author]
\bauthor{\bsnm{Eaton},~\bfnm{Morris~L.}\binits{M.~L.}} \AND
  \bauthor{\bsnm{Sudderth},~\bfnm{William~D.}\binits{W.~D.}}
(\byear{2002}).
\btitle{Group invariant inference and right {Haar} measure}.
\bjournal{Journal of Statistical Planning and Inference}
\bvolume{103}
\bpages{87--99}.
\bdoi{10.1016/S0378-3758(01)00199-9}
\end{barticle}
\endbibitem

\bibitem[\protect\citeauthoryear{Giri, Kiefer and Stein}{1963}]{Giri1963}
\begin{barticle}[author]
\bauthor{\bsnm{Giri},~\bfnm{N.}\binits{N.}},
  \bauthor{\bsnm{Kiefer},~\bfnm{J.}\binits{J.}} \AND
  \bauthor{\bsnm{Stein},~\bfnm{C.}\binits{C.}}
(\byear{1963}).
\btitle{{Minimax character of Hotelling's $T^2$ test in the simplest case}}.
\bjournal{The Annals of Mathematical Statistics}
\bvolume{34}
\bpages{1524 -- 1535}.
\bdoi{10.1214/aoms/1177703884}
\end{barticle}
\endbibitem

\bibitem[\protect\citeauthoryear{Gr{\"u}nwald}{2023}]{Grunwald23}
\begin{barticle}[author]
\bauthor{\bsnm{Gr{\"u}nwald},~\bfnm{Peter}\binits{P.}}
(\byear{2023}).
\btitle{The {E}-Posterior}.
\bjournal{Philosophical Transactions of the Royal Society, Series A}.
\bdoi{10.1098/rsta.2022.146}
\end{barticle}
\endbibitem

\bibitem[\protect\citeauthoryear{Grünwald, de~Heide and
  Koolen}{2023}]{grunwald_safe_2023}
\begin{barticle}[author]
\bauthor{\bsnm{Grünwald},~\bfnm{Peter}\binits{P.}}, \bauthor{\bparticle{de}
  \bsnm{Heide},~\bfnm{Rianne}\binits{R.}} \AND
  \bauthor{\bsnm{Koolen},~\bfnm{Wouter}\binits{W.}}
(\byear{2023}).
\btitle{Safe {testing}}.
\bjournal{arXiv:1906.07801 [cs, math, stat]}.
\bnote{First version on arXiv 2019; to appear in Journal of the Royal
  Statistical Society, Series B}.
\end{barticle}
\endbibitem

\bibitem[\protect\citeauthoryear{Hall, Wijs\-man and
  Ghosh}{1965}]{hall_relationship_1965}
\begin{barticle}[author]
\bauthor{\bsnm{Hall},~\bfnm{W.~J.}\binits{W.~J.}},
  \bauthor{\bsnm{Wijs\-man},~\bfnm{R.~A.}\binits{R.~A.}} \AND
  \bauthor{\bsnm{Ghosh},~\bfnm{J.~K.}\binits{J.~K.}}
(\byear{1965}).
\btitle{The {relationship} {between} {sufficiency} and {invariance} with
  {applications} in {sequential} {analysis}}.
\bjournal{The Annals of Mathematical Statistics}
\bvolume{36}
\bpages{575--614}.
\bdoi{10.1214/aoms/1177700169}
\end{barticle}
\endbibitem

\bibitem[\protect\citeauthoryear{Hall, Wijsman and
  Ghosh}{1995}]{hall_correction_1995}
\begin{barticle}[author]
\bauthor{\bsnm{Hall},~\bfnm{W.~J.}\binits{W.~J.}},
  \bauthor{\bsnm{Wijsman},~\bfnm{R.~A.}\binits{R.~A.}} \AND
  \bauthor{\bsnm{Ghosh},~\bfnm{J.~K.}\binits{J.~K.}}
(\byear{1995}).
\btitle{Correction: {The} {relationship} {between} {sufficiency} and
  {invariance} with {applications} in {sequential} {analysis}}.
\bjournal{The Annals of Statistics}
\bvolume{23}
\bpages{705--705}.
\bdoi{10.1214/aos/1176324543}
\end{barticle}
\endbibitem

\bibitem[\protect\citeauthoryear{Henzi et~al.}{2023}]{henzi2023safe}
\begin{barticle}[author]
\bauthor{\bsnm{Henzi},~\bfnm{Alexander}\binits{A.}},
  \bauthor{\bsnm{Puke},~\bfnm{Marius}\binits{M.}},
  \bauthor{\bsnm{Dimitriadis},~\bfnm{Timo}\binits{T.}} \AND
  \bauthor{\bsnm{Ziegel},~\bfnm{Johanna}\binits{J.}}
(\byear{2023}).
\btitle{A safe Hosmer-Lemeshow test}.
\bjournal{The New England Journal of Statistics in Data Science}.
\bnote{(accepted, to appear)}.
\end{barticle}
\endbibitem

\bibitem[\protect\citeauthoryear{Jeffreys}{1961}]{Jeffreys61}
\begin{bbook}[author]
\bauthor{\bsnm{Jeffreys},~\bfnm{H.}\binits{H.}}
(\byear{1961}).
\btitle{Theory of probability},
\bedition{3rd} ed.
\bpublisher{Oxford University Press}, \baddress{London}.
\end{bbook}
\endbibitem

\bibitem[\protect\citeauthoryear{Kariya}{1980}]{Kariya1980a}
\begin{barticle}[author]
\bauthor{\bsnm{Kariya},~\bfnm{Takeaki}\binits{T.}}
(\byear{1980}).
\btitle{{Locally robust tests for serial correlation in least squares
  regression}}.
\bjournal{The Annals of Statistics}
\bvolume{8}
\bpages{1065--1070}.
\bdoi{10.1214/aos/1176345143}
\end{barticle}
\endbibitem

\bibitem[\protect\citeauthoryear{Kullback and
  Leibler}{1951}]{kullback_information_1951}
\begin{barticle}[author]
\bauthor{\bsnm{Kullback},~\bfnm{S.}\binits{S.}} \AND
  \bauthor{\bsnm{Leibler},~\bfnm{R.~A.}\binits{R.~A.}}
(\byear{1951}).
\btitle{On {Information} and {Sufficiency}}.
\bjournal{The Annals of Mathematical Statistics}
\bvolume{22}
\bpages{79--86}.
\bdoi{10.1214/aoms/1177729694}
\end{barticle}
\endbibitem

\bibitem[\protect\citeauthoryear{Lai}{1976}]{lai1976confidence}
\begin{barticle}[author]
\bauthor{\bsnm{Lai},~\bfnm{Tze~Leung}\binits{T.~L.}}
(\byear{1976}).
\btitle{On confidence sequences}.
\bjournal{The Annals of Statistics}
\bvolume{4}
\bpages{265--280}.
\bdoi{10.1214/aos/1176343406}
\end{barticle}
\endbibitem

\bibitem[\protect\citeauthoryear{Lehmann and
  Romano}{2005}]{lehmann_testing_2005}
\begin{bbook}[author]
\bauthor{\bsnm{Lehmann},~\bfnm{Erich~L.}\binits{E.~L.}} \AND
  \bauthor{\bsnm{Romano},~\bfnm{Joseph~P.}\binits{J.~P.}}
(\byear{2005}).
\btitle{Testing {statistical} {hypotheses}},
\bedition{3rd} ed.
\bseries{Springer {Texts} in {Statistics}}.
\bpublisher{Springer-Verlag}, \baddress{New York}.
\bdoi{10.1007/0-387-27605-X}
\end{bbook}
\endbibitem

\bibitem[\protect\citeauthoryear{Liang and Barron}{2004}]{liang2004exact}
\begin{barticle}[author]
\bauthor{\bsnm{Liang},~\bfnm{Feng}\binits{F.}} \AND
  \bauthor{\bsnm{Barron},~\bfnm{Andrew}\binits{A.}}
(\byear{2004}).
\btitle{Exact minimax strategies for predictive density estimation, data
  compression, and model selection}.
\bjournal{IEEE Transactions on Information Theory}
\bvolume{50}
\bpages{2708--2726}.
\bdoi{10.1109/TIT.2004.836922}
\end{barticle}
\endbibitem

\bibitem[\protect\citeauthoryear{Nogales and
  Oyola}{1996}]{nogales_remarks_1996}
\begin{barticle}[author]
\bauthor{\bsnm{Nogales},~\bfnm{A.~G.}\binits{A.~G.}} \AND
  \bauthor{\bsnm{Oyola},~\bfnm{J.~A.}\binits{J.~A.}}
(\byear{1996}).
\btitle{Some {remarks} on {sufficiency}, {invariance} and {conditional}
  {independence}}.
\bjournal{The Annals of Statistics}
\bvolume{24}
\bpages{906--909}.
\bdoi{10.1214/aos/1032894473}
\end{barticle}
\endbibitem

\bibitem[\protect\citeauthoryear{Paterson}{1988}]{paterson_amenability_2000}
\begin{bbook}[author]
\bauthor{\bsnm{Paterson},~\bfnm{Alan~LT}\binits{A.~L.}}
(\byear{1988}).
\btitle{Amenability},
\bedition{1st} ed.
\bseries{Mathematical surveys and monographs}
\bvolume{29}.
\bpublisher{American Mathematical Soc.}
\bdoi{10.1090/surv/029}
\end{bbook}
\endbibitem

\bibitem[\protect\citeauthoryear{Ramdas et~al.}{2020}]{ramdas_admissible_2020}
\begin{barticle}[author]
\bauthor{\bsnm{Ramdas},~\bfnm{Aaditya}\binits{A.}},
  \bauthor{\bsnm{Ruf},~\bfnm{Johannes}\binits{J.}},
  \bauthor{\bsnm{Larsson},~\bfnm{Martin}\binits{M.}} \AND
  \bauthor{\bsnm{Koolen},~\bfnm{Wouter}\binits{W.}}
(\byear{2020}).
\btitle{Admissible anytime-valid sequential inference must rely on nonnegative
  martingales}.
\bjournal{arXiv:2009.03167 [math, stat]}.
\bnote{arXiv: 2009.03167}.
\end{barticle}
\endbibitem

\bibitem[\protect\citeauthoryear{Ramdas et~al.}{2023}]{ramdas2023savi}
\begin{barticle}[author]
\bauthor{\bsnm{Ramdas},~\bfnm{Aaditya}\binits{A.}},
  \bauthor{\bsnm{Gr{\"u}nwald},~\bfnm{Peter}\binits{P.}},
  \bauthor{\bsnm{Vovk},~\bfnm{Volodya}\binits{V.}} \AND
  \bauthor{\bsnm{Shafer},~\bfnm{Glenn}\binits{G.}}
(\byear{2023}).
\btitle{game-theoretic statistics and safe anytime-valid inference}.
\bjournal{Statistical Science}.
\bnote{To appear}.
\end{barticle}
\endbibitem

\bibitem[\protect\citeauthoryear{Reiter and
  Stegeman}{2000}]{reiter_classical_2000}
\begin{bbook}[author]
\bauthor{\bsnm{Reiter},~\bfnm{Hans}\binits{H.}} \AND
  \bauthor{\bsnm{Stegeman},~\bfnm{Jan~D.}\binits{J.~D.}}
(\byear{2000}).
\btitle{Classical {harmonic} {analysis} and {locally} {compact} {groups}},
\bedition{2nd} ed.
\bseries{London {Mathematical} {Society} {Monographs}}.
\bpublisher{Oxford University Press}, \baddress{Oxford, New York}.
\end{bbook}
\endbibitem

\bibitem[\protect\citeauthoryear{Ren and Barber}{2022}]{RenB22}
\begin{barticle}[author]
\bauthor{\bsnm{Ren},~\bfnm{Zhimei}\binits{Z.}} \AND
  \bauthor{\bsnm{Barber},~\bfnm{Rina~Foygel}\binits{R.~F.}}
(\byear{2022}).
\btitle{Derandomized knockoffs: Leveraging e-values for false discovery rate
  control}.
\bjournal{arXiv preprint arXiv:2205.15461}.
\end{barticle}
\endbibitem

\bibitem[\protect\citeauthoryear{Robbins}{1970}]{robbins1970statistical}
\begin{barticle}[author]
\bauthor{\bsnm{Robbins},~\bfnm{Herbert}\binits{H.}}
(\byear{1970}).
\btitle{Statistical methods related to the law of the iterated logarithm}.
\bjournal{The Annals of Mathematical Statistics}
\bvolume{41}
\bpages{1397--1409}.
\end{barticle}
\endbibitem

\bibitem[\protect\citeauthoryear{Rouder et~al.}{2009}]{rouder-2009-bayes}
\begin{barticle}[author]
\bauthor{\bsnm{Rouder},~\bfnm{Jeffrey~N.}\binits{J.~N.}},
  \bauthor{\bsnm{Speckman},~\bfnm{Paul~L.}\binits{P.~L.}},
  \bauthor{\bsnm{Sun},~\bfnm{Dongchu}\binits{D.}},
  \bauthor{\bsnm{Morey},~\bfnm{Richard~D.}\binits{R.~D.}} \AND
  \bauthor{\bsnm{Iverson},~\bfnm{Geoffrey}\binits{G.}}
(\byear{2009}).
\btitle{Bayesian t-tests for accepting and rejecting the null hypothesis}.
\bjournal{Psychonomic Bulletin \& Review}
\bvolume{16}
\bpages{225--237}.
\bdoi{10.3758/PBR.16.2.225}
\end{barticle}
\endbibitem

\bibitem[\protect\citeauthoryear{Roy and Bargmann}{1958}]{roy_tests_1958}
\begin{barticle}[author]
\bauthor{\bsnm{Roy},~\bfnm{S.~N.}\binits{S.~N.}} \AND
  \bauthor{\bsnm{Bargmann},~\bfnm{R.~E.}\binits{R.~E.}}
(\byear{1958}).
\btitle{Tests of {multiple} {independence} and the {associated} {confidence}
  {bounds}}.
\bjournal{The Annals of Mathematical Statistics}
\bvolume{29}
\bpages{491--503}.
\bdoi{10.1214/aoms/1177706624}
\end{barticle}
\endbibitem

\bibitem[\protect\citeauthoryear{Rushton}{1950}]{Rushton50}
\begin{barticle}[author]
\bauthor{\bsnm{Rushton},~\bfnm{S.}\binits{S.}}
(\byear{1950}).
\btitle{On a sequential t-test}.
\bjournal{biometrika}
\bvolume{37}
\bpages{326--333}.
\bdoi{10.2307/2332385}
\end{barticle}
\endbibitem

\bibitem[\protect\citeauthoryear{Shafer}{2021}]{shafer_language_2019}
\begin{barticle}[author]
\bauthor{\bsnm{Shafer},~\bfnm{G.}\binits{G.}}
(\byear{2021}).
\btitle{Testing by betting: A strategy for statistical and scientific
  communication}.
\bjournal{Journal of the Royal Statistical Society, Series A}.
\bdoi{10.1111/rssa.12647}
\end{barticle}
\endbibitem

\bibitem[\protect\citeauthoryear{Shalaevskii}{1971}]{shalaevskii_minimax_1971}
\begin{bincollection}[author]
\bauthor{\bsnm{Shalaevskii},~\bfnm{O.~V.}\binits{O.~V.}}
(\byear{1971}).
\btitle{Minimax {character} of {Hotelling}’s {T2} {test}. {I}}.
In \bbooktitle{Investigations in {Classical} {Problems} of {Probability}
  {Theory} and {Mathematical} {Statistics}: {Part} {I}}
\bedition{1st} ed.
(\beditor{\bfnm{V.~M.}\binits{V.~M.}~\bsnm{Kalinin}} \AND
  \beditor{\bfnm{O.~V.}\binits{O.~V.}~\bsnm{Shalaevskii}}, eds.)
\bpages{74--101}.
\bpublisher{Springer US}, \baddress{Boston, MA}.
\bdoi{10.1007/978-1-4684-8211-9\_2}
\end{bincollection}
\endbibitem

\bibitem[\protect\citeauthoryear{Subbaiah and
  Mudholkar}{1978}]{subbaiah_comparison_1978}
\begin{barticle}[author]
\bauthor{\bsnm{Subbaiah},~\bfnm{Perla}\binits{P.}} \AND
  \bauthor{\bsnm{Mudholkar},~\bfnm{Govind~S.}\binits{G.~S.}}
(\byear{1978}).
\btitle{A {comparison} of {two} {tests} for the {significance} of a {mean}
  {vector}}.
\bjournal{Journal of the American Statistical Association}
\bvolume{73}
\bpages{414--418}.
\bdoi{10.1080/01621459.1978.10481592}
\end{barticle}
\endbibitem

\bibitem[\protect\citeauthoryear{Sun and Berger}{2007}]{sun_objective_2007}
\begin{barticle}[author]
\bauthor{\bsnm{Sun},~\bfnm{Dongchu}\binits{D.}} \AND
  \bauthor{\bsnm{Berger},~\bfnm{James~O.}\binits{J.~O.}}
(\byear{2007}).
\btitle{Objective {Bayesian} analysis for the multivariate normal model}.
\bjournal{Bayesian Statistics}
\bvolume{8}
\bpages{525--562}.
\end{barticle}
\endbibitem

\bibitem[\protect\citeauthoryear{Turner, Ly and
  Grünwald}{2023}]{turner_two-sample_2023}
\begin{barticle}[author]
\bauthor{\bsnm{Turner},~\bfnm{Rosanne}\binits{R.}},
  \bauthor{\bsnm{Ly},~\bfnm{Alexander}\binits{A.}} \AND
  \bauthor{\bsnm{Grünwald},~\bfnm{Peter}\binits{P.}}
(\byear{2023}).
\btitle{Generic E-Variables for Exact Sequential k-Sample Tests that allow for
  Optional Stopping}.
\bjournal{{\em accepted for publication in} Journal of Statistical Planning and
  Inference}.
\bnote{arXiv: 2106.02693}.
\end{barticle}
\endbibitem

\bibitem[\protect\citeauthoryear{Vovk and Wang}{2021}]{vovk_e-values_2021}
\begin{barticle}[author]
\bauthor{\bsnm{Vovk},~\bfnm{Vladimir}\binits{V.}} \AND
  \bauthor{\bsnm{Wang},~\bfnm{Ruodu}\binits{R.}}
(\byear{2021}).
\btitle{E-values: {Calibration}, combination and applications}.
\bjournal{The Annals of Statistics}
\bvolume{49}
\bpages{1736--1754}.
\bdoi{10.1214/20-AOS2020}
\end{barticle}
\endbibitem

\bibitem[\protect\citeauthoryear{Wald}{1945}]{Wald45}
\begin{barticle}[author]
\bauthor{\bsnm{Wald},~\bfnm{Abraham}\binits{A.}}
(\byear{1945}).
\btitle{Sequential tests of statistical hypotheses}.
\bjournal{The Annals of Mathematical Statistics}
\bvolume{16}
\bpages{117-186}.
\bdoi{10.1214/aoms/1177731118}
\end{barticle}
\endbibitem

\bibitem[\protect\citeauthoryear{Wang and Ramdas}{2022}]{WangR22}
\begin{barticle}[author]
\bauthor{\bsnm{Wang},~\bfnm{Ruodu}\binits{R.}} \AND
  \bauthor{\bsnm{Ramdas},~\bfnm{Aaditya}\binits{A.}}
(\byear{2022}).
\btitle{False discovery rate control with e‐values}.
\bjournal{Journal of the Royal Statistical Society, Series B (Methodological)}
\bvolume{84}
\bpages{822-852}.
\bdoi{10.1111/rssb.12489}
\end{barticle}
\endbibitem

\bibitem[\protect\citeauthoryear{Waudby-Smith and
  Ramdas}{2023}]{waudby2020estimating}
\begin{barticle}[author]
\bauthor{\bsnm{Waudby-Smith},~\bfnm{Ian}\binits{I.}} \AND
  \bauthor{\bsnm{Ramdas},~\bfnm{Aaditya}\binits{A.}}
(\byear{2023}).
\btitle{Estimating means of bounded random variables by betting}.
\bjournal{Journal of the Royal Statistical Society: Series B (Statistical
  Methodology)}.
\bnote{(to appear with discussion)}.
\end{barticle}
\endbibitem

\bibitem[\protect\citeauthoryear{Zhang, Ramdas and Wang}{2023}]{zhang2023exact}
\begin{bmisc}[author]
\bauthor{\bsnm{Zhang},~\bfnm{Zhenyuan}\binits{Z.}},
  \bauthor{\bsnm{Ramdas},~\bfnm{Aaditya}\binits{A.}} \AND
  \bauthor{\bsnm{Wang},~\bfnm{Ruodu}\binits{R.}}
(\byear{2023}).
\btitle{When do exact and powerful p-values and e-values exist?}
\bnote{arXiv preprint arXiv:2305.16539}.
\end{bmisc}
\endbibitem

\end{thebibliography}

\newpage

\appendix
\section{Invariance and Sufficiency}\label{app:sufficiency}

The relationship between invariance and sufficiency has been thoroughly
investigated
\citep{hall_relationship_1965,hall_correction_1995,berk_note_1972,nogales_remarks_1996}.
Consider a $G$-invariant hypothesis testing problem such that a sufficient
statistic is available. If the action of $G$ on the original data space induces
a free action on the sufficient statistic, there must be a maximally invariant
function of the sufficient statistic. With this structure in mind, the results
presented thus far suggest two approaches for solving the hypothesis testing
problem. The first is to reduce the data using the sufficient statistic, and to
test the problem using the maximally invariant function of the sufficient
statistic. The second approach is to use the maximally invariant function of the
original data. These two approaches yield two potentially different
growth-optimal $\E$-statistics, and one question arises naturally: are both
approaches equivalent? In this section we show that this is indeed the case,
under certain conditions.

We now introduce the setup formally. At the end of this section we revisit our
guiding example, the t-test, and show how the results of this section apply to
it. Let $\Theta$ be the parameter space, and let $\delta = \delta(\theta)$ be a
maximally invariant function of $\theta$ for the action of $G$ on $\Theta$. Let
$s_n:{\cal X}^n\to {\cal S}_n$ be a sufficient statistic for
$\theta\in\Theta$. % Let
% $\Theta_0, \Theta_1$ be two distinct orbits in $\Theta$ of the action of $G$.
Consider again the hypothesis testing problem in the form presented in
(\ref{eq:target_hypothesis_problem}). Assume further that $G$ acts freely and
continuously on the image space ${\cal S}_n$ of the sufficient statistic
$S_n = s_n(X^n)$, and assume that $s_n$ is compatible with the action of $G$ in
the sense that, for any $X^n\in {\cal X}^n$ and any $g\in G$, the identity
$gs_n(X^n) = s_n(gX^n)$ holds, where $(g,s)\mapsto gs$ makes reference to the
action of $G$ on ${\cal S}_n$. Let $M_{{\cal X},n} = m_{{\cal X},n}(X^n)$ and
$M_{{\cal S},n} = m_{{\cal S},n}(S_n)$ be two maximally invariant functions for
the actions of $G$ on ${\cal X}^n$ and ${\cal S}_n$, respectively. Because of
their invariance, the distributions of $M_{{\cal X},n}$ and $M_{{\cal S},n}$
depend only on the maximally invariant parameter $\delta$. \citet[][Section
II.3]{hall_relationship_1965} proved that, under regularity conditions, if
$S_{{\cal X},n} = s_{{\cal X},n}(X^n)$ is sufficient for $\theta\in \Theta$,
then the statistic $M_{{\cal S},n} = m_{{\cal S},n}(s_n(X^n))$ is sufficient for
$\delta$. In that case, we call $M_{{\cal S},n}$ invariantly sufficient. Here we
state the version of their result, attributed by \citet{hall_relationship_1965}
to C.\ Stein, that suits best our purposes\footnote{The assumption that there
  exists an invariant measure on $G$ implies what \citet{hall_relationship_1965}
  call Assumption A \citep[see][discussion in p. 581]{hall_relationship_1965}}.

\begin{theorem}[C.\ Stein]\label{theo:stein_invariantly_sufficient}
  If there exists a Haar measure on the group $G$, the statistic
  $M_{{\cal S},n} = m_{{\cal S},n}(s_n(X^n))$ is invariantly sufficient, that
  is, it is sufficient for the maximally invariant parameter $\delta$.
\end{theorem}

With this theorem at hand, and the fact that the $\KL$ divergence does not
decrease by the application of sufficient transformations, we obtain the
following proposition.
\begin{proposition}\label{prop:max_inv_suff_kl_equivalence}
  Let $s_n:{\cal X}^n\to {\cal S}_n$ be sufficient statistic for
  $\theta\in\Theta$. Assume that $G$ acts freely on ${\cal S}_n$ and that
  $s_n(gX^n) = gs_n(x^n)$ for all $X^n\in{\cal X}^n$ and $g\in G$. Let
  $m_{{\cal S},n}$ be a maximal invariant for the action of $G$ on ${\cal S}_n$,
  and let $M_{{\cal S},n} = m_{{\cal S},n}(s_n(X^n))$. Then,
  \begin{equation*}
    \KL\paren{\mathbf{P}_{\delta_1}^{M_{{\cal X},n}}, \mathbf{P}_{\delta_0}^{M_{{\cal X},n}}}
    =
    \KL\paren{\mathbf{P}_{\delta_1}^{M_{{\cal S},n}},
      \mathbf{P}_{\delta_0}^{M_{{\cal S},n}}}.
  \end{equation*}
\end{proposition}
\begin{proof}
  The function $M_{{\cal S},n} = m_{{\cal S},n}(s_n(X^n))$ is invariant, and
  consequently its distribution only depends on the maximally invariant parameter
  $\delta$. Since $M_{{\cal X},n}$ is maximally invariant for the action of $G$
  on ${\cal X}^n$, there is a function $f$ such that
  $M_{{\cal S},n} = f(M_{{\cal X},n})$. By Stein's theorem,
  Theorem~\ref{theo:stein_invariantly_sufficient}, $M_{{\cal S},n}$ is
  sufficient for $\delta$. Consequently, $f$ is a sufficient transformation.
  Hence, from the invariance of the $\KL$ divergence under sufficient
  transformations, the result follows.
\end{proof}
Via the factorization theorem of Fisher and Neyman, the likelihood ratio for the
maximal invariant $M_{{\cal X},n}$ coincides with that of the invariantly
sufficient $M_{{\cal S},n}$. As a consequence, we obtain the answer to the
motivating question of this section: performing an invariance reduction on the
original data and on the sufficient statistic are equivalent. 
\begin{corollary}\label{cor:e-statistic-sufficiency}
  Under the assumptions of Proposition~\ref{prop:max_inv_suff_kl_equivalence},
  if $S_n = s_n(X^n)$,
  \begin{equation*}
    \frac{q^{M_{{\cal X},n}}(m_{{\cal X},n}(X^n))}{
      p^{M_{{\cal X},n}}(m_{{\cal X},n}(X^n))}
    = \frac{q^{M_{{\cal S},n}}(m_{{\cal S},n}(S_n))}{
      p^{M_{{\cal S},n}}(m_{{\cal S},n}(S_n))}.
  \end{equation*}
  Hence, if assumptions of Corollary~\ref{cor:main_corollary} also hold, the
  likelihood ratio for the invariantly sufficient statistic $M_{{\cal S},n}$ is
  (relatively) GROW.
\end{corollary}

\begin{example}[continues=ex:t-test]
  We have seen that a maximally invariant function of the data is
  $M_{{\cal X},n} = m_{{\cal X},n}(X^n) = \paren{X_1/\abs{X_1}, \dots, X_n /
    \abs{X_1}}$ while the t-statistic
  $M_{{\cal S},n} = m_{{\cal S},n}(X^n)\propto \hat{\mu}_n / \hat{\sigma}_n$ is
  a maximally invariant function of the sufficient statistic
  $s_n(X^n) = \paren{\hat{\mu}_n,\hat{\sigma}_n}$. Stein's theorem
  (Theorem~\ref{theo:stein_invariantly_sufficient}) shows that the t-statistic
  $M_{{\cal S},n} $ is sufficient for the maximally invariant parameter
  $\delta = \mu/\sigma$. Corollary~\ref{cor:e-statistic-sufficiency} shows that
  the likelihood ratio for the t-statistic is relatively GROW.
\end{example}

\section{Detailed comparison to Sun and Berger (2007) and Liang and Barron (2004): two families vs.\ one}
\label{app:herecomesthesun}
As the  example in Section~\ref{sec:ameneyourself}  illustrates, it is sometimes possible to represent the same
$\mathcal{H}_0$ and $\mathcal{H}_1$ via (at least) two different groups, say
$G_a$ and $G_b$.
Group $G_a$ is combined with parameter of interest in some
space $\Delta_a$ and priors ${\mathbf \Pi}_j^{*\delta_a}$ on $\Delta_a$
achieving (\ref{eq:superpair}) relative to group $G_a$, for $j = 0,1$; group
$G_b$ has parameter of interest in $\Delta_b$ and priors
${\mathbf \Pi}_j^{*\delta_b}$ achieving (\ref{eq:superpair}) relative to group
$G_b$; yet the tuples
${\cal T}_a = (G_a,\Delta_a,\{{\mathbf \Pi}_j^{*\delta_a} \}_{j = 0,1})$ and
${\cal T}_b = (G_b,\Delta_b,\{{\mathbf \Pi}_j^{*\delta_b} \}_{j = 0,1})$ define
the same hypotheses $\mathcal{H}_0$ and $\mathcal{H}_1$. That is, the set of
distributions $\{ {\mathbf P}^*_g \}_{g \in G_a}$ obtained by applying
Proposition~\ref{prop:composite_invariant_grow} with group $G_a$ (representing
$\mathcal{H}_0$ defined relative to group $G_a$) coincides with the set of
distributions $\{ {\mathbf P}^*_g \}_{g \in G_b}$ obtained by applying
Proposition~\ref{prop:composite_invariant_grow} with group $G_b$ (representing
$\mathcal{H}_0$ defined relative to group $G_b$); and analogously for the set of
distributions $\{ {\mathbf P}^*_g \}_{g \in G_a}$ and the set of distributions
$\{ {\mathbf P}^*_g \}_{g \in G_b}$. In the example, $G_a$ was $\GL(d)$
and the priors ${\mathbf \Pi}_0^{*\delta_a}, {\mathbf \Pi}_1^{*\delta_a}$ were
degenerate priors on $0$ and $\gamma$ as in (\ref{eq:hotelling-test-problem}),
respectively; $G_b$ was the lower triangular group with a specific prior as
indicated in the example. In such a case with multiple representations of the same
${\cal H}_0$ and ${\cal H}_1$, using the fact that the notion of "GROW" does not
refer to the underlying group, Corollary~\ref{cor:composite} can be used to
identify the GROW $\E$-statistic as soon as the assumptions of
Proposition~\ref{prop:composite_invariant_grow} hold for at least one of the
tuples ${\cal T}_a$ or ${\cal T}_b$. Namely, if the assumptions hold for just
one of the two tuples, we use Corollary~\ref{cor:composite} with that tuple;
then $T^*$ as defined in the corollary must be GROW, irrespective of
whether $T^*$ based on the other tuple is the same (as it was in the
example above) or different. If the assumptions hold for {\em both\/} groups,
then, using the fact that the GROW $\E$-statistic is essentially unique (see
Theorem~1 of \citetalias{grunwald_safe_2023} for definition and proof), it follows that $T^*(X^n)$ as
defined in Corollary~\ref{cor:composite} must coincide for both tuples.

Superficially, this may seem to contradict \citet{sun_objective_2007} who point
out that in some settings, the right Haar prior is not uniquely defined, and
different choices for right Haar prior give different posteriors. To resolve the paradox, note that, whereas we
always formulate two models $\mathcal{H}_0$ and $\mathcal{H}_1$,
\citet{sun_objective_2007} start with a single probabilistic model, say
$\mathcal{P}$, that can be written as in (\ref{eq:def_invariant_model}) for some
group $G$. Their example shows that the same $\mathcal{P}$ can sometimes arise
from two different groups, and then it is not clear what group, and hence what Haar
prior to pick, and their quantity of interest, the Bayesian posterior, can depend on the choice. 

In contrast, our
quantity of interest, the GROW $\E$-statistic ${T}^*_n$, is uniquely defined as soon as there exists one group $G$ with
${\cal H}_0$ and ${\cal H}_1$ as in (\ref{eq:target_hypothesis_problem}) for
which the assumptions of Theorem~\ref{theo:main_theorem} hold; or more
generally, as soon as there exists one tuple
${\cal T}= (G,\Delta,\{{\mathbf \Pi}_j^{*\delta} \}_{j = 0,1})$ for which the
assumptions of Proposition~\ref{prop:composite_invariant_grow} hold, even if
there exist other such tuples.

To reconcile uniqueness of the GROW $\E$-statistic ${T}^*_n$ with nonuniqueness of the Bayes posterior, note that the former 
is a ratio between Bayes
marginals for different models ${\cal H}_0$ and ${\cal H}_1$ at the same sample
size $n$. In contrast, the Bayes predictive distribution based on a single model $\mathcal{P}$ is
a ratio between Bayes marginals for the same $\mathcal{P}$ at different
sample sizes $n$ and $n-1$. The role of `same' and `different' being interchanged, it turns out that this Bayes predictive distribution {\em can\/}  depend on the group on which the right Haar prior for $\mathcal{P}$ is based. Since the Bayes predictive distribution can be rewritten as a marginal over the Bayes posterior for $\mathcal{P}$, it is then not surprising that this Bayes posterior may also change if the underlying group is changed.

The consideration of two families $\mathcal{H}_0$ and $\mathcal{H}_1$ vs.\ a
single $\mathcal{P}$ is also one of the main differences between our setting and
the one of \citet{liang2004exact}, who provide exact min-max procedures for
predictive density estimation for general location and scale families under
Kullback-Leibler loss. Their results apply to any invariant probabilistic model
${\mathcal P}$ as in (\ref{eq:def_invariant_model}) where the invariance is with
respect to location or scale (and more generally, with respect to some other
groups including the subset of the affine group that we consider in
Section~\ref{sec:regression}). Consider then such a ${\cal P}$ and let
$p^{M_n}(m_n(X^n))$ be as in (\ref{eq:papi_wijsman}). As is well-known, provided
that $n' $ is larger than some minimum value, for all $n > n'$,
$r(X_{n'+1}, \ldots, X_{n} \mid X_1, \ldots, X_{n'}) :=
p^{M_n}(m_n(X^n))/p^{M_{n'}}(m_{n'}(X^{n'}))$ defines a conditional probability
density for $X_{n'+1}, \ldots, X_n$; this is a consequence of the formal-Bayes
posterior corresponding to the right Haar prior becoming proper after $n'$
observations, a.s. under all ${\bf P} \in {\cal P}$. For example, in the t-test
setting, $n'=1$. \citet{liang2004exact} show that the distribution corresponding
to $r$ minimizes the ${\bf P}^{n'}$- expected KL divergence to the conditional
distribution ${\bf P}^{n} \mid X^{n'}$, in the worst case over all
${\bf P} \in {\cal P}$. Even though their optimal density $r$ is defined in
terms of the same quantities as our optimal statistic $T^*_n$, it is, just as
\citet{berger_objective_2008}, considered above, a ratio between likelihoods for
the same model at different sample sizes, rather than, as in our setting,
between likelihoods for different models, both composite, at the same sample
sizes. Our setting requires a joint KL minimization over two families, and
therefore our proof techniques turn out quite different from their information-
and decision-theoretic ones.

\section{Anytime-valid testing under Optional Stopping and Optinal Continuation}
\label{app:av}
%We now assume that the observations are made sequentially. At the end of the section we describe the consequences to our main example, the t-test. We begin by defining our working model for this scenario. 
Consider the setting of Section~\ref{sec:intro-group-invariance}. Let $X = (X_n)_{n\in\nats}$ be a random
process, where each $X_n$ is an observation that takes values on a space
${\cal X}$. Let $(M_n)_{n\in \nats}$ be a sequence where, for each $n$,
$M_n = m_n(X^n)$ is a maximally invariant function for the action of $G$ on
${\cal X}^n$. 

Suppose that data
$X_1, X_2, \dots$ are gathered one by one. Here, a sequential test is a sequence
of zero-one-valued statistics $\xi = (\xi_n)_{n\in\nats}$ adapted to the natural
filtration generated by $X_1,X_2,\dots$. We consider the test defined by
$\xi_n = \indicator{T^{M_n} \geq 1 / \alpha}$ for some value $\alpha$. We note that Wald-style---Sequential Probability Ratio Tests---tests are different because they would output "no decision" until a particular sample size $n$. Afterwards, they would output $1$ ("reject the null") or $0$ ("there is no evidence to reject the null") forever. In contrast, in the present setting $\xi_n = 1$ means "if you stop now, for whatever reason, it is safe to reject the null".  Below  we prove the 
anytime validity of $\xi$. Additionally, we show that, for certain stopping
times $\tau\leq \infty$, the optionally stopped $\E$-statistic $T^{M_\tau}$
remains an $\E$-statistic. This fact validates the use of the stopped $T^{M_\tau}$ for
optional continuation because we can multiply the \E-statistics $T^{M_\tau}$ across studies while retaining type-I error control. This result is not new and we add it merely for completeness; it follows by standard arguments as \citet{ramdas2023savi} or \citetalias{grunwald_safe_2023}.  
\begin{proposition}\label{prop:main_sequential_result}
  Let $T^* = (T^{M_n})_{n\in\nats}$, where, for each $n$, $T^{M_n}$ is the
  likelihood ratio for the maximally invariant function $M_n = m_n(X^n)$ for the
  action of $G$ on ${\cal X}^n$. Let $\xi = (\xi_n)_{n\in\nats}$ be the
  sequential test given by $\xi_n = \indicator{T^{M_n}\geq 1/ \alpha}$. Then, the
  following two properties hold:
  \begin{enumerate}
  \item The sequential test $\xi$ is anytime valid at level $\alpha$, that is,
    \begin{equation*}
      \text{for any random time $N$, }
      \sup_{\theta_0\in\Theta_0}
      \mathbf{P}_{\theta_0}\bracks{\xi_N = 1} \leq
      \alpha.
    \end{equation*}\label{item:anytime_validity}
  \item Suppose that $\tau\leq \infty$ is a stopping time with respect to the filtration induced by $M = (M_n)_{n\in\nats}$. Then the optionally stopped $\E$-statistic
    $T^{M_\tau}$ is also an $\E$-statistic, that is,
    \begin{equation}\label{eq:optstopone}
      \sup_{\theta_0\in\Theta_0}\mathbf{E}_{\theta_0}^{\mathbf{P}}[T^{M_\tau}]\leq 1.
    \end{equation}\label{item:stopped_e_stat}
  \end{enumerate}
\end{proposition}
\commentout{The mechanism of the proof of this proposition---showing that
$T^* = (T^*_n)_{n\in\nats}$ is a nonnegative martingale with expected value
one---is, by now, standard; we perform it in Section \ref{sec:anyt-valid-test}.
The main ingredient, where invariance plays a role, pertains how the maximally
invariant statistic at step $n$ contains all information about the invariant
component of the data at previous steps. An inequality of
\citet{ville_etude_1939} and standard optional stopping theorems give the
desired results.}

It is natural to ask whether (\ref{eq:optstopone}) also holds
for stopping times that are adapted to the full data $(X^n)_{n\in\nats}$ but not to the reduced $(M_n)_{n\in\nats}$. In our t-test example, this
could be a stopping time $\tau^*$ such as ``$\tau^*:=1$ if
$|X_1| \not \in [a,b]$; $\tau^*=2$ otherwise'' for some $0 < a < b$. The answer
is negative: after proving Proposition~\ref{prop:main_sequential_result}, we show that, for
appropriate choice of $a$ and $b$, this $\tau^*$ is a counterexample. This
means that such nonadapted $\tau^*$ cannot be safely used under optional
continuation. However, using such a stopping time has no repercussions for optional stopping, since the time $N$ in part 1 of the proposition above is not even required to be a stopping time---$N$ is not restricted by the filtration induced by $M$ and it is even allowed to depend on future observations.

\commentout{
If, at each sample size $n$, the assumptions of
Corollary~\ref{cor:main_corollary} hold, we have shown that
\begin{equation}\label{eq:grow_sequential_estat}
  T^*_n(X^n) = \frac{q^{M_n}(m_n(X^n))}{p^{M_n}(m_n(X^n))},
\end{equation}
the likelihood ratio for the maximal invariant $M_n = m_n(X^n)$, defines a
sequence $T^* = (T^*_n)_{n\in\nats}$ of relatively GROW $\E$-statistics for
(\ref{eq:main_problem_group_parametrization}).

We will now use 
that $T^* = (T^*_n)_{n\in\nats}$ is a martingale as shown in Proposition~\ref{prop:grow_also_av}, 
to prove
Proposition~\ref{prop:main_sequential_result} from
Section~\ref{sec:main-results-intro}, the main result in this work pertaining
sequential testing. We end this section with the implications to the t-test.
}
\begin{proof}[Proof of Proposition~\ref{prop:main_sequential_result}]
  From Proposition~\ref{prop:grow_also_av}, we know that
  $T^* = (T^{M_n})_{n\in\nats}$ is a nonnegative martingale with expected value equal
  to one. Let $\xi = (\xi_n)_n$ be the sequential test given by
  $\xi_n = \indicator{T^{M_n} \geq 1 / \alpha}$. The anytime-validity at level
  $\alpha$ of $\xi$, is a consequence of Ville's inequality, and the fact that
  the distribution of each $T^{M_n}$ does not depend on $g$. Indeed, these two,
  together, imply that
  \begin{equation*}
    \sup_{g\in G}\mathbf{P}_g\{T^{M_n} \geq 1 / \alpha \text{ for some } n\in\nats
    \} \leq \alpha.
  \end{equation*}
  This implies the first statement. Now, let $\tau\leq \infty$ be a stopping
  time with respect to the filtration induced by $M$. 
  If the stopping time $\tau$ is almost surely
  bounded, $T^{M_\tau}$ is an $\E$-statistic by virtue of the optional stopping
  theorem. However, since $T^*$ is a nonnegative martingale, Doob's martingale
  convergence theorem implies the existence of an almost sure limit
  $T^*_\infty$. Even when $\tau$ might be infinite with positive probability,
  Theorem~4.8.4 of \citet{durrett_probability_2019} implies that $T^{M_\tau}$ is
  still an $\E$-statistic.
\end{proof}

\subsection{Importance of the filtration for randomly stopped E-Statistics}
\label{app:filtration_counterexample}
Consider the  t-test as in Example~\ref{ex:t-test}. Fix some $0 < a < b$, and
define the stopping time $\tau^* :=1$ if $|X_1| \not \in [a,b]$. $\tau^*=2$ otherwise.
Then $\tau^*$ is not adapted to (hence not a stopping time relative to)
$(M_n)_n$ as defined in that example, since $M_1 \in \{-1,1\}$ coarsens out all
information in $X_1$ except its sign. Now let $\delta_0 := 0$ (so that
${\cal H}_0$ represents the normal distributions with mean $\mu = 0$ and
arbitrary variance). Let $T^{*,\delta_1}_n(X^n)$ be equal to the GROW
$\E$-statistic $T^{M_n}(X^n)$ as in (\ref{eq:ttesthaar}); here we make explicit
its dependence on $\delta_1$. For ${\cal H}_1$, to simplify computations, we put
a prior $\tilde{\mathbf \Pi}^{\delta}_1$ on $\Delta_1 := {\mathbb R}$. We take
$\tilde{\mathbf \Pi}^{\delta}_1$ to be a normal distribution with mean $0$ and
variance $\kappa$. We can now apply Corollary~\ref{cor:delta_fixed_prior} (with
prior $\tilde{\mathbf \Pi}^{\delta}_0$ putting mass $1$ on
$\delta = \delta_0 = 0$), which gives that $\tilde{T}_n=\tilde{t}_n(X^n)$ is an
$\E$-statistic, where
$$
\tilde{t}_n(x^n)
= \int \frac{1}{\sqrt{2 \pi \kappa^2}} \exp\left(- \frac{\delta_1^2}{2\kappa^2} \right) \cdot T^{*,\delta_1}_n(x^n) \rmd \delta_1
$$
coincides with a standard type of Bayes factor used in Bayesian statistics.
By exchanging the integrals in the numerator, this expression can be calculated analytically.
The Bayes factor $\tilde{T}_1$ for $x^1 = x_1$ is found to be equal to $1$ for all $x_1 \neq 0$, and the Bayes factor for $(x_1,x_2)$ is given by:
$$
\tilde{T}_2= \frac{\sqrt{2 \kappa^2 +1} \cdot(x_1^2+x_2^2) }{\kappa^2(x_1-x_2)^2 + (x_1^2 + x_2^2)}.
$$
Now we consider the function
$$
f(x) := {\bf E}_{X_2 \sim N(0,1)} [\tilde{t}_2(x,X_2)].
$$
$f(x)$ is continuous and even.
We want to show that, with $\tau^*$ as above, $\tilde{T}_{\tau^*}$ is not an E-variable for some specific choices of $a,b$ and $\kappa$. Since, for any $\sigma > 0$, the null contains the  distribution under which the $X_i$ are i.i.d.\ $N(0,\sigma)$, the data may, under the null, in particular be sampled from $N(0,1)$. It thus  suffices to show that
$$
{\bf E}_{X_1,X_2 \sim N(0,1)}[\tilde{T}_{\tau^*}] = \mathbf{P}_{X_1 \sim
  N(0,1)}\{|X_1| \not \in [a,b]\} + {\bf E}_{X_1 \sim N(0,1)}[ {\bf 1}_{|X_1|
  \in [a,b]} f(X_1)] > 1.
$$
From numerical integration we find that $f(x)> 1$ on $[a,b]$ and $[-b,-a]$ if we take  $\kappa=200$, $a \approx 0.44$ and $b \approx 1.70$. The above expectation is then approximately equal to $1.19$, which shows that, even though $\tilde{T}_n$ is an $\E$-statistic at each $n$ by Corollary~\ref{cor:delta_fixed_prior} (it is even a GROW one), $\tilde{T}_{\tau^*}$ is not an $\E$-statistic (its expectation is $0.19$ too large), providing the claimed counterexample.

\section{Further Derivations, Computations and Proofs}
In this appendix, we prove the technical lemmas whose proof was omitted from the main text. In Section~\ref{app:proof_tech_lemmas}, we prove the lemmas used in the proof of Theorem~\ref{theo:main_theorem}. In Section~\ref{sec:deriv-likel-ratios}, we show the computations omitted from Section~\ref{sec:lower_triang}. 

\label{app:furtherproofs}
\subsection{Proof of technical lemmas \ref{lem:prior_sequence_existence},
  \ref{lem:first_term_bound}, and \ref{lem:second_term_bound} for Theorem~\ref{theo:main_theorem}}\label{app:proof_tech_lemmas}
\begin{proof}[Proof of Lemma~\ref{lem:prior_sequence_existence}]
  Let $\{\varepsilon_i\}_i$ be a sequence of positive numbers decreasing to
  zero. Let $\{K_i\}_{i\in\nats}$ and $\{L_i\}_{i\in\nats}$ be two arbitrary
  sequences of compact symmetric subsets that increase to cover $G$. Fix
  $i\in\nats$. The set $K_iL_i$ is compact and by our assumption there exists a
  sequence $\{J_l\}_{l\in\nats}$ and such that
  $\rho\{J_l\}/\rho\{J_lK_iL_i\}\to 1$ as $l\to\infty$. Pick $l(i)$ to be such
  that $\rho\{ J_{l(i)}\}/\rho\{J_{l(i)}K_iL_i\}\geq 1- \varepsilon_i$. The
  claim follows from a relabeling of the sequences.
\end{proof}

\begin{proof}[Proof of Lemma~\ref{lem:first_term_bound}]
  Let $h\in N$. Then we can write
  % \begin{align*}
  %   \int_{NL}q_{g}(h|m)\rmd \rho(g)
  %   &=
  %     \int_{NL}q(g^{-1}h|m)\rmd \rho(g)
  %   \\
  %   &=
  %     \Delta(h)\int_{h^{-1}NL}q(g^{-1}|m)\rmd \rho(g)\\
  %   &=
  %     \Delta(h)\int_{(NL)^{-1}h}q(g|m)\rmd \lambda(g).
        %     \end{align*}
    \begin{align*}
   &    \int \indicator{g \in NL}  \ q_{g}(h|m) \rmd\rho(g)
      =
        \int
        \indicator{g\in NL}  \ q_1(g^{-1}h|m)
        \rmd\rho(g)
      \\
      = & 
        \int \indicator{g \in (NL)^{-1}}  \ q_1(gh|m) \rmd\lambda(g)
      =
        \Delta(h^{-1})\int  \indicator{g\in (NL)^{-1}h}  \ q_1(g|m) \rmd\lambda(g)\\
      = &  \Delta(h^{-1})\mathbf{Q}_1\{H\in (NL)^{-1}h \ | \ M = m\}
    \end{align*}
    The same computation can be carried out for $p$. Consequently
  \begin{align*}
    \ln
    \frac{
    \int \indicator{g\in NL}  \ q_{g}(h|m) \rmd\rho(g)
    }{
    \int \indicator{g\in NL}  \ p_{g}(h|m) \rmd\rho(g)
    }
    & =
      \ln
      \frac{
      \mathbf{Q}_1\{H\in (NL)^{-1}h \ | \ M = m\}
      }{
      \mathbf{P}_1\{H\in (NL)^{-1}h \ | \ M = m\}
      }\\
    & \leq -\ln\mathbf{P}_1\{ H\in (NL)^{-1}h \ | \ M = m \}.
  \end{align*}
  By our assumption that $h\in N$, we have that
  $(NL)^{-1}h = L^{-1}N^{-1}h \supseteq L^{-1} = L$. This implies that the last
  quantity of the previous display is smaller than
  $-\ln\mathbf{P}_1\{H\in L \ | \ M = m \}$. The result follows.
\end{proof}

\begin{proof}[Proof of Lemma~\ref{lem:second_term_bound}]
  The result follows from a rewriting and an application of Jensen's inequality.
  Indeed,
  \begin{align*}
    -\ln\frac{\int p_g(h|m)\rmd \mathbf{\Pi}(g)}{\int q_g(h|m)\rmd
    \mathbf{\Pi}(g)}
    &=
      -\ln\frac{\int q_{g}(h|m)\frac{p_{g}(h|m)}{q_{g}(h|m)}\rmd \mathbf{\Pi}(g)}{\int q_{g}(h|m)\rmd
      \mathbf{\Pi}(g)}
    =
      -\ln\int \frac{p_{g}(h|m)}{q_{g}(h|m)} \rmd \mathbf{\Pi}(g | h,m)\\
    &\leq
      -\int \ln\frac{p_{g}(h|m)}{q_{g}(h|m)} \rmd \mathbf{\Pi}(g | h,m) = 
      \int \ln\frac{q_{g}(h|m)}{p_{g}(h|m)} \rmd \mathbf{\Pi}(g | h,m),
  \end{align*}
  as it was to be shown.
\end{proof}
\subsection{Derivation and Computation for Section~\ref{sec:lower_triang}}
\label{sec:deriv-likel-ratios}
We now provide Proposition~\ref{prop:triangular_maxinv_distr}, giving the derivation underlying Lemma~\ref{lem:submit} in the main text about the likelihood ratio
$T^*_{{\cal S},n}$ for $\delta_0=0$, followed by details about numerical computation.
\begin{proposition}\label{prop:triangular_maxinv_distr}
  Let $X\sim N(\gamma, I)$, and let $m S \sim W(m, I)$ be independent random
  variables. Let $LL' = S$ be the Cholesky decomposition of $S$, and let
  $M = \frac{1}{\sqrt{m}}L^{-1}X$. If $\mathbf{P}_{0,n}$ is the probability distribution
  under which $X\sim N(0, I)$, then, the likelihood
  $ p_{\gamma, m}^M / p_{0, m}^M$ ratio is given by
  \begin{equation*}
    \frac{p_{\gamma,m}^M(M)}{p_{0,m}^M(M)}
    =
    \rme^{-\frac{1}{2}\norm{\gamma}^2}
    \int
    \rme^{\inner{\gamma}{T A^{-1}M}}
    \rmd\mathbf{P}_{m+1, I}(T)
  \end{equation*}
  where $A\in{\cal L}^+$ is the Cholesky factor
  $AA' = I + MM'$, and $\mathbf{P}_{m+1, I}^T$ is the probability
  distribution on ${\cal L}^+$ such that $TT'\sim W(m + 1, I)$.
\end{proposition}
\begin{proof}
  Let $\Sigma = \Lambda\Lambda'$ be the Cholesky decomposition of $\Sigma$. The
  density $p_{\gamma,\Lambda}^{X}$ of $X$ with respect to the Lebesgue measure on
  $\reals^d$ is
  \begin{equation*}
    p_{\gamma, \Lambda}^{X}(X)
    =
    \frac{1}{(2\pi)^{d/2}\det(\Lambda)}
    \etr\paren{
      - \frac{1}{2}(\Lambda^{-1}X - \gamma)(\Lambda^{-1}X-\gamma)'
    },
  \end{equation*}
  where, for a square matrix $A$, we define $\etr(A)$ to be the exponential of
  the trace of $A$. Let $W = mS$. Then, the density $p^{W}_{\gamma, \Lambda}$ of
  $W$ with respect to the Lebesgue measure on $\reals^{d(d-1)/2}$ is
  \begin{equation*}
    p^{W}_{\gamma, \Lambda}(W)
    =
    \frac{1}{2^{md/2}\Gamma_d(n/2)\det(\Lambda)^{m}}
    \det(S)^{(m - d - 1)/2}\etr\paren{-\frac{1}{2}(\Lambda\Lambda')^{-1}W}.
  \end{equation*}
  Now, let $W = TT'$ be the Cholesky decomposition of $W$. We seek to compute
  the distribution of the random lower lower triangular matrix $T$. To this end,
  the change of variables $W\mapsto T$ is one-to-one, and has Jacobian
  determinant equal to $2^d\prod_{i = 1}^dt_{ii}^{d-i+1}$. Consequently, the
  density $p^T_{\gamma, \Lambda}(T)$ of $T$ with respect to the Lebesgue measure is
  \begin{equation*}
    p^T_{\gamma, \Lambda}(T)
    =
    \frac{2^d}{2^{md/2}\Gamma_d(m/2)}
    \det(\Lambda^{-1}T)^m
    \etr\paren{-\frac{1}{2}(\Lambda^{-1}T)(\Lambda^{-1}T)'}
    \prod_{i = 1}^dt_{ii}^{-i}.
  \end{equation*}
  We recognize $\rmd\nu(T) = \prod_{i = 1}^dt_{ii}^{-i}\rmd T$ to be a left Haar
  measure on ${\cal L}_+$, and consequently
  \begin{equation*}
    \tilde{p}^T_{\gamma, \Lambda}(T)
    =
    \frac{2^d}{2^{md/2}\Gamma_d(m/2)}
    \det(\Lambda^{-1}T)^m
    \etr\paren{-\frac{1}{2}(\Lambda^{-1}T)(\Lambda^{-1}T)'}
  \end{equation*}
  is the density of $T$ with respect to $\rmd \nu(T)$. After these rewritings,
  The density $\tilde{p}_{\gamma, \Lambda}^{X,T}(X,T)$ of the pair $(X,T)$ with
  respect to $\rmd X\times \rmd \nu(T)$ is given by
  \begin{equation*}
    \tilde{p}_{\gamma, \Lambda}^{X,T}(X,T)
    =
    \frac{2^d}{K}
    \frac{\det(\Lambda^{-1}T)^{m}}{\det(\Lambda)}
    \etr\paren{-\frac{1}{2}(\Lambda^{-1}T)(\Lambda^{-1}T)'
      - \frac{1}{2}(\Lambda^{-1}X - \gamma)(\Lambda^{-1}X-\gamma)'}
  \end{equation*}
  with $K = (2\pi)^{d/2}2^{md/2}\Gamma_d(n/2)$. The change of variables
  $(X,T)\mapsto (T^{-1}X, T)$ has Jacobian determinant equal to $\det(T)$. If
  $M = T^{-1}X$, then, the density $\tilde{p}^{M,T}_{\gamma,\Lambda}$ of
  $(M, T)$ with respect to $\rmd M\times \rmd\nu(T)$ is given by
  \begin{equation*}
    \tilde{p}^{M,T}_{\gamma,\Lambda}(M,T)
    =
    \frac{\det(\Lambda^{-1}T)^{m + 1}}{K''}
    \etr\paren{-\frac{1}{2}(\Lambda^{-1}T)(\Lambda^{-1}T)'
      - \frac{1}{2}(\Lambda^{-1}TM - \gamma)(\Lambda^{-1}TM-\gamma)'}.
  \end{equation*}
  We now marginalize $T$ to obtain the distribution of the maximal invariant
  $M$. Since the integral is with respect to the left Haar measure $\rmd
  \nu(T)$, we have that
  \begin{equation*}
    \int_{T\in{\cal L}^+}\tilde{p}^{M,T}_{\gamma,\Lambda}(M,T) \rmd\nu(T)
    =
    \int_{T\in{\cal L}^+}\tilde{p}^{M,T}_{\gamma,I}(M,\Lambda^{-1}T) \rmd\nu(T)
    =
    \int_{T\in{\cal L}^+}\tilde{p}^{M,T}_{\gamma,I}(M,T) \rmd\nu(T),
  \end{equation*}
  and consequently,
  \begin{align*}
    p_{\gamma,\Lambda}^M(M)
    &=
    \frac{2^d}{K}
    \int_{T\in{\cal L}^+}
    \det(T)^{m+1}
    \etr\paren{-\frac{1}{2}TT'
      - \frac{1}{2}(TM - \gamma)(TM - \gamma)'}
      \rmd\nu(T)\\
    &=
      \frac{2^d}{K}\rme^{-\frac{1}{2}\norm{\gamma}^2}
      \int_{T\in{\cal L}^+}
      \det(T)^{m+1}
      \etr\paren{-\frac{1}{2}T(I + MM')T'
      + \gamma(TM)'}
      \rmd\nu(T).
  \end{align*}
  The matrix $I + MM'$ is positive definite and symmetric. It is then possible to
  perform its Cholesky decomposition  $(I + MM') = AA'$. With this at hand, the
  previous display can be written as
  \begin{equation*}
    p_{\gamma,\Lambda}^M(M)
    =
    \frac{\rme^{-\frac{1}{2}\norm{\gamma}^2}}{K}
    \int_{T\in{\cal L}^+}
    \det(T)^{m+1}
    \etr\paren{-\frac{1}{2}(TA)(TA)'
      + \gamma(TM)'}
    \rmd\nu(T).
  \end{equation*}
  We now perform the change of variable $T\mapsto TA^{-1}$. To this end, notice that
  $\rmd\nu(A^{-1}) = \rmd\nu(T)\prod_{i=1}^da_{ii}^{-(d-2i+1)}$, and consequently
  \begin{align*}
    p_{\gamma,\Lambda}^M(M)
    &=
    \frac{2^d}{K}
    \frac{\rme^{-\frac{1}{2}\norm{\gamma}^2}\prod_{i=1}^da_{ii}^{2i}}{\det(A)^{m
      +d + 2}}
    \int_{T\in{\cal L}^+}
    \det(T)^{m+1}
    \etr\paren{-\frac{1}{2}TT'
      + \gamma(TA^{-1}M)'}
      \rmd\nu(T)\\
    &=
      \frac{\Gamma_d\paren{\frac{m+1}{2}}}{\pi^{d/2}\Gamma_d\paren{\frac{m}{2}}}
      \frac{\prod_{i=1}^da_{ii}^{2i}}{\det(A)^{m
      +d + 2}}
      \rme^{-\frac{1}{2}\norm{\gamma}^2}
      \mathbf{P}^T_{m+1}\sqbrack{\rme^{\inner{\gamma}{TA^{-1}M}}},
  \end{align*}
  so that that at $\gamma = 0$ the density $p_{0,\Lambda}^M(M)$ takes the form
  \begin{equation*}
    p_{0,\Lambda}^M(M)
    =
    \frac{\Gamma_d\paren{\frac{m+1}{2}}}{\pi^{d/2}\Gamma_d\paren{\frac{m}{2}}}
    \frac{\prod_{i=1}^da_{ii}^{2i}}{\det(A)^{m
        +d + 2}},
  \end{equation*}
  and consequently the likelihood ratio is
  \begin{align*}
    \frac{p_{\gamma,\Lambda}^M(M)}{p_{0,\Lambda}^M(M)}
    =
    \rme^{-\frac{1}{2}\norm{\gamma}^2}
    \int \rme^{\inner{\gamma}{TA^{-1}M}} \rmd\mathbf{P}_{m+1}(T).
  \end{align*}
\end{proof}

\begin{remark}[Numerical computation]
  Computing the optimal $\E$-statistic is feasible numerically. We are interested in
  computing
  \begin{equation*}
    \int \rme^{\inner{x}{T y}} \rmd\mathbf{P}_{m + 1}(T),
  \end{equation*}
  where $T$ is a ${\cal L}^+$-valued random lower triangular matrix such that
  $TT'\sim W(m + 1, I)$, and $x,y \in \reals^d$. Define, for $i\geq j$, the numbers
  $a_{ij} = x_iy_j$. Then $\inner{x}{Ty} = \sum_{i\geq j}a_{ij}T_{ij}$. By
  Bartlett's decomposition, the entries of the matrix $T$ are independent and
  $T_{ii}^2\sim \chi^2((m + 1) - i + 1)$, and $T_{ij}\sim N(0,1)$ for $i>j$. Hence,
  our target quantity satisfies
  \begin{equation*}
    \int [\rme^{\inner{x}{Ty}}] \mathbf{P}_{m + 1}(T)
    =
    \int \rme^{\sum_{i\geq j}a_{ij}T_{ij}} \rmd\mathbf{P}_{m + 1}(T)
    =
    \int  \prod_{i\geq j}\rme^{a_{ij}T_{ij}} \rmd\mathbf{P}_{m + 1}(T).
  \end{equation*}
  On the one hand, for the off-diagonal elements satisfy, using the expression
  for the moment generating function of a standard normal random variable,
  \begin{equation*}
    \mathbf{E}^{\mathbf{P}}_{m + 1}[\rme^{a_{ij} T_{ij}}]
    =\exp\paren{\frac{1}{2}a_{ij}^2}.
  \end{equation*}
  For the diagonal elements the situation is not as simple, but a numerical
  solution is possible. Indeed, for $a_{ii} \geq 0$, and $k_i = (m + 1) - i + 1$
  \begin{align*}
    \mathbf{E}^{\mathbf{P}}_{m}[\rme^{a_{ii} T_{ii}}]
    &=\frac{1}{2^{\frac{k_i}{2}}\Gamma\paren{\frac{k_i}{2}}}
    \int_0^\infty x^{\frac{k_i}{2} - 1}\exp\paren{-\frac{1}{2}x +
      a_{ii}\sqrt{x}}\rmd x\\
    &=
      {}_1F_1\paren{\frac{k_i}{2}, \frac{1}{2}, \frac{a_{ii}^2}{2}}
      +
      \frac{
      \sqrt{2}a_{ii}\Gamma\paren{\frac{k_i + 1}{2}}
      }{
      \Gamma\paren{\frac{k_i}{2}}
      }
      {}_1F_1\paren{\frac{k_i + 1}{2}, \frac{3}{2}, \frac{a_{ii}^2}{2}},
  \end{align*}
  where ${}_1F_1(a,b, z)$ is the Kummer confluent hypergeometric function. For
  $a_{ii}<0$,
  \begin{equation*}
   \frac{1}{2^{k_i/2}\Gamma\paren{\frac{k_i}{2}}} \int_0^\infty
    x^{k_i/2 - 1}\exp\paren{-\frac{1}{2}x + a_{ii}\sqrt{x}}\rmd x
    =
    \frac{\Gamma\paren{k_i}}{2^{k_i - 1} \Gamma\paren{\frac{k_i}{2}}}
    U\paren{\frac{k_i}{2}, \frac{1}{2}, \frac{a_{ii}^2}{2}},
  \end{equation*}
  and $U$ is Kummer's U function.
\end{remark}
\end{document}